\documentclass[reqno,11pt]{amsart}
\usepackage{amsmath,amsfonts,amssymb,amsxtra,latexsym,amscd,enumerate,amsthm,verbatim}

\usepackage[margin=1.40in]{geometry}
\numberwithin{equation}{section}

\newcommand{\R}{\mathbb{R}}

\newcommand{\E}{\mathcal{E}}

\newcommand{\G}{\mathcal{G}}
\newcommand{\T}{\mathbb{T}}
\newcommand{\C}{\mathbb{C}}
\newcommand{\Z}{\mathbb{Z}}
\newcommand{\eps}{\epsilon}

\newcommand{\va}{\vartheta}

\numberwithin{equation}{section} %pour numeroter les equations par section

\newtheorem{theorem}{Theorem}[section]
\newtheorem{lemma}[theorem]{Lemma}
\newtheorem{proposition}[theorem]{Proposition}

\newtheorem{remark}[theorem]{Remark}

\begin{document}

\title{Inviscid damping near the Couette flow in a channel}

\author{Alexandru D. Ionescu}
\address{Princeton University}
\email{aionescu@math.princeton.edu}

\author{Hao Jia}
\address{University of Minnesota}
\email{jia@umn.edu}

\thanks{The first author was supported in part by NSF grant DMS-1600028 and by NSF-FRG grant DMS-1463753.  The second author was supported in part by DMS-1600779 and he is also
grateful to IAS where part of the work was carried out}

\begin{abstract}
%{\small}
We prove asymptotic stability of the Couette flow for the 2D Euler equations in the domain $\T\times[0,1]$. More precisely we prove that if we start with a small and smooth perturbation (in a suitable Gevrey space) of the Couette flow, then the velocity field converges strongly to a nearby shear flow. 

Our solutions are defined on the compact set $\T\times[0,1]$ (``the channel") and therefore have finite energy. The vorticity perturbation, which is initially assumed to be supported in the interior of the channel, will remain supported in the interior of the channel at all times, will be driven to higher  frequencies by the linear flow, and will converge weakly to another shear flow as $t\to\infty$.
%%%
%%% EULER-MAXWELL
%%% Quasi-linear NLS
%%%

\end{abstract}

\maketitle

\setcounter{tocdepth}{1}

\tableofcontents

\section{Introduction}

In this paper we consider the two dimensional Euler equation
\begin{equation}\label{euler}
\partial_tu+u\cdot\nabla u+\nabla p=0,\qquad {\rm div}\,u=0,
\end{equation}
in $(x,y,t)\in \mathbb{T}\times [0,1]\times[0,T)$, with the boundary condition $u^y|_{y=0,\,1}\equiv 0$. Letting
$$\omega:=\nabla^{\perp}u=-\partial_yu^x+\partial_xu^y$$
as the vorticity, the equation (\ref{euler}) can be written in vorticity form as
\begin{equation}\label{Euler}
  \partial_t\omega+u\cdot\nabla \omega=0,
 \end{equation}
for $(x,y,t)\in \mathbb{T}\times [0,1]\times[0,\infty)$. Here
\begin{equation}\label{uvort}
 u=\nabla^{\perp}\psi=(-\partial_y\psi, \partial_x\psi),
\end{equation}
and the stream function $\psi$ is given by 
\begin{equation}\label{stream}
 \Delta \psi=\omega,\quad{\rm on}\,\,\mathbb{T}\times[0,1],\qquad \psi(x,0)\equiv 0,\,\psi(x,1)\equiv C_0.
\end{equation}
The constant $C_0$ is given by 
$$C_0:=-\frac{1}{2\pi}\int_{\mathbb{T}\times[0,1]}u^x(t,x,y)\,dxdy,$$
 which is constant over time. 
 
The two dimensional Euler equation \eqref{Euler}--\eqref{stream} is globally well-posed for smooth initial data and the solution remains smooth for all times, by the Beale-Kato-Majda criteria. See also Yudovich \cite{Yudovich1, Yudovich2} for wellposedness with initial data having bounded vorticity only.

An important open question for the two dimensional Euler equation is the long time behavior of the vorticity $\omega$. There are a variety of nontrivial behavior. For instance,  there are many steady states such as shear flows  and periodic in time
solutions such as the Kirchhoff elliptic vortices. It is quite possible that quasi-periodic in time solutions may arise from perturbations of periodic in time solutions. Recently solutions whose vorticity gradients grow double exponentially over time have also been constructed in \cite{KiselevSverak}. 

Stability analysis of special solutions to the two dimensional Euler equations is a classical topic in hydrodynamics, studied by prominent figures such as Rayleigh and Kelvin with focus on the linearized spectral stability. The important work of Arnold, see \cite{Arnold}, provides a rare criteria for nonlinear stability. General statements on the dynamics in the non-perturbative regime are harder to obtain. We mention here a qualitative result due to \v{S}ver\'ak (see section 35 in \cite{SverakNotes}) on the general existence of solutions $\omega(t)$ with {\it pre-compact} trajectory. 

Numerical simulations seem to suggest that, in the generic case, the solutions, while exhibiting complicated 
fine-scale behavior, may have certain structure on the unit scale. A proposed mathematical explanation is that the vorticity $\omega(t)$ converges {\it weakly} but not strongly as $t\to\infty$. This would explain the local chaos versus global structure phenomenon. It is an attractive conjecture, but it  seems hard
to rigorously formulate, let alone prove such a conjecture. We refer to \cite{SverakNotes} for some very interesting discussions in this direction.

\subsection{Asymptotic stability} Here we consider a perturbative regime for the Euler equation (\ref{euler}). More precisely, we consider velocity fields of the form $(b(y),0))+u(x,y)$, which are close to the steady solution $(b(y),0)$ of \eqref{euler}. In vorticity form, the system we are considering is{\footnote{Compared to the system \eqref{euler}--\eqref{stream} there is a slight abuse of notation, in the sense that $u$ is replaced by $(b(y),0))+u(x,y)$ and $\omega$ is replaced by $-b'(y)+\omega(x,y)$. The identity $\psi(x,1)=0$ in \eqref{eu2} and \eqref{eu4} can be assumed to hold after modifying $b$ by a linear flow $c_0+c_1y$.}}
\begin{equation}\label{eu1}
\partial_t\omega+b(y)\partial_x\omega+u\cdot \nabla(-b'(y)+w)=0,
\end{equation}
where $b:\mathbb{R}\to\mathbb{R}$ is a smooth function. Here $\omega:\mathbb{T}\times [0,1]\times I\to\R$ is the main dynamical variable (the vorticity deviation). The velocity field $u$ can be recovered from $\omega$ by first solving the Dirichlet problem 
\begin{equation}\label{eu2}
\Delta\psi=\omega,\qquad \psi(x,0)\equiv\psi(x,1)\equiv0,
\end{equation}
to find the associated stream function $\psi$, and then setting
\begin{equation}\label{eu3}
u=(u^x,u^y):=\nabla^\perp\psi=(-\partial_y\psi,\partial_x\psi).
\end{equation}
Without loss of generality (modifying $b(y)$ by a linear flow $c_0+c_1y$ if necessary), we may assume that $u,\omega$ satisfy the normalization
\begin{equation}\label{eu4}
\int_{\mathbb{T}\times[0,1]}u^x(x,y,t)\,dxdy=\int_{\mathbb{T}\times[0,1]}\omega(x,y,t)\,dxdy\equiv 0\qquad\text{ for any }t\in I,
\end{equation}
which is propagated by the flow.

In this paper we will take $b(y):=y$, and the main equation \eqref{eu1} becomes
\begin{equation}\label{Euler1}
\partial_t\omega+y\partial_x\omega+u\cdot\nabla\omega=0.
\end{equation}
The vector-field $u$ and the stream function $\psi$ are obtained as before. The general case of linear flow $b(y)=c_0+c_1y$ with $c_1\neq0$ can be treated with the same method. Our goal is to understand the long time dynamics of (\ref{Euler1}) with small initial data, that is, we will study the asymptotic stability of the shear flow $(b(y),0)$. 

A first step is to understand the linearized equation
$$\partial_t\omega+y\partial_x\omega=0,$$
which was studied by Orr \cite{Orr}.

To simplify the discussion, let us  ignore the boundary effects for the moment. Hence we assume $y\in \R$. One can then solve this equation explicitly and calculate
$$\omega(t,x,y)=\omega_0(x-yt,y).$$
The equation for the stream function becomes
$$\Delta\psi(t,x,y)=\omega(t,x,y)=\omega_0(x-yt,y)$$
for $(x,y)\in \mathbb{T}\times \R$ and hence
\begin{equation}\label{decaycostderivative}
\widetilde{\psi}(t,k,\xi)=-\frac{\widetilde{\omega}(t,k,\xi)}{k^2+|\xi|^2}=-\frac{\widetilde{\omega_0}(k,\xi+kt)}{k^2+|\xi|^2}.
\end{equation}
In the above, $\widetilde{h}$ denotes the Fourier transform of $h$ in $x,\,y$.
Assume that $\omega_0$ is smooth, so $\widetilde{\omega}_0(k,\xi)$ decays fast in $k,\,\xi$. Then we can view $\xi$ as
$$\xi=-kt+O(1),$$ 
and hence $\widetilde{\psi}(t,k,\xi)$ decays like $|k|^{-2}\langle t\rangle^{-2}$ for each $k\neq 0$. Similarly, using the relations $u^x=-\partial_y\psi$ and $u^y=\partial_x\psi,$
we conclude that $\widetilde{u^x}$ decays like $|k|^{-1}\langle t\rangle^{-1}$ and $\widetilde{u^y}$ decays like $|k|^{-1}\langle t\rangle^{-2}$ for all $k\neq0$. Hence, the velocity field decays to another shear flow $(u_{\infty}(y),0)$. 

One can obtain rigorous and precise estimates with more careful work (with boundary), see \cite{ZhiWu}. The linear damping problem for the case of general monotone shear flow (with boundary) has recently been solved by Wei--Zhang--Zhao in \cite{dongyi}, see also earlier work of Zillinger \cite{Zillinger1,Zillinger2},  and a recent new approach by Grenier et al \cite{Grenier} using techniques from the study of Schr\"{o}dinger operators. 
We also refer the reader to important developments for the linear inviscid damping in the case of non-monotone shear flows \cite{Dongyi2, Dongyi3} and circular flows \cite{Bed2,Zillinger3}.

\subsubsection{Nonlinear stability} The natural question to consider is whether this linear damping mechanism persists for the full nonlinear problem \eqref{Euler1}. One immediately meets several key difficulties in extending the linear analysis to the nonlinear problem, such as the fact that $\partial_y\omega$ grows linearly in time for the linear flow, the decay of $u^x$ is slow and not integrable over time, and the decay of the stream function $\psi$ ``costs" derivatives. In a first approximation, these issues, and others, can be seen by analyzing the explicit formulas above (\ref{decaycostderivative}).

In recent remarkable work Bedrossian--Masmoudi \cite{BedMas} introduced several new and important ideas which addressed these difficulties, and successfully extended the inviscid damping from the linear to the nonlinear level. We will give a brief overview of their methods in subsection \ref{subsection:reviewBedMas} below. Our goal in this paper is to improve on this work in two main directions:

\setlength{\leftmargini}{2.8em}
\begin{itemize}
  \item[(1)] We work in a bounded domain, the channel $\T\times[0,1]$, with the natural no penetration boundary conditions $u^y|_{y=0,1}=0$. The main point is to be able to work with {\it{finite energy}} solutions $u$, which was not possible in the case of the unbounded domain $\T\times\R$ analyzed in \cite{BedMas}. As a result, our analysis has to take into account nontrivial boundary effects.
\smallskip
  \item[(2)] We work with small data in the critical Gevrey regularity space $\G^{\beta,s}$ with exponent $s=1/2$. This is a natural choice in view of the recent work of  Deng--Masmoudi \cite{Deng}, who constructed examples of instability for initial perturbations in a slightly less regular space. It is known that inviscid damping fails for $H^s$ perturbations of the Couettte flow with $s<3/2$, see \cite{ZhiWu}.
\end{itemize}

Before stating our main theorem we remark also that the inviscid damping phenomena have been studied in other contexts, most famously in the Landau damping effect for Vlasov-Poisson equations. We refer to the pioneering work of Landau \cite{Landau} on the linear damping and the celebrated breakthrough of Mouhot--Villani \cite{Villani} on the nonlinear damping, for the physical background and more references. We also refer the interested reader to recent results \cite{Bed4}-\cite{Bed7} and \cite{Dongyi4}, where mixing of the vorticity still plays an important role when considering Navier Stokes equations with small viscosity near the Couette flow.

\subsection{The main theorem} To state our main theorem we define the Gevrey spaces $\mathcal{G}^{\lambda,s}\big(\mathbb{T}\times \R\big)$ as the space of $L^2$ functions $f$ on $\mathbb{T}\times \R$ defined by the norm
\begin{equation}\label{Gev}
\|f\|_{\G^{\lambda,s}(\mathbb{T}\times \R)}:=\big\|e^{\lambda \langle k,\xi\rangle^s}\widetilde{f}(k,\xi)\big\|_{L^2_{k\in\mathbb{Z},\xi\in\R}}<\infty.
\end{equation}
In the above, $\widetilde{f}$ denotes the Fourier transform of $f$ in $x,y$; $s\in(0,1]$ and $\lambda>0$. 

More generally, for any interval $I\subseteq\R$ we define the Gevrey spaces $\mathcal{G}^{\lambda,s}\big(\mathbb{T}\times I\big)$ by
\begin{equation}\label{Gev2}
\|f\|_{\G^{\lambda,s}(\mathbb{T}\times I)}:=\|Ef\|_{\G^{\lambda,s}(\mathbb{T}\times \R)},
\end{equation}
where $Ef(x):=f(x)$ if $x\in I$ and $Ef(x):=0$ if $x\notin I$. For any function $H(x,y)$ let $\langle H\rangle(y)$ denote the average of $H$ in $x$. Our main theorem in this paper is the following:

\begin{theorem}\label{maintheoremINTRO}
Assume that $\beta_0,\vartheta_0\in (0,1/8]$. Then there are constants $\beta_1=\beta_1(\beta_0,\vartheta_0)>0$  and $\overline{\eps}=\overline{\eps}(\beta_0,\vartheta_0)>0$ such that the following statement is true:

Assume that the initial data $\omega_0$ has compact support in $\T\times [2\vartheta_0,1-2\vartheta_0]$ and satisfies
\begin{equation}\label{Eur0}
\|\omega_0\|_{\G^{\beta_0,1/2}(\mathbb{T}\times \R)}=\epsilon\leq\overline{\epsilon},\qquad\int_{\T\times[0,1]}\omega_0(x,y)\,dxdy=0.
\end{equation}
 Let $\omega(t)$ be the smooth solution to the system 
 \begin{equation}\label{Eur1}
 \begin{split}
 &\partial_t\omega+y\partial_x\omega+u\cdot\nabla\omega=0,\\
 &u=(u^x,u^y)=(-\partial_y\psi,\partial_x\psi),\qquad \Delta\psi=\omega,\qquad \psi(x,0)=\psi(x,1)=0,
 \end{split}
 \end{equation}
for $(t,x,y)\in [0,\infty)\times\mathbb{T}\times [0,1]$ with initial data $\omega_0$.
Then we have the following conclusions:

(i) For all $t\ge 0$, ${\rm supp}\,\omega(t)\subseteq \mathbb{T}\times[\vartheta_0,1-\vartheta_0]$.

(ii) There exists $F_{\infty}(x,y) \in \G^{\beta_1,1/2}$ with ${\rm supp}\,F_{\infty}\subseteq \mathbb{T}\times [\vartheta_0,1-\vartheta_0]$ such that for all $t\ge 0$,
\begin{equation}\label{convergence}
\left\|\omega(t,x+ty+\Phi(t,y),y)-F_{\infty}(x,y)\right\|_{\G^{\beta_1,1/2}(\mathbb{T}\times[0,1])} \lesssim_{\beta_0,\vartheta_0}\frac{\epsilon}{\langle t\rangle}.
\end{equation}
Here 
\begin{equation}\label{DefPhi}
\Phi(t,y):=\int_0^t\langle u^x\rangle(\tau,y)\,d\tau
\end{equation}
satisfies
\begin{equation}\label{AsymPhi}
|\Phi(t,y)-tu_{\infty}(y)|\lesssim_{\beta_0,\vartheta_0}\epsilon,
\end{equation}
where $u_{\infty}(y):=\lim\limits_{t\to\infty}\langle u^x\rangle(t,y)$ is given explicitly as
\begin{equation}\label{AsymPhi2}
u_{\infty}(y)=-\partial_y(\Delta^{-1}\langle F_{\infty}\rangle)(y).
\end{equation}
In the above $\Delta$ is inverted with zero Dirichlet boundary condition at $y=0$ and $y=1$. 
 
(iii) The velocity field $u=(u^x,u^y)$ satisfies
\begin{equation}\label{convergenceofux}
\left\|\langle u^x\rangle(t,y)-u_{\infty}(y)\right\|_{\G^{\beta_1,1/2}(\mathbb{T}\times [0,1])}\lesssim_{\beta_0,\vartheta_0}\frac{\epsilon}{\langle t\rangle^2},
\end{equation}

\begin{equation}\label{convergencetomean}
\left\|u^x(t,x,y)-\langle u^x\rangle(t,y)\right\|_{L^{\infty}(\mathbb{T}\times [0,1])}\lesssim_{\beta_0,\vartheta_0}\frac{\epsilon}{\langle t\rangle},
\end{equation}

\begin{equation}\label{convergenceuy}
\left\|u^y(t,x,y)\right\|_{L^{\infty}(\mathbb{T}\times [0,1])}\lesssim_{\beta_0,\vartheta_0}\frac{\epsilon}{\langle t\rangle^2}.
\end{equation}
\end{theorem}

\begin{remark}\label{MainRemarks} (1) We will show in (\ref{rea8}) that for some $c_0\in\R$,
 $$\langle u^x\rangle(t,y)\equiv c_0\qquad{\rm for}\qquad y\in [0,1]\backslash[\vartheta_0,1-\vartheta_0]\quad{\rm and}\quad t\ge0.$$
In particular, $\langle u^x\rangle(t,y)-u_{\infty}(y)$ is compactly supported in $[\vartheta_0,1-\vartheta_0]$.

(2) The assumption that the support of $\omega_0$ is separated from the boundaries $y=0$ and $y=1$ is important. Remarkably, this property persists through the flow for all times, due to the quadratic time decay of the component $u^y$ in \eqref{convergenceuy}. We note that some vanishing assumption on $\omega_0$ at the boundary is likely needed, as suggested by the work \cite{Zillinger1, Zillinger2}, where it was shown that even in the linear case, ``scattering of vorticity" does not hold up to the boundary in high Sobolev spaces if the vorticity does not vanish at the boundary.

(3) The integral normalization condition in (\ref{Eur0}) can be dropped by considering more general linear flows $b(y)=C_0+C_1y$ with $C_1\neq 0$ in equation (\ref{eu1}), without significant changes in the proof.
\end{remark}

\subsection{Main ideas}\label{subsection:reviewBedMas} We describe now some of the main ideas involved in the proof.

\subsubsection{Nonlinear adapted coordinates} This is an important construction introduced by Be\-dro\-ssian-Mas\-mou\-di in  \cite{BedMas}. The idea is to work in a new system of coordinates $z,v$ which are defined as
\begin{equation}\label{changeofcoordinateIntro}
z=x-tv\,\qquad {\rm and}\qquad v=y+\frac{1}{t}\int_0^t\langle u^x\rangle(\tau,y)\,d\tau,
\end{equation}
where $\langle h\rangle$ denotes the average over $x$ of $\langle h\rangle$ for any $h$. The change of variable (\ref{changeofcoordinateIntro}) is a nonlinear refinement of the linear change of coordinates $z=x-ty$, and we refer the reader to Section \ref{sec:variables} for motivation and the detailed calculation. We remark that the nonlinear choice of $v$ in \eqref{changeofcoordinateIntro} is optimal to minimize the transport in $x$ as much as possible. 

The change of variables (\ref{changeofcoordinateIntro}) automatically ``adapts" to the asymptotic profile $u_{\infty}(y)$, which has to be determined by the nonlinear flow, as $t\to\infty$. Given  (\ref{changeofcoordinateIntro}), define 
$$f(t,z,v)=\omega(t,x,y),\,\,\,\phi(t,z,v)=\psi(t,x,y)$$
and 
$$V'(t,v)=\partial_yv(t,y),\,\,\,\,V''(t,v)=\partial_y^2v(t,y),\,\,\,\dot{V}(t,v)=\partial_tv(t,y).$$
Then direct calculations show that $f$ and $\phi$ satisfy the equations
\begin{equation}\label{eq:fIntroFull}
\partial_tf-V'\partial_vP_{\neq0}\phi\,\partial_zf+\dot{V}\,\partial_vf+V'\partial_z\phi\,\partial_vf=0
\end{equation}
and 
\begin{equation}\label{eq:phiIntroFull}
\partial_z^2\phi+(V')^2(\partial_v-t\partial_z)^2\phi+V''(\partial_v-t\partial_z)\phi=f.
\end{equation}
In the above, $P_{\neq0}$ is the projection off the zero mode, i.e., for any function $h(t,z,v)$, 
$$P_{\neq0}h(t,z,v)=h(t,z,v)-\langle h\rangle(t,v).$$
We refer again to Section \ref{sec:variables} for the detailed calculations. 

By the definition of $V',\,V''$, as we are in a perturbative regime, we can expect that
$V'\approx 1$
and that $V''$ is small. The analysis of the functions $V'$, $V''$ and $\dot{V}$ requires significant ideas; however, we ignore these issues for now and make the significant simplifying assumptions
$$V'\equiv 1,\qquad V''\equiv 0, \qquad{\rm and}\qquad \dot{V}\equiv0.$$
 With this (significant) simplification, the equations for $f,\,\phi$ become
\begin{equation}\label{fIntro}
\partial_tf-\partial_vP_{\neq0}\phi\,\partial_zf+\partial_z\phi\,\partial_vf=0
\end{equation}
and
\begin{equation}\label{phiIntro}
\partial_z^2\phi+(\partial_v-t\partial_z)^2\phi=f
\end{equation}
for $(z,v,t)\in \mathbb{T}\times \R\times[0,\infty)$. 
The main task is to show that $f$ remains uniformly regular over all times, from which the inviscid damping of $\omega$ follows relatively easily. 

\subsubsection{Time-dependent imbalanced weights} In Fourier variables, (\ref{phiIntro}) gives
$$\widetilde{\phi}(t,k,\xi)=-\frac{\widetilde{f}(t,k,\xi)}{k^2+(\xi-kt)^2}.$$
Hence for  fixed $k\neq 0,\,\xi\in \R$, $\widetilde{\phi}(t,k,\xi)$, $\mathcal{F}\big(\partial_vP_{\neq0}\phi\big)(t,k,\xi)$ and $\mathcal{F}\big(\partial_z\phi\big)(t,k,\xi)$ have decay rate $\langle t\rangle^{-2}$ as $t\to\infty$. Here $\mathcal{F}(h)(k,\xi)$ denotes the Fourier transform of $h$ in $(z,v)$. Equation (\ref{fIntro}) is thus a transport equation for $f$ with velocity fields having integrable decay rate, and we can expect $f(t)$ to stabilize over time. The proof of stability is however rather complicated due to the ``loss of derivatives" when one tries to take advantage of the decay. More precisely, note that
\begin{equation}\label{partialphi1}
\mathcal{F}\big(\partial_vP_{\neq0}\phi\big)(t,k,\xi)=-\frac{i\xi}{k^2}\frac{\widetilde{f}(t,k,\xi)}{1+|t-\xi/k|^2}\mathbf{1}_{k\neq0}.
\end{equation}
When $|\xi|\gg k^2$, the factor $\xi/k^2$ in (\ref{partialphi1}) indicates a loss of one full derivative in $v$, which is the main difficulty in proving stability.

An important idea, going back to the classical work of Cauchy-Kolwalevski, is to use time dependent norms to control quantities that lose regularity over time. Following \cite{BedMas}, we shall apply this idea and use a carefully designed time-dependent energy functional to control $f$ in (\ref{fIntro}). Define the energy functional
$$\mathcal{E}(t):=\frac{1}{2}\sum_{k\in\mathbb{Z}}\int_{\mathbb{R}}A_k^2(t,\xi)\big|\widetilde{f}(t,k,\xi)\big|^2\,d\xi.$$

The goal is to find a suitable $A_k(t,\xi)$ which is decreasing in $t$, so that the associated energy $\mathcal{E}(t)$ is decreasing over time. To illustrate the essential difficulty, we shall only consider the term $\partial_vP_{\neq0}\phi\,\partial_zf$ in the equation (\ref{fIntro}) for $\partial_tf$. We shall also assume that $\phi$ has much higher frequency than $f$. These reductions are motivated by the fact that precisely in this case {\it decay costs the most regularity}, which forces the choice of the main weight $A_k(t,\xi)$.

 Assuming this simplification and considering all other terms as error terms, we arrive at
\begin{equation}\label{dEIn}
\begin{split}
\frac{d}{dt}\mathcal{E}(t)&\leq\sum_{k\in\mathbb{Z}}\int_{\mathbb{R}}\frac{\partial_tA_k(t,\xi)}{A_k(t,\xi)}A_k^2(t,\xi)\big|\widetilde{f}(t,k,\xi)\big|^2\,d\xi\\
&+\epsilon\sum_{k=\ell+O(1)}\int_{\xi=\eta+O(1)}A_k^2(t,\xi)\left|\widetilde{f}(t,k,\xi)\right|\frac{|\eta(k-\ell)|\mathbf{1}_{\ell\neq0}}{\langle t-\eta/\ell\rangle\ell^2}\left|\widetilde{f}(t,\ell,\eta)\right|\,d\xi\,d\eta+{\rm error}\\
&=CK_f+R+{\rm error}.
\end{split}
\end{equation}
In the above, $CK_f$ denotes ``Cauchy-Kowalevski" as this favorable term comes from the decrease in weight, analogous to the classical ideas of Cauchy and Kowalevski, and $\epsilon$ is a small number coming from the smallness of $f$.

The main difficulty is to deal with the case $|\eta|\gg\ell^2$. In this case the factor 
$$\frac{\eta}{\ell^2}\frac{1}{1+| t-\eta/\ell|^2}$$
cannot be controlled in the standard way, using the $CK_f$ term. The key original idea of \cite{BedMas} is to define and use {\it imbalanced weights} to absorb this large factor, taking advantage of the asymmetry between $k$ and $\ell$. Algebraically, the main observation is that if $\left|t-\eta/\ell\right|\ll\big|\eta/\ell^2\big|$, $k\neq\ell$, $\xi=\eta+O(1)$, $k\sim\ell$, $|\eta|\gg \ell^2$, then
$$\left|t-\xi/k\right|\ge\left|\eta/\ell-\eta/k\right|-\left|t-\eta/\ell\right|-\left|\eta/k-\xi/k\right|\gtrsim\left|\xi/k^2\right|.$$
Hence, if $t$ is close to the critical time $\eta/\ell$ then $t$ is far from the critical time $\xi/k$, which is the asymmetry we indicated earlier. This asymmetry turns out to allow one to define a weight which is so imbalanced that
\begin{equation}\label{imbaweight}
\frac{A_{\ell}(t,\eta)}{A_{k}(t,\xi)}\sim\left|\frac{\eta}{\ell^2}\right|\frac{1}{1+|t-\eta/\ell|},
\end{equation}
when $k\neq \ell$, $\xi=\eta+O(1)$, $k=\ell+O(1)$, and $t$ is resonant for the frequency $(\ell,\eta)$ (meaning $\left|t-\eta/\ell\right|\ll|\eta|/\ell^2$). Therefore, even though the factor 
$\left|\frac{\eta}{\ell^2}\right|\frac{1}{1+|t-\eta/\ell|}$ in the term $R$ in \eqref{dEIn}
can be very large, the switch from $A_k(t,\xi)$ to $A_{\ell}(t,\eta)$ helps absorb this factor, and the resulting contribution can be estimated. See the proof of \eqref{TLXH1.2} (in particular the bounds \eqref{nonRtR}) for the precise details.

The use of imbalanced weights and imbalanced energy functionals successfully resolves the large factor problem. However, it also brings numerous new issues, which are not seen in the standard case of balanced weights, i.e. weights $B$ satisfying $B(k+O(1),\xi+O(1))\approx B(k,\xi)$.  As a result, one has to deal with a large number of cases and the argument becomes correspondingly complicated. One also needs to obtain sufficiently strong control on the coordinate functions $V',\,V'',\,\dot{V}$ which we have ignored in our discussion. 

\subsubsection{The effect of boundary} For the Couette flow in a channel, in addition to the ideas outlined above that were introduced in \cite{BedMas}, we also need to control the boundary effect. The boundary effect is most visible from the equation (\ref{eu2}) for the stream function $\psi(t,x,y)$, which now has zero Dirichlet boundary conditions at $y=0,1$. 

Since the argument in \cite{BedMas} depends crucially on Fourier transforms, a possible idea would be to extend, in a suitable way, the stream function $\psi$ from $[0,1]$ to $\mathbb{R}$. One can, for example, set $\psi$ to be zero for $y\in \mathbb{R}\backslash [0,1]$. However, such an extension is not smooth at $y=0,1$, as $\partial_y\psi|_{y=0,\,1}$ does not vanish in general. This is a problem as we need to work in the Gevrey space.  To obtain an extension that is compatible with Gevrey regularity, one would need to match infinitely many derivatives $\partial_y^k\psi$ at $y=0,\,1$. This seems complicated, as all the derivatives depend on the nonlinear flow.

To resolve this difficulty, we first make the key observation that if initially the vorticity $\omega_0$ is compactly supported away from the boundary $y=0,\,1$, then by the decay of $u^y$, $$\left|u^y(t,x,y)\right|\lesssim\frac{\epsilon}{\langle t\rangle^2},$$
we can expect that the support of $\omega(t)$ will stay away from $y=0,\,1$  for all $t\ge0$. In the transport term 
$$u\cdot\nabla\omega=-\partial_y\psi\,\partial_x\omega+\partial_x\psi\,\partial_y\omega,$$
we can therefore replace $\psi$ by $\Psi\psi$ with some Gevrey class cutoff function $\Psi\in C_0^{\infty}(0,1)$ with $\Psi\equiv 1$ on the support of $\omega$. In the variables $z,\,v$, the corresponding $\phi$ can be replaced $\Phi\phi$ where $\Phi(v)=\Psi(y)$. Formally the localized stream function $\Phi\phi$ is in Gevrey space and we can adapt the method of \cite{BedMas} to our case.

One still needs to control the effect of the Dirichlet boundary condition. The main property we need is, roughly speaking,
\begin{equation}\label{eq:ellipticIntroRough}
\big(\partial_z^2+(\partial_v-t\partial_z)^2\big)(\Phi\phi)\sim f.
\end{equation}
See (\ref{eq:phiIntroFull}) for the equation for $\phi$. Here ``$h_1\sim h_2$" means that $h_1$ and $h_2$ have ``similar" regularity. By the equation (\ref{eq:phiIntroFull}) for $\phi$, in order to show (\ref{eq:ellipticIntroRough}), we need to bound a number of terms including those coming from the boundary. The key fact for us is that the localized Green function $\Phi(v)\,G_k(j,v)$, $j=0\,\,{\rm or}\,\,1$, is in $\mathcal{G}^{\lambda,s}$ for some $s>\frac{1}{2}$ since $\Phi$ is supported away from $v=0,1$. Here $G_k$ is the Green function of  $-\frac{d^2}{dy^2}+k^2$ with zero Dirichlet boundary conditions. The regularity of $\Phi(v)\,G_k(j,v)$ with $j=0,\,1$  allows us to obtain crucial bounds on $\partial_v\phi$ at the boundary, which encodes the boundary effect. One can show that the boundary normal derivative $\partial_v\phi|_{v=0, {\rm or}\,1}$ decays rapidly over time, faster than the reciprocal of any polynomial, see \eqref{dar21} below. However, the functions ``generated" by such boundary effect do not have uniform regularity, see \eqref{dar5} below. The issue can already be seen by considering the simple function $\cos{x}$ which, written in the new coordinates $(z,v)$, is $\cos{(z+tv)}$ and loses regularity very quickly as $t\to\infty$. In then end,  the fast decay of $\partial_v\phi|_{v=0, {\rm or}\,1}$ exactly compensates the loss of regularity, and there is no room to spare in our estimates, which demonstrates, on a technical level at least, the necessity of assuming the support of $\omega$ to be separated from the boundary of the channel. 

\subsubsection{Sharper $\G^{\lambda,1/2}$ regularity} Our second aim is to work in the ``sharp" $\mathcal{G}^{\lambda,1/2}$ Gevrey space for the perturbation, and we need to be much more precise with the definition of the main weights $A_k(t,\xi)$. Here our first main idea is to introduce an intermediate scale $\delta$, with
$$\epsilon\ll\delta\ll 1.$$
In the work \cite{BedMas}, for small times $t\lesssim\sqrt{\xi}$, the weight $w_k(t,\xi)$ has no growth and hence provide us with no useful control of the nonlinearity. The choice is dictated by the delicate balance between the requisite continuity of $w_k(t,\xi)$ in $\xi$ and the discrete nature of the definition of the weight $w_k(t,\xi)$. In this work, the weight $w_k(t,\xi)$ is designed to satisfy
$$\frac{\partial_tw_k(t,\xi)}{w_k(t,\xi)}\sim \delta^3$$
for small times $t\ll \delta^{-\frac{3}{2}}\sqrt{\xi}$, hence the weight will provide us with powerful control on the nonlinearity even for small times. See section \ref{weightsdefin} for the definitions of the weights. 

Such a definition necessarily creates a discontinuity in $\xi$ for the weight $w_k(t,\xi)$. On the other hand, the smoothness of the weights in $\xi$ is important to treat several terms in the nonlinearity when we wish to take advantage of cancellations, such as in the transport structure. To resolve this issue, we redefine our main weights by taking a suitable average over $\xi$ (see \eqref{dor1}). The average is designed carefully (in time-dependent fashion) so that the necessary bounds are preserved while one can also restore certain smoothness for the weight in $\xi$, especially for $t\lesssim \sqrt{|\xi|}$. The weights thus obtained enjoy optimal smoothness in $\xi$ and this allows us to analyze efficiently the transport terms in the problem.

\subsection{Organization} The rest of the paper is organized as follows. In section 2 we introduce the change of variables and set up the main bootstrap Proposition \ref{MainBootstrap}. In section 3 we give a proof of the main theorem \ref{maintheoremINTRO} assuming the bootstrap proposition. In sections 4-6 we prove the main bootstrap Proposition \ref{MainBootstrap}. In section 7 we define the weights and prove some of their basic properties. In section 8 we prove several weighted multiplicative inequalities involving the weights needed in the proof of the main bootstrap proposition. We note that the bounds on the weights proved in section 7-8 are used throughout sections 4-6. Finally, in Appendix A we review some properties of the Gevrey spaces and provide a self-contained proof of local well-posedness of the two dimensional Euler equation in Gevrey spaces.

\section{The main equations and the bootstrap proposition}

\subsection{The change of variables}\label{sec:variables}
Assume that $\omega:[0,T]\times\mathbb{T}\times[0,1]$ is a sufficiently smooth solution of the system \eqref{Eur1} on some time interval $[0,T]$. Due to the transport term $y\partial_x\omega$, as time goes to infinity, $\omega(t)$ has larger and larger derivatives in $y$. To obtain a uniform estimate, we need to
 choose correct coordinates, which ``unwind'' the transport. 

To this end, let us define
\begin{equation}\label{rea0}
 v=y+h(t,y),\qquad z=x-tv.
\end{equation}
We illustrate now what are the right choices for $h$. Since we are in the perturbative regime, we can imagine that $h$ is small, so that the transformation 
from $y$ to $v$ is invertible. Denote
\begin{equation}\label{rea1}
f(t,z,v):=\omega(t,x,y),\qquad \phi(t,z,v):=\psi(t,x,y).
\end{equation}
Our first task is to rewrite the equation \eqref{Euler1} in terms of $f$ in the variables $z,v,t$. Direct calculation yields
\begin{equation}\label{rea2}
\begin{split}
\partial_t\omega&=\partial_tf-\partial_t(tv)\,\partial_zf+\partial_tv\,\partial_vf;\\
\partial_y\omega&=-t\partial_yv\,\partial_zf+\partial_yv\,\partial_vf;\\
\partial_x\omega&=\partial_zf.
\end{split}
\end{equation}
Similar relations hold if we replace $f$ and $\omega$ by $\phi$ and $\psi$. We shall denote
$V',V''$ and $\dot{V}$ as the functions $\partial_yv(t,y),\,\partial_{yy}v(t,y)$ and $\partial_tv(t,y)$ in the variables $t,v$. That is,
\begin{equation}\label{rea3}
V'(t,v):=\partial_yv(t,y);\qquad V''(t,v):=\partial_{yy}v(t,y);\qquad \dot{V}(t,v):=\partial_tv(t,y).
\end{equation}
We also denote $U(t,z,v)=(U^z, U^v)(t,z,v):=u(t,x,y)$, and notice that 
\begin{equation}\label{eq:uz}
U^z(t,z,v)=-\partial_y\psi(t,x,y)=tV'\,\partial_z\phi-V'\,\partial_v\phi
\end{equation}
and
\begin{equation}\label{eq:uv}
 U^v(t,z,v)=\partial_z\phi.
\end{equation}
Then, we obtain
\begin{equation*}
\begin{split}
 0&=\,\partial_tf-\partial_t(tv)\partial_zf+\partial_tv\,\partial_vf+y\partial_zf+\,\big(tV'\,\partial_z\phi-V'\,\partial_v\phi\big)\,\partial_zf+\partial_z\phi\,\big(V'\,\partial_vf-tV'\,\partial_zf\big)\\
 &=\,\partial_tf+\big[-\partial_t(tv)+y+tV'\,\partial_z\phi-V'\,\partial_v\phi-t\partial_z\phi\,V'\big]\partial_zf+\,\big(\dot{V}+\partial_z\phi\,V'\big)\partial_vf\\
 &=\,\partial_tf+\big[-\partial_t(tv)+y-V'\,\partial_v\phi\big]\partial_zf+\,\big[\dot{V}+\partial_z\phi\,V'\big]\partial_vf.
\end{split}
\end{equation*}
In order to avoid terms with power of $t$, it is convenient to write
\begin{equation}\label{rea7}
h(t,y)=\frac{1}{t}\int_0^tg(\tau,y)\,d\tau.
\end{equation}
Then, using \eqref{rea0}, $$-\partial_t(tv)+y=-g(t,y).$$
Consequently, the equation (\ref{Euler1}) in the variables $z,v,t$ becomes
\begin{equation}\label{changeofvariableequationf}
\partial_tf-\big[g(t,y)+V'\partial_v\phi\big]\partial_zf+(\dot{V}+V'\partial_z\phi)\,\partial_vf=0.
\end{equation}

 We still have freedom to choose $g$. We remark that the main purpose of the change of variables is to unwind the mixing in the $x$ direction. Denoting 
 $\big<w\big>(t,y)$ as the average of $w$ in the $x$ direction (and equivalently in the $z$ direction), we have
 $$\big<V'\partial_v\phi\big>=-\big<tV'\,\partial_z\phi-V'\,\partial_v\phi\big>=-\big<u^x\big>(t,y),$$
see \eqref{eq:uz}. Thus, to optimally cancel the mixing in the $x$ direction, it is reasonable to choose $g(t,y)=\big<u^x\big>(t,y)$.
To summarize, we make the change of variables 
 \begin{equation}\label{changeofvariables1}
 v=y+\frac{1}{t}\int_0^t\big<u^x\big>(\tau,y)\,d\tau,\qquad z=x-tv.
 \end{equation}
Under this change of variable, the equation (\ref{Euler1}) becomes
\begin{equation}\label{main}
 \partial_tf-V'\partial_vP_{\neq 0}\phi\,\partial_zf+(\dot{V}+V'\partial_z\phi)\,\partial_vf=0,
\end{equation}
where $P_{\neq 0}$ is projection off the zero mode, i.e., for any function $H(t,z,v)$
$$P_{\neq 0}H(t,z,v)=H(t,z,v)-\langle H\rangle(t,v).$$

There is a technical issue in the change of variables that we need to deal with. We know that $y\in[0,1]$, and we would like to find the range of the variable $v$. In this paper, we work with the assumption that ${\rm supp}\,\omega(t)\subset \mathbb{T}\times[\vartheta_0,1-\vartheta_0]$ for all $t\in[0,T]$ (we need to prove, of course, that such an assumption is consistent as long as the initial vorticity is supported in a smaller region and is sufficiently small). 
With this assumption, we show that 
\begin{equation}\label{rea8}
\big<u^x\big>(t,y)=c_0\qquad\text{ for any }t\in[0,T]\text{ and }y\in[0,\vartheta_0]\cup [1-\vartheta_0, 1].
\end{equation}
Indeed, since $\psi$ is harmonic in $\mathbb{T}\times\left[[0,\vartheta_0]\cup [1-\vartheta_0, 1]\right]$ we have
\begin{equation}\label{harmonic}
 \partial_y\big<u^x\big>=-\big<\partial_{yy}\psi\big>=\big<\partial_{xx}\psi\big>=0,\qquad {\rm for}\,\,y\in[0,\vartheta_0]\cup [1-\vartheta_0, 1].
\end{equation}
On the other hand, using the Euler equation for the velocity component $u^x+y$, we have
\begin{equation}\label{perEu}
\partial_tu^x+(u^x+y)\partial_xu^x+u^y\partial_y(u^x+y)+\partial_xp=0.
\end{equation}
Therefore, from (\ref{harmonic}), we get that
\begin{equation}\label{rea9}
\partial_t\big<u^x\big>(t,y)=0\qquad {\rm for}\,\,y\in[0,\vartheta_0]\cup [1-\vartheta_0, 1].
\end{equation}
Furthermore, using $-\partial_yu^x+\partial_xu^y=\omega$ and the assumption \eqref{eu4}, integrating over $\mathbb{T}\times[0,1]$, we get that $\big<u^x\big>(t,1)=\big<u^x\big>(t,0)$. The conclusion \eqref{rea8} follows. Therefore 
\begin{equation}\label{rea10}
v\in[c_0,1+c_0]\,\,\text{ and }\,\,{\rm supp}\,f(t)\subset \mathbb{T}\times[c_0+\vartheta_0,\,c_0+1-\vartheta_0]\,\,\text{ for any }\,\,t\in[0,T].
\end{equation}

We consider now the equation for $\phi$. From the relations 
\begin{equation}\label{rea12}
\partial_x\psi=\partial_z\phi,\qquad \partial_y\psi=V'(\partial_v\phi-t\partial_z\phi)=V'(\partial_v-t\partial_z)\phi,
\end{equation}
 we get that
\begin{equation}\label{rea12.1}
\partial_{xx}\psi=\partial_{zz}\phi;\quad \partial_{yy}\psi=(V')^2(\partial_v-t\partial_z)^2\phi+V''(\partial_v-t\partial_z)\phi.
\end{equation}
Recalling the equation $\Delta\psi=\omega$, we see that $\phi$ satisfies
\begin{equation}\label{eq:Delta_t}
\partial_z^2\phi+(V')^2(\partial_v-t\partial_z)^2\phi+V''(\partial_v-t\partial_z)\phi=f,
\end{equation}
with $\phi(t,x,c_0)=\phi(t,x,1+c_0)=0$ for any $t\in[0,T]$ and $x\in\mathbb{T}$.

\subsubsection{The functions $V'$ and $\dot{V}$} We also need to understand the (nonlinear) change of coordinates. Using \eqref{changeofvariables1} and the observation $-\partial_y\big<u^x\big>=\big<-\partial_yu^x+\partial_xu^y\big>=\langle\omega\rangle$, we have
\begin{equation}\label{rea15}
\begin{split}
\partial_yv(t,y)&=1-\frac{1}{t}\int_0^t\big<\omega\big>(\tau,y)\,d\tau,\\
\partial_tv(t,y)&=\frac{1}{t}\Big[-\frac{1}{t}\int_0^t\big<u^x\big>(\tau,y)\,d\tau+\big<u^x\big>(t,y)\Big],\\
\partial_y\partial_tv(t,y)&=\frac{1}{t}\Big[\frac{1}{t}\int_0^t\big<\omega\big>(\tau,y)\,d\tau-\big<\omega\big>(t,y)\Big].
\end{split}
\end{equation}
Thus
\begin{equation}\label{rea15.5}
-\frac{1}{t}\int_0^t\big<\omega\big>(\tau,y)d\tau=V'(t,v(t,y))-1.
\end{equation}
By the chain rule it follows that
\begin{equation}\label{rea16}
\partial_t\big[t(V'(t,v)-1)\big]+t\dot{V}(t,v)\partial_vV'(t,v)=-\big<f\big>(t,v):=-\frac{1}{2\pi}\int_{\mathbb{T}}f(t,z,v)\,dz.
\end{equation}
We notice that
\begin{equation}\label{rea17}
\partial_y(\partial_tv(t,y))=\partial_y\big[\dot{V}(t,v(t,y))\big]=V'(t,v(t,y))\partial_v\dot{V}(t,v(t,y)).
\end{equation}
Hence, using the last identity in \eqref{rea15} and the identities \eqref{rea15.5} and \eqref{rea17}, we have
\begin{equation}\label{rea18}
tV'(t,v)\partial_v\dot{V}(t,v)=1-V'(t,v)-\big<f\big>(t,v).
\end{equation}

We derive now our main evolution equations. It follows from \eqref{rea16} and \eqref{rea18} that
\begin{equation}\label{rea19}
\partial_t(V'-1)=V'\partial_v\dot{V}-\dot{V}\partial_v(V'-1).
\end{equation}
Set
\begin{equation}\label{rea19.5}
\mathcal{H}:=tV'\partial_v\dot{V}=1-V'-\big<f\big>.
\end{equation}
Using \eqref{rea19} and \eqref{main} we calculate
\begin{equation*}
\partial_t\mathcal{H}=-\partial_t(V'-1)-\partial_tf_0=-V'\partial_v\dot{V}+\dot{V}\partial_v(V'-1)-V'\big<\partial_vP_{\neq 0}\phi\,\partial_zf\big>+\big<(\dot{V}+V'\partial_z\phi)\,\partial_vf\big>
\end{equation*}
Using again \eqref{rea19.5} and simplifying, we get 
\begin{equation*}
\partial_t\mathcal{H}=-\frac{\mathcal{H}}{t}-\dot{V}\partial_v\mathcal{H}-V'\big<\partial_vP_{\neq0}\phi\,\partial_zf\big>+V'\big<\partial_z\phi\,\partial_vf\big>.
\end{equation*}

We summarize our calculations so far in the following:

\begin{proposition}\label{ChangedEquations} Assume $\omega:[0,T]\times\mathbb{T}\times[0,1]\to\mathbb{R}$ is a sufficiently smooth solution of the system \eqref{Eur1} on some time interval $[0,T]$. Assume that $\omega(t)$ is supported in $\mathbb{T}\times[\va_0,1-\va_0]$ and that $\|\langle\omega\rangle(t)\|_{H^{10}}\ll 1$  for all $t\in[0,T]$.\footnote{The smallness of $\|\langle\omega\rangle(t)\|_{H^{10}}$ is a qualitative condition that is only needed to guarantee that the map $y\to v$ is indeed a smooth bijective change of coordinates.} Then there is $c_0\in\R$ such that
\begin{equation}\label{rea20.5}
\big<u^x\big>(t,y)=c_0\qquad\text{ for any }t\in[0,T]\text{ and }y\in [0,\vartheta_0]\cup [1-\vartheta_0, 1].
\end{equation}

We define the change-of-coordinates functions $(z,v):\mathbb{T}\times[0,1]\to\mathbb{T}\times[c_0,c_0+1]$, 
 \begin{equation}\label{rea20}
 v:=y+\frac{1}{t}\int_0^t\big<u^x\big>(\tau,y)\,d\tau,\qquad z:=x-tv,
 \end{equation}
and the new variables $f,\phi:[0,T]\times\mathbb{T}\times[c_0,c_0+1]\to\mathbb{R}$ and $V',V'',\dot{V},\mathcal{H}:[0,T]\times[c_0,c_0+1]\to\mathbb{R}$, 
\begin{equation}\label{rea21}
f(t,z,v):=\omega(t,x,y),\qquad \phi(t,z,v):=\psi(t,x,y),
\end{equation}
\begin{equation}\label{rea22}
V'(t,v):=\partial_yv(t,y),\qquad V''(t,v)=\partial_{yy}v(t,y),\qquad \dot{V}(t,v)=\partial_tv(t,y),
\end{equation}
\begin{equation}\label{rea23}
\mathcal{H}(t,v):=tV'(t,v)\partial_v\dot{V}(t,v)=1-V'(t,v)-\langle f\rangle(t,v).
\end{equation}
Then the new variables $f$, $V'-1$, and $\mathcal{H}$ are supported in $[0,T]\times\mathbb{T}\times[c_0+\va_0,c_0+1-\va_0]$ and satisfy the evolution equations
\begin{equation}\label{rea23.1}
 \partial_tf=V'\partial_vP_{\neq 0}\phi\,\partial_zf-(\dot{V}+V'\partial_z\phi)\,\partial_vf,
\end{equation}
\begin{equation}\label{rea24}
\partial_t(V'-1)=\mathcal{H}/t-\dot{V}\partial_v(V'-1),
\end{equation}
\begin{equation}\label{rea25}
\partial_t\mathcal{H}=-\mathcal{H}/t-\dot{V}\partial_v\mathcal{H}-V'\big<\partial_vP_{\neq0}\phi\,\partial_zf\big>+V'\big<\partial_z\phi\,\partial_vf\big>.
\end{equation}
The variables $\phi$, $V''$, and $\dot{V}$ satisfy the elliptic-type identities
\begin{equation}\label{rea26}
\partial_z^2\phi+(V')^2(\partial_v-t\partial_z)^2\phi+V''(\partial_v-t\partial_z)\phi=f,
\end{equation}
\begin{equation}\label{rea27}
\partial_v\dot{V}=\mathcal{H}/(tV'),\qquad \dot{V}(t,c_0)=\dot{V}(t,1+c_0)=0,\qquad V''=V'\partial_vV'.
\end{equation}
\end{proposition}

\subsection{Weights, energy functionals, and the bootstrap proposition}\label{weightsdef} In this subsection we construct our main energy functionals and state our main bootstrap proposition.

\subsubsection{Definition of the main weights} As discussed in the introduction, the key idea in controlling the nonlinear effect is to estimate the increment of suitable energy functionals, which are defined using special weights. These special weights are ``imbalanced" and can distinguish ``resonant" and ``non-resonant" times.

To state our main bootstrap proposition we need to define three main weights $A_{NR}$, $A_R$, and $A_k$. Fix $\delta_0>0$ and for small $\sigma_0>0$ (say $\sigma_0=0.01$), we define the function $\lambda$ by
\begin{equation}\label{reb10.5}
\lambda(0)=\frac{3}{2}\delta_0,\,\,\,\,\lambda'(t)=-\frac{\delta_0\sigma_0^2}{\langle t\rangle^{1+\sigma_0}}.
\end{equation}
In particular, $\lambda$ is decreasing on $[0,\infty)$ and $\lambda(t)\in[5\delta_0/4,3\delta_0/2]$. Define
\begin{equation}\label{reb11}
A_R(t,\xi):=\frac{e^{\lambda(t)\langle\xi\rangle^{1/2}}}{b_R(t,\xi)}e^{\sqrt{\delta}\langle\xi\rangle^{1/2}},\qquad A_{NR}(t,\xi):=\frac{e^{\lambda(t)\langle\xi\rangle^{1/2}}}{b_{NR}(t,\xi)}e^{\sqrt{\delta}\langle\xi\rangle^{1/2}},
\end{equation}
where $\delta>0$ is a small parameter that may depend only on $\delta_0$ and $\vartheta_0$. Then we define
\begin{equation}\label{reb12}
A_k(t,\xi):=e^{\lambda(t)\langle k,\xi\rangle^{1/2}}\Big(\frac{e^{\sqrt{\delta}\langle\xi\rangle^{1/2}}}{b_k(t,\xi)}+e^{\sqrt{\delta}|k|^{1/2}}\Big).
\end{equation}

The precise definitions of the weights $b_{NR},\,b_R,\,b_k$ are very important; all the details are provided in section \ref{weights}. For now we simply note that, for any $t,\xi,k$,
\begin{equation}\label{reb13}
e^{-\delta\sqrt{|\xi|}}\leq b_R(t,\xi)\leq b_k(t,\xi)\leq b_{NR}(t,\xi)\leq 1,
\end{equation}
 In other words, the weights $1/b_{NR},\,1/b_R,\,1/b_k$ are small when compared to the main factors $e^{\lambda(t)\langle\xi\rangle^{1/2}}$ and $e^{\lambda(t)\langle k,\xi\rangle^{1/2}}$ in the weights $A_{NR},\,A_R,\,A_k$. However, their relative contributions are important as they are used to distinguish between ``resonant" and ``non-resonant" times. 

\subsubsection{The main bootstrap proposition} As is Proposition \ref{ChangedEquations}, assume that $\omega:[0,T]\times\mathbb{T}\times[0,1]\to\mathbb{R}$ is a sufficiently smooth solution of the system \eqref{eu2}--\eqref{Euler1} on some time interval $[0,T]$, which is supported in $\mathbb{T}\times[\va_0,1-\va_0]$ and satisfies $\|\langle\omega\rangle(t)\|_{H^{10}}\ll 1$  for all $t\in[0,T]$. Define $f,V',\mathcal{H},\phi$ as in \eqref{rea21}--\eqref{rea23}, and recall that $f,V'-1,\mathcal{H}$ are supported in $\mathbb{T}\times[c_0+\va_0,c_0+1-\va_0]\times[0,T]$. We define
\begin{equation}\label{defgellip} 
\Theta(t,z,v):=(\partial_z^2+(\partial_v-t\partial_z)^2)\left(\Psi(v)\,\phi(t,z,v)\right),
\end{equation}
 where $\Psi:\mathbb{R}\to[0,1]$ is a Gevrey class cut-off function, satisfying
\begin{equation}\label{rec0}
\begin{split}
&\big\|e^{\langle\xi\rangle^{3/4}}\widetilde{\Psi}(\xi)\big\|_{L^\infty}\lesssim 1,\\
&{\rm supp}\,\Psi\subseteq \big[c_0+\va_0/4,c_0+1-\va_0/4\big], \quad\Psi\equiv 1\text{ in }\big[c_0+\va_0/3,c_0+1-\va_0/3\big].
 \end{split}
\end{equation}
See subsection \ref{GevSec} for the construction of such functions $\Psi$.

We define the energy functionals
\begin{equation}\label{rec1}
\mathcal{E}_f(t):=\sum_{k\in \mathbb{Z}}\int_{\R}A_k^2(t,\xi)\big|\widetilde{f}(t,k,\xi)\big|^2\,d\xi,
\end{equation}
\begin{equation}\label{rec2}
\mathcal{E}_{V'-1}(t):=\int_{\R}A_R^2(t,\xi)\big|\widetilde{(V'-1)}(t,\xi)\big|^2\,d\xi,
\end{equation}
\begin{equation}\label{rec3}
\mathcal{E}_{\mathcal{H}}(t):=K_\delta^{2}\int_{\R}A_{NR}^2(t,\xi)\big(\langle t\rangle/\langle\xi\rangle\big)^{3/2}\big|\widetilde{\mathcal{H}}(t,\xi)\big|^2\,d\xi,
\end{equation}
\begin{equation}\label{rec3.5}
\mathcal{E}_{\Theta}(t):=\sum_{k\in \mathbb{Z}\setminus\{0\}}\int_{\R}A_k^2(t,\xi)\frac{|k|^2\langle t\rangle^2}{|\xi|^2+|k|^2\langle t\rangle^2}\big|\widetilde{\Theta}(t,k,\xi)\big|^2\,d\xi,
\end{equation}
where $K_\delta\geq 1$ is a large constant that depends only on $\delta$ (in fact, it depends on the implicit constants in Lemmas \ref{bweights}--\ref{lm:CDW}). We define also $\dot{A}_\ast(t,\xi):=(\partial_t A_\ast)(t,\xi)$, $\ast\in\{NR,R,k\}$, and the space-time integrals
\begin{equation}\label{rec4}
\mathcal{B}_f(t):=\int_1^t\sum_{k\in \mathbb{Z}}\int_{\R}|\dot{A}_k(s,\xi)|A_k(s,\xi)\big|\widetilde{f}(s,k,\xi)\big|^2\,d\xi ds,
\end{equation}
\begin{equation}\label{rec5}
\mathcal{B}_{V'-1}(t):=\int_1^t\int_{\R}|\dot{A}_R(s,\xi)|A_R(s,\xi)\big|\widetilde{(V'-1)}(s,\xi)\big|^2\,d\xi ds,
\end{equation}
\begin{equation}\label{rec6}
\mathcal{B}_{\mathcal{H}}(t):=K_\delta^{2}\int_1^t\int_{\R}|\dot{A}_{NR}(s,\xi)|A_{NR}(s,\xi)\big(\langle s\rangle/\langle\xi\rangle\big)^{3/2}\big|\widetilde{\mathcal{H}}(s,\xi)\big|^2\,d\xi ds,
\end{equation}
\begin{equation}\label{rec6.5}
\mathcal{B}_{\Theta}(t):=\int_1^t\sum_{k\in \mathbb{Z}\setminus\{0\}}\int_{\R}|\dot{A}_k(s,\xi)|A_k(s,\xi)\frac{|k|^2\langle s\rangle^2}{|\xi|^2+|k|^2\langle s\rangle^2}\big|\widetilde{\Theta}(s,k,\xi)\big|^2\,d\xi ds.
\end{equation}

Our main proposition is the following: 

\begin{proposition}\label{MainBootstrap}
Assume $T\geq 1$ and $\omega\in C([0,T]:\G^{2\delta_0,1/2})$ is a sufficiently smooth solution of the system \eqref{Eur1}, with the property that $\omega(t)$ is supported in $\mathbb{T}\times[\va_0,1-\va_0]$ and that $\|\langle\omega\rangle(t)\|_{H^{10}}\ll 1$  for all $t\in[0,T]$. Define $f,\phi,\Theta,\,V',V'',\dot{V},\mathcal{H}$ as above. Assume that $\eps_1$ is sufficiently small depending on $\delta_0$,
\begin{equation}\label{boot1}
\sum_{g\in\{f,V'-1,\mathcal{H},\Theta\}}\mathcal{E}_g(t)\leq\eps_1^3\qquad\text{ for any }t\in[0,1],
\end{equation}
and
\begin{equation}\label{boot2}
\sum_{g\in\{f,V'-1,\mathcal{H},\Theta\}}\big[\mathcal{E}_g(t)+\mathcal{B}_g(t)\big]\leq\eps_1^2\qquad\text{ for any }t\in[1,T].
\end{equation}
Then for any $t\in[1,T]$ we have the improved bounds
\begin{equation}\label{boot3}
\sum_{g\in\{f,V'-1,\mathcal{H},\Theta\}}\big[\mathcal{E}_g(t)+\mathcal{B}_g(t)\big]\leq\eps_1^2/2.
\end{equation}
Moreover, for $g\in\{f,\Theta\}$, we have the stronger bounds for $t\in[1,T]$
\begin{equation}\label{boot3'}
\sum_{g\in\{f,\Theta\}}\big[\mathcal{E}_g(t)+\mathcal{B}_g(t)\big]\lesssim_{\delta}\eps_1^3.
\end{equation}

\end{proposition}

The proof of Proposition \ref{MainBootstrap} is the main part of this paper, and covers sections \ref{fimprov}--\ref{coimprov2}. In the next section we show how to use this proposition to prove our main theorem. As a matter of notation, all implicit constants in inequalities such as \eqref{boot3'} are allowed to depend on $\va_0$, here and in the rest of the paper.

\section{Proof of the main theorem}\label{mainProof}
In this section, we show how to use Proposition \ref{MainBootstrap} to prove our main Theorem \ref{maintheoremINTRO}. We will also need the following local regularity lemma.

\begin{lemma}\label{lm:persistenceofhigherregularity}
Assume that $s\in[1/4,3/4]$, $\lambda_0\in(0,1)$, $\vartheta\in(0,1/4]$, and that $\mathrm{supp}\,\omega_0\subseteq \T\times [\vartheta,1-\vartheta]$. Assume also that
\begin{equation}\label{ini1}
A:=\left\|\langle \nabla\rangle^3\,\omega_0\right\|_{\mathcal{G}^{\lambda_0,s}}<\infty,\qquad \int_{\mathbb{T}\times[0,1]}\omega_0(x,y)\,dxdy=0.
\end{equation}
 Let $\omega\in C([0,\infty):H^{10})$ denote the unique smooth solution of the system \eqref{Eur1}. Assume that for some $T>0$ and all $t\in[0,T]$,
 \begin{equation}\label{SupA1}
 \mathrm{supp}\,\omega(t)\subseteq \T\times [\vartheta/2,1-\vartheta/2].
 \end{equation}
   Then, for any $t\in[0,T]$ we have
\begin{equation}\label{ini2}
\begin{split}
\left\|\langle\nabla\rangle^3\,\omega(t)\right\|_{\mathcal{G}^{\lambda(t),s}}&\leq \exp{\Big[C_{\ast}\int_0^t(\|\omega(s)\|_{H^6}+1)ds\Big]}\|\langle\nabla\rangle^3\omega_0\|_{\mathcal{G}^{\lambda_0,s}},\\
\end{split}
\end{equation}
if we choose 
\begin{equation}\label{fdLam1}
\lambda(t):=\lambda_0\exp{\Big\{-C'_{\ast}A\,t\exp\Big[C_\ast\int_0^t(\|\omega(s)\|_{H^6}+1)ds\Big]-C_{\ast}'t\Big\}},
\end{equation}
where $C_\ast=C_\ast(\vartheta)$ and $C'_{\ast}(\vartheta)$ are suitable large constants.
\end{lemma}

In our case, the support assumption (\ref{SupA1}) on $\omega(t)$ is satisfied if $T=2$, as a consequence of the smallness and the support assumptions on $\omega_0$, and the standard local well-posedness theory in Sobolev spaces of the Euler equation \eqref{Eur1}. In fact, as we show below, it is satisfied as part of the bootstrap argument for all $t\in[0,\infty)$.

The conclusions of the lemma can be deduced by following the argument in \cite{Vicol}, see Theorem 6.1 and Remark 6.2.  For the sake of completeness, we provide a self-contained proof  (in our simpler particular case) using the Fourier transforms in appendix \ref{appendix}. The arguments are well known, see also \cite{Foias,Levermore,Vicol2}. An important aspect of the regularity theory for Euler equations in Gevrey spaces is the shrinking in time, at a fast rate, of the radius of convergence (the function $\lambda(t)$ in Lemma \ref{lm:persistenceofhigherregularity}). 

We are now ready to  proceed to the proof of our main theorem.

\begin{proof}[Proof of Theorem \ref{maintheoremINTRO}] For the purpose of proving continuity in time of the energy functionals $\mathcal{E}_g$ and $\mathcal{B}_g$, we make the {\it{a priori} } assumption that $\omega_0\in\mathcal{G}^{1,2/3}$. Indeed, we may replace $\omega_0$ with $\omega_0^n:=\omega_0\ast K_n$, where $K_n\in\mathcal{G}^{1,2/3}$ is an approximation of the identity sequence and $\mathrm{supp}\,K_n\subseteq [-2^{-n},2^{-n}]$ (see  subsection \ref{GevCut} for an axplicit construction of such kernels). Then we prove uniform bounds in $n$ on the solutions generated by the mollified data $\omega_0^n$ , and finally pass to the limit $n\to\infty$ on any finite time interval $[0,T]$.

{\bf{Step 1.}} Given small data $\omega_0$ satisfying \eqref{Eur0} we apply first Lemma \ref{lm:persistenceofhigherregularity}. Therefore $\omega\in C([0,2]:\G^{\lambda_1,2/3})$, $\lambda_1>0$, satisfies the quantitative estimates
\begin{equation}\label{smallnessofomega}
\sup_{t\in[0,2]}\big\|e^{\beta'_0\langle k,\xi\rangle^{1/2}}\widetilde{\omega}(t,k,\xi)\big\|_{L^2_{k,\xi}}\lesssim \eps,
\end{equation}
for some $\beta'_0=\beta'_0(\beta_0,\vartheta_0)>0$. Using also Lemma \ref{ineq6}, and letting $\Psi'\in\mathcal{G}^{1,3/4}$ denote a cutoff function supported in $[\vartheta_0/8,1-\vartheta_0/8]$ and equal to $1$ in $[\vartheta_0/4,1-\vartheta_0/4]$, the localized stream function $\Psi'\psi$ satisfies similar bounds,
\begin{equation}\label{smallnessofpsi}
\sup_{t\in[0,2]}\big\|\langle k,\xi\rangle^2e^{\beta'_0\langle k,\xi\rangle^{1/2}}\widetilde{(\Psi'\psi)}(t,k,\xi)\big\|_{L^2_{k,\xi}}\lesssim \eps.
\end{equation}

Using the formula (see \eqref{rea15})
$$(\partial_yv-1)(t,y)=-\frac{1}{t}\int_0^t\langle\omega\rangle(\tau,y)\,d\tau,$$
and Lemma \ref{lm:Gevrey}, it follows that, for some constant $K_1=K_1(\beta_0,\vartheta_0)$,
\begin{equation}\label{proo1}
|D^\alpha_y v(t,y)|\leq K_1^m(m+1)^{2m},\qquad \partial_yv(t,y)\ge 1/2,
\end{equation}
for any $(t,y)\in[0,2]\times\mathbb{R}$, $m\geq 1$, and $|\alpha|\in[1,m]$. Using now Lemma \ref{GPF} (ii) and letting 
$\mathcal{Y}(t,v)$ denote the inverse of the function $y\to v(t,y)$, we have
\begin{equation}\label{proo2}
|D^\alpha_v \mathcal{Y}(t,v)|\leq K_2^m(m+1)^{2m},
\end{equation}
for any $(t,v)\in[0,2]\times\mathbb{R}$, $m\geq 1$, and $|\alpha|\in[1,m]$. Recall the formulas (see Proposition \ref{ChangedEquations})
\begin{equation*}
\begin{split}
&f(t,z,v)=\omega(t,z+tv,\mathcal{Y}(t,v)),\qquad\phi(t,z,v)=\psi(t,z+tv,\mathcal{Y}(t,v)),\\
&(V'-1)(t,v)=(\partial_yv-1)(t,\mathcal{Y}(t,v)),\qquad \mathcal{H}=1-V'-\langle f\rangle.
\end{split}
\end{equation*}
Using these identities, the bounds \eqref{smallnessofomega}--\eqref{proo2}, and Lemma \ref{GPF} (i), we have
\begin{equation*}
\sup_{t\in[0,2]}\big\|e^{\beta''_0\langle k,\xi\rangle^{1/2}}\widetilde{f}(t,k,\xi)\big\|_{L^2_{k,\xi}}
+\sup_{t\in[0,2]}\big\|e^{\beta''_0\langle k,\xi\rangle^{1/2}}\widetilde{\Theta}(t,k,\xi)\big\|_{L^2_{k,\xi}}\lesssim \eps,
\end{equation*}
for some constant $\beta''_0=\beta''_0(\beta_0,\vartheta_0)>0$. The desired bounds \eqref{boot1} follow if $\delta_0$ is sufficiently small and $\eps_1\approx \eps^{2/3}$, see \eqref{reb11}-\eqref{reb13}.

Assume now that the solution $\omega$ satisfies the bounds in the hypothesis of Proposition \ref{MainBootstrap} on a given interval $[0,T]$, $T\geq 1$. We would like to show that the support of $\omega(t)$ is contained in $\T\times\big[3\vartheta_0/2,1-3\vartheta_0/2\big]$ for any $t\in[0,T]$. Indeed, for this we notice that only transportation in the $y$ direction, given by the term $u^y\,\partial_y\omega$, could enlarge the support of $\omega$ in $y$. Notice that on $\mathbb{T}\times[\vartheta_0,1-\vartheta_0]$,
\begin{equation}\label{uyM}
u^y(t,x,y)=(\partial_x\psi)(t,x,y)=\partial_zP_{\neq0}\big(\Psi\phi\big)(t,x-tv(t,y),v(t,y)).
\end{equation}
Using the bound on $\mathcal{E}_{\Theta}$ from \eqref{boot2}, we can bound, for all $t\in[0,T]$,
\begin{equation}\label{integuy}
\sup_{(x,y)\in \mathbb{T}\times [\vartheta_0,1-\vartheta_0]}\big|u^y(x,y,t)\big|\lesssim \epsilon_1\langle t\rangle^{-2}.
\end{equation}
Since the support of $\omega(0)$ is contained in $\mathbb{T}\times[2\vartheta_0,1-2\vartheta_0]$, we can conclude that ${\rm supp}\,\omega(t)\subseteq \mathbb{T}\times\big[3\vartheta_0/2,1-3\vartheta_0/2\big]$ for any $t\in[0,T]$, 
as long as $\epsilon_1$ is sufficiently small.  

We can now use a simple continuity argument to show that if $\omega_0\in\G^{1,2/3}$ has compact support in $\T\times [2\vartheta_0,1-2\vartheta_0]$ and satisfies the assumptions \eqref{Eur0}, then the solution $\omega$ is in $C([0,\infty):\G^{1,3/5})$, has compact support in $[\vartheta_0,1-\vartheta_0]$ and satisfies $\|\langle\omega\rangle(t)\|_{H^{10}}\lesssim\eps^{2/3}$  for all $t\in[0,\infty)$. Moreover, if we define $f,\Theta,V',\mathcal{H}$ as before then
\begin{equation}\label{proo6}
\sum_{g\in\{f,V'-1,\mathcal{H},\Theta\}}\big[\mathcal{E}_g(t)+\mathcal{B}_g(t)\big]\leq\eps_1^2\qquad\text{ for any }t\in[0,\infty).
\end{equation}
Additionally, $f,\Theta$ satisfy the stronger bounds
\begin{equation}\label{proo6'}
\sum_{g\in\{f,\Theta\}}\big[\mathcal{E}_g(t)+\mathcal{B}_g(t)\big]\lesssim_{\delta}\eps_1^3\qquad\text{ for any }t\in[0,\infty).
\end{equation}
{\bf Step 2.} We turn now to the main conclusions stated in Theorem \ref{maintheoremINTRO}.
By the definition of the weights we have $A_k(t,\xi)\ge e^{\delta_0\langle k,\xi\rangle^{1/2}}$ and $A_R(t,\xi)\ge A_{NR} (t,\xi)\geq e^{\delta_0\langle\xi\rangle^{1/2}}$ for any $(t,\xi,k)\in[0,\infty)\times\R\times\Z$. Using \eqref{proo6} and \eqref{proo6'} it follows that
\begin{equation}\label{uniformf}
\big\|e^{\delta_0\langle k,\xi\rangle^{1/2}}\widetilde{f}(t,k,\xi)\big\|_{L^2_{k,\xi}}+\big\|\mathbf{1}_{k\neq 0}e^{\delta_0\langle k,\xi\rangle^{1/2}}\widetilde{\Theta}(t,k,\xi)\big\|_{L^2_{k,\xi}}+\big\|e^{\delta_0\langle\xi\rangle^{1/2}}\widetilde{(V'-1)}(t,\xi)\big\|_{L^2_{\xi}}\lesssim \epsilon_1,
\end{equation}
and
\begin{equation}\label{uniformf'}
\big\|e^{\delta_0\langle k,\xi\rangle^{1/2}}\widetilde{f}(t,k,\xi)\big\|_{L^2_{k,\xi}}+\big\|\mathbf{1}_{k\neq 0}e^{\delta_0\langle k,\xi\rangle^{1/2}}\widetilde{\Theta}(t,k,\xi)\big\|_{L^2_{k,\xi}}\lesssim_{\delta} \epsilon_1^{3/2}. 
\end{equation}

We show first that if $t\in[0,\infty)$ then
 \begin{equation}\label{DdtvM}
\big\|e^{\delta_1\langle \xi\rangle^{1/2}}\widetilde{\partial_tv}(t,\xi)\big\|_{L^2_\xi}\lesssim\frac{\epsilon_1^2}{\langle t\rangle^2},
\end{equation}
for some $\delta_1=\delta_1(\beta_0,\vartheta_0)>0$. We start from \eqref{rea15}, so
\begin{equation}\label{proo7}
\partial_tv(t,y)=\frac{1}{t}\left[\,-\frac{1}{t}\int_0^t\langle u^x\rangle(\tau,y)\,d\tau+\langle u^x\rangle(t,y)\right]=\frac{1}{t^2}\int_0^t\int_{\tau}^t\partial_s\langle u^x\rangle(s,y)\,ds \,d\tau.
\end{equation}
By the perturbed Euler equation (\ref{perEu}) for $u^x$, we have
\begin{equation}\label{decayofuxDt}
\partial_t\langle u^x\rangle(t,y)+\langle u^y\partial_y u^x\rangle(t,y)=0.
\end{equation}
Using $u^x=-\partial_y\psi$, $u^y=\partial_x\psi$ we see that 
\begin{equation}\label{decayofuxVorticity}
\big\langle u^y\partial_y u^x\big\rangle=-\big\langle\partial_x\psi\,\partial_y^2\psi\big\rangle.
\end{equation}
Moreover $\partial_t\langle u^x\rangle(t,y)=0$ if $y\in[0,\vartheta_0]\cup[1-\vartheta_0,1]$, see \eqref{rea9}. Therefore we can choose an appropriate cutoff function $\Xi(v)\in\mathcal{G}^{1,2/3}$ supported in $[c_0+2\vartheta_0/3,c_0+1-2\vartheta_0/3]$ and equal to $1$ in $[c_0+\vartheta_0,c_0+1-\vartheta_0]$,   so that equation \eqref{decayofuxDt} can be rewritten as 
\begin{equation}\label{dtduxMain}
\partial_t\langle u^x\rangle(t,y)=\Xi(v(t,y))\big\langle\partial_x\psi\,\partial_y^2\psi\big\rangle=F(t,v(t,y)),
\end{equation}
where (see \eqref{rea12}--\eqref{rea12.1})
\begin{equation}\label{proo10}
\begin{split}
F(t,v):=&\,\Xi(v)\,|V'|^2\big\langle\partial_z(\Psi\,\phi)\,(\partial_v-t\partial_z)^2P_{\neq0}(\Psi\,\phi)\big\rangle(t,v)\\
&+\,\Xi(v)\,V''\big\langle\partial_z(\Psi\phi)\,(\partial_v-t\partial_z)P_{\neq0}(\Psi\phi)\big\rangle(t,v).
\end{split}
\end{equation}
The key point is that on the right hand side of \eqref{proo10}
the term with $t^2$ coefficient vanishes
$$t^2\big\langle\partial_z(\Psi\,\phi)\,\partial_z^2P_{\neq0}(\Psi\,\phi)\big\rangle=0.$$
Hence $\partial_t\langle u^x\rangle(t,y)$ decays with rate $\langle t\rangle^{-3}$. More precisely, using \eqref{uniformf} and Lemma \ref{Multi0} we have
\begin{equation}\label{eq:decaydtdvMain}
\big\|e^{(\delta_0/2)\langle\xi\rangle^{1/2}}\widetilde{F}(t,\xi)\big\|_{L^2_{\xi}}\lesssim \epsilon_1^{2}\langle t\rangle^{-3}.
\end{equation} 

Notice that $\partial_v\mathcal{Y}(t,v)=(1/V')(t,v)$, where $\mathcal{Y}(t,.)$ the inverse of the function $y\to v(t,y)$. Using \eqref{uniformf} and Lemmas \ref{lm:Gevrey} and \ref{GPF}, we have, for some constant $K_3=K_3(\beta_0,\vartheta_0)$,
\begin{equation}\label{proo12}
|D^\alpha_v \mathcal{Y}(t,v)|\leq K_3^m(m+1)^{2m},\qquad |D^\alpha_y v(t,y)|\leq K_3^m(m+1)^{2m},
\end{equation}
for all $m\geq 1$ and $|\alpha|\in[1,m]$. Using again Lemma \ref{GPF} and \eqref{dtduxMain}--\eqref{eq:decaydtdvMain}, we have
\begin{equation}\label{eq:cubicdecaydtu}
\big\|e^{\delta_1\langle\xi\rangle^{1/2}}\widetilde{\langle\partial_tu^x\rangle}(t,\xi)\big\|_{L^2_\xi}\lesssim \epsilon_1^2\langle t\rangle^{-3},
\end{equation}
for some $\delta_1=\delta_1(\beta_0,\vartheta_0)>0$.
From \eqref{proo7}, we get that 
\begin{equation*}
\begin{split}
\left\|e^{\delta_1\langle\xi\rangle^{1/2}}\,\widetilde{\partial_tv}(t,\xi)\right\|_{L^2_\xi}&\lesssim\frac{1}{\langle t\rangle^2}\int_0^t\int_{\tau}^t\left\|e^{\delta_1\langle\xi\rangle^{1/2}}\,\widetilde{\langle\partial_tu^x\rangle}(s,\xi)\right\|_{L^2_\xi}\,ds\,d\tau\lesssim \frac{\epsilon_1^2}{\langle t\rangle^2},
\end{split}
\end{equation*}
which gives \eqref{DdtvM}. Consequently $v_{\infty}(y):=\lim_{t\to\infty}v(t,y)$ exists in $ \mathcal{G}^{\delta_1,1/2}$, and we have
\begin{equation}\label{convergencedvty}
\left\|e^{\delta_1\langle\xi\rangle^{1/2}}\big[\widetilde{v}(t,\xi)-\widetilde{v_{\infty}}(\xi)\big]\right\|_{L^2_\xi}\lesssim \epsilon_1^2\langle t\rangle^{-1}.
\end{equation}

{\bf Step 3.} We prove now convergence of the profile $f$. Using (\ref{rea23.1}) and ${\rm supp}\,f\subseteq[c_0+\vartheta_0,c_0+1-\vartheta_0]$, we have
\begin{equation}\label{LfM}
\partial_tf-V'\partial_vP_{\neq 0}(\Psi\phi)\,\partial_zf+\dot{V}\,\partial_vf+V'\partial_z(\Psi\phi)\,\partial_vf=0.
\end{equation}
Using the bounds \eqref{proo6} on $\mathcal{E}_{\Theta}$, we get that
\begin{equation}\label{QDPhi}
\big\|\mathbf{1}_{k\neq0}\,e^{\delta_0\langle k,\xi\rangle^{1/2}}\widetilde{\Psi\phi}(t,k,\xi)\big\|_{L^2_{k,\xi}}\lesssim\frac{\epsilon_1}{\langle t\rangle^2}.
\end{equation}
Since $\dot{V}(t,v)=\partial_tv(t,y)$, we can use (\ref{DdtvM}),  \eqref{proo12}, and Lemma \ref{GPF} to conclude that
\begin{equation}\label{DdVM}
\left\|e^{\delta_1'\langle \xi\rangle^{1/2}}\,\widetilde{\dot{V}}(t,\xi)\right\|_{L^2_\xi}\lesssim\frac{\epsilon_1}{\langle t\rangle^2},
\end{equation}
for some $\delta_1'=\delta'_1(\beta_0,\vartheta_0)>0$.
Using (\ref{LfM})-(\ref{DdVM}),  and the bounds (\ref{uniformf}), we have
\begin{equation}\label{QDtf}
\big\|e^{\delta_2\langle k,\xi\rangle^{1/2}}\widetilde{\partial_tf}(t,k,\xi)\big\|_{L^2_{k,\xi}}\lesssim \frac{\epsilon_1^2}{\langle t\rangle^2},
\end{equation}
for some $\delta_2=\delta_2(\beta_0,\vartheta_0)>0$. In particular $f(t,z,v)$ converges to $f_{\infty}(z,v)$ in $\mathcal{G}^{\delta_2,1/2}$, with
\begin{equation}\label{convergencef0}
\left\|e^{\delta_2\langle k,\xi\rangle^{1/2}}\big[\,\widetilde{f}(t,k,\xi)-\widetilde{f_{\infty}}(k,\xi)\,\big]\right\|_{L^2_{k,\xi}}\lesssim\frac{\epsilon_1^2}{\langle t\rangle}.
\end{equation}

Therefore, using (\ref{convergencedvty}), (\ref{convergencef0}), and Lemma \ref{GPF},  we have
$$\omega(t,x+tv(t,y),y)=f(t,x,v(t,y))$$
converges to $f_{\infty}(x,v_{\infty}(y))$ with
\begin{equation}\label{Cf1}
\left\|e^{\delta_2'\langle k,\xi\rangle^{1/2}}\big[\,\mathcal{F}(\omega(t,x+tv(t,y),y))(t,k,\xi)-\mathcal{F}(f_{\infty}(x,v_{\infty}(y)))(k,\xi)\,\big]\right\|_{L^2_{k,\xi}}\lesssim \frac{\epsilon_1^2}{\langle t\rangle}.
\end{equation}

{\bf Step 4.} We are now ready to prove the bounds (\ref{convergence})-(\ref{convergenceuy}). Let $F_{\infty}(x,y):=f_{\infty}(x,v_{\infty}(y))$. The bounds (\ref{convergence}) follow from (\ref{Cf1}) and the definitions of $v(t,y)$ and $\Phi(t,y)$. Moreover, let
$$u_{\infty}(y):=\lim_{t\to\infty}\langle u^x\rangle(t,y).$$
The existence of the limit in $\G^{\delta_1,1/2}$ follows from (\ref{eq:cubicdecaydtu}), and the bounds (\ref{convergenceofux}) and (\ref{AsymPhi}) follow from definitions. To identify $u_\infty$, according to \eqref{AsymPhi2}, from the Biot-Savart law and the equation for stream function $\psi$, we have
$\langle u^x\rangle=-\langle\partial_y\psi\rangle$
and
$$\partial^2_{y}\langle\psi\rangle(t,y)=\langle\omega\rangle(t,y)=\langle f\rangle(t,v(t,y))$$
with $\langle\psi\rangle|_{y=0,\,1}=0$. The desired identity \eqref{AsymPhi2} follows from the convergence of $f(t,x,v(t,y))$ and $v(t,y)$, see  (\ref{convergencef0}) and (\ref{convergencedvty}).

To prove the decay estimates \eqref{convergencetomean} and \eqref{convergenceuy} for $u^x-\langle u^x\rangle$ and $u^y$ we use properties of the stream function $\psi$. The starting point is the equation
\begin{equation*}
\Delta\psi(t,x,y)=\omega(t,x,y)=f(t,x-tv(t,y),v(t,y)),\qquad \psi(x,0)=\psi(x,1)=0
\end{equation*}
for $(x,y)\in \mathbb{T}\times[0,1]$. Taking partial Fourier transforms in $x$, we get
\begin{equation}\label{formulapsikMain}
\psi^\ast(t,k,y)=-\int_0^1G_k(y,z)\,f^\ast(t,k,v(t,z))\,e^{-iktv(t,z)}\,dz,
\end{equation}
and
\begin{equation}\label{formulapsikMain'}
(\partial_y\psi^\ast)(t,k,y)=-\int_0^1\partial_yG_k(y,z)\,f^\ast(t,k,v(t,z))\,e^{-iktv(t,z)}\,dz.
\end{equation}
See Lemma \ref{ineq6} for such identities and for the formulas of the Green functions $G_k$. Moreover
$$\big|u^y(t,x,y)\big|\lesssim\sum_{k}|k|\big|\psi^\ast(t,k,y)\big|\lesssim \sup_{k\neq0}|k|^3\big|\psi^\ast(t,k,y)\big|$$
and
$$\big|u^x(t,x,y)-\langle u^x\rangle(t,y)\big|\lesssim\sum_{k\neq0}\big|\partial_y\psi^\ast(t,k,y)\big|\lesssim \sup_{k\neq0}|k|^2\big|(\partial_y\psi^\ast)(t,k,y)\big|.$$
We can now integrate by parts in $z$ in the identities (\ref{formulapsikMain}) (twice) and (\ref{formulapsikMain'}) (once), and use the formulas \eqref{ini15.1} and the smoothness of the functions $f$ and $v$. Thus
\begin{equation}\label{pduy}
\big|u^y(t,x,y)\big|\lesssim_{\delta}\frac{\epsilon_1^{3/2}}{\langle t\rangle^2},\qquad \big|u^x(t,x,y)-\langle u^x\rangle(t,y)\big|\lesssim_{\delta}\frac{\epsilon_1^{3/2}}{\langle t\rangle}.
\end{equation}
The desired bounds (\ref{convergencetomean}) and (\ref{convergenceuy}) follow, which completes the proof of Theorem \ref{maintheoremINTRO}.
\end{proof}

\section{Improved control of the normalized vorticity $f$}\label{fimprov}

We prove first the main bounds \eqref{boot3} for the function $f$. More precisely:

\begin{proposition}\label{BootImp1}
With the definitions and assumptions in Proposition \ref{MainBootstrap}, we have
\begin{equation}\label{nar1}
\mathcal{E}_f(t)+\mathcal{B}_f(t)\lesssim_\delta\eps_1^3\qquad\text{ for any }t\in[1,T].
\end{equation}
\end{proposition}

The rest of the section is concerned with the proof of this proposition. Recall the definitions \eqref{rec1}--\eqref{rec6.5}. We calculate
\begin{equation}\label{nar1.5}
\begin{split}
\frac{d}{dt}\E_f(t)=&\sum_{k\in \mathbb{Z}}\int_\R 2\dot{A}_k(t,\xi)A_k(t,\xi)\big|\widetilde{f}(t,k,\xi)\big|^2\,d\xi\\
&+2\Re\sum_{k\in \mathbb{Z}}\int_{\R}A_k^2(t,\xi)\partial_t\widetilde{f}(t,k,\xi)\overline{\widetilde{f}(t,k,\xi)}\,d\xi.
\end{split}
\end{equation}
Therefore, since $\partial_tA_k\leq 0$, for any $t\in[1,T]$ we have
\begin{equation*}
\begin{split}
&\E_f(t)+\int_1^t\sum_{k\in \mathbb{Z}}\int_\R 2|\dot{A}_k(s,\xi)|A_k(s,\xi)\big|\widetilde{f}(s,k,\xi)\big|^2\,d\xi ds\\
&=\E_f(1)+\int_1^t\Big\{2\Re\sum_{k\in \mathbb{Z}}\int_{\R}A_k^2(s,\xi)\partial_s\widetilde{f}(s,k,\xi)\overline{\widetilde{f}(s,k,\xi)}\,d\xi\Big\}ds.
\end{split}
\end{equation*}
Since $\E_f(1)\lesssim\eps_1^3$ (see \eqref{boot1}), for \eqref{nar1} it suffices to prove that, for any $t\in[0,T]$,
\begin{equation}\label{nar2}
\Big|2\Re\int_1^t\sum_{k\in \mathbb{Z}}\int_{\R}A_k^2(s,\xi)\partial_s\widetilde{f}(s,k,\xi)\overline{\widetilde{f}(s,k,\xi)}\,d\xi ds\Big|\lesssim_\delta \eps_1^3.
\end{equation}

Before we proceed with the proof of \eqref{nar2}, we record several bounds on the functions $\mathcal{H}$, $\dot{V}$, $V'-1$, and $V''=\partial_v(V'-1)+(V'-1)\partial_v(V'-1)$. In sections 4-6 we will often use bounds on the main weights proved in sections \ref{weights} and \ref{BilinWeights}.

\begin{lemma}\label{nar8}
(i) For any $t\in[1,T]$ and $F\in\{V'-1,(V'-1)^2,\langle\partial_v\rangle^{-1}V''\}$ we have
\begin{equation}\label{nar4}
\begin{split}
&\int_{\R}A_R^2(t,\xi)\big|\widetilde{F}(t,\xi)\big|^2\,d\xi\lesssim_\delta\eps_1^2,\\
&\int_1^t\int_{\R}|\dot{A}_R(s,\xi)|A_R(s,\xi)\big|\widetilde{F}(s,\xi)\big|^2\,d\xi ds\lesssim_\delta\eps_1^2.
\end{split}
\end{equation}

(ii) Moreover, for any $t\in[1,T]$,
\begin{equation}\label{nar6}
\begin{split}
&\int_{\R}A_{NR}^2(t,\xi)\big(1+\langle\xi\rangle^{-3/2}\langle t\rangle^{3/2}\big)\big|\widetilde{\mathcal{H}}(t,\xi)\big|^2\,d\xi\lesssim_\delta\eps_1^2,\\
&\int_1^t\int_{\R}|\dot{A}_{NR}(s,\xi)|A_{NR}(s,\xi)\big(1+\langle\xi\rangle^{-3/2}\langle s\rangle^{3/2}\big)\big|\widetilde{\mathcal{H}}(s,\xi)\big|^2\,d\xi ds\lesssim_\delta\eps_1^2.
\end{split}
\end{equation}
and
\begin{equation}\label{nar7}
\begin{split}
&\int_{\R}A_{NR}^2(t,\xi)\big(\langle\xi\rangle^{2}\langle t\rangle^2+\langle\xi\rangle^{1/2}\langle t\rangle^{7/2}\big)\big|\widetilde{\dot{V}}(t,\xi)\big|^2\,d\xi\lesssim_\delta\eps_1^2,\\
&\int_1^t\int_{\R}|\dot{A}_{NR}(s,\xi)|A_{NR}(s,\xi)\big(\langle\xi\rangle^{2}\langle t\rangle^2+\langle\xi\rangle^{1/2}\langle t\rangle^{7/2}\big)\big|\widetilde{\dot{V}}(s,\xi)\big|^2\,d\xi ds\lesssim_\delta\eps_1^2.
\end{split}
\end{equation}
\end{lemma}

\begin{proof}
(i) The bounds \eqref{nar4} follow from the bootstrap assumption $\E_{V'-1}+\mathcal{B}_{V'-1}\leq\eps_1^2$ and the bilinear estimates in Lemma \ref{Multi0} (i) and Lemma \ref{lm:Multi}. 

(ii) Recall that $A_{NR}(t,\xi)\leq \min(A_R(t,\xi),A_0(t,\xi))$ (see \eqref{reb13}), and the bounds \eqref{vfc30.5}. The estimates \eqref{nar6} follow from the identity $\mathcal{H}=1-V'-\langle f\rangle$ and the bootstrap assumptions on $V'-1$ and $f$ in \eqref{boot2}. 

To prove the bounds \eqref{nar7} we examine \eqref{rea23}, thus $\mathcal{H}/t=V'\partial_v{\dot{V}}$. Therefore
\begin{equation*}
\partial_v{\dot{V}}=\sum_{n\geq 0}(-1)^n(V'-1)^n\cdot (\mathcal{H}/t).
\end{equation*}
The functions $(V'-1)^n$ satisfy bounds similar to \eqref{nar4}, with additional decaying $2^{-n}$ factors in the right-hand side, while $\mathcal{H}$ satisfies \eqref{nar6}. The small frequencies $|\xi|\ll 1$ of $\dot{V}$ can be controlled by the uncertainty principle, due to the compact support in $v$ of $\dot{V}$ (write $\dot{V}(t,v)=\dot{V}(t,v)\cdot \Psi(v)$ and estimate the low frequencies in $L^\infty$). The desired bounds \eqref{nar7} follow using again Lemma \ref{Multi0} (i) and Lemma \ref{lm:Multi}.
\end{proof}

We examine now the space-time integrals in the left-hand side of \eqref{nar2}, and use the identity \eqref{rea23.1}. Therefore, recalling the support property of $f$,
\begin{equation}\label{nar8.1}
\begin{split}
&\partial_sf=\mathcal{N}_1+\mathcal{N}_2+\mathcal{N}_3,\\
&\mathcal{N}_1:=V'\partial_vP_{\neq 0}(\Psi\phi)\,\partial_zf,\qquad \mathcal{N}_2:=-V'\partial_z(\Psi\phi)\,\partial_vf,\qquad \mathcal{N}_3:=-\dot{V}\,\partial_vf.
\end{split}
\end{equation}
We bound the contributions of the terms $\mathcal{N}_1$, $\mathcal{N}_2$ and $\mathcal{N}_3$ in the next three subsections. 

We will sometimes use the following elementary lemma:

\begin{lemma}\label{veryelem}
Let $a,\,b\in \mathbb{R}^d$ with $d\ge 1$, $\beta\in [0,1]$. Then
\begin{equation}\label{<a>to|b|}
 \,\,\langle a\rangle \ge \beta\langle b\rangle \qquad{\rm implies}\qquad |a|\ge \beta|b|-1,
\end{equation}
\begin{equation}\label{|a|to<b>}
|a|\ge\beta|b|\qquad{\rm implies}\qquad \langle a\rangle\ge \beta\langle b\rangle,
\end{equation}
and
\begin{equation}\label{b>a}
    \langle b\rangle\ge \beta \langle a-b\rangle \qquad{\rm implies} \qquad \langle a\rangle^{1/2} \leq \langle b\rangle ^{1/2}+\big(1-\sqrt{\beta}/2\big)\langle a-b\rangle^{1/2}.
    \end{equation} 

\end{lemma}

\begin{proof} The estimates \eqref{<a>to|b|} and \eqref{|a|to<b>} follow from definitions. For \eqref{b>a} it suffices to prove that
\begin{equation}\label{a>b2}
\langle b\rangle+(2-\sqrt{\beta})\langle b\rangle^{1/2}\langle a-b\rangle^{1/2}+(1-\sqrt{\beta}/2)^2\langle a-b\rangle\ge \langle b\rangle+\langle a-b\rangle.
\end{equation}
Simplifying (\ref{a>b2}) and then using $ \langle b\rangle\ge \beta \langle a-b\rangle$, we reduce to prove
\begin{equation*}
\sqrt{\beta}\,(2-\sqrt{\beta})+(1-\sqrt{\beta}/2)^2\ge 1.
\end{equation*}
This follows from a simple calculation and the assumption $\beta\in[0,1]$.
\end{proof}

\subsection{The nonlinearity $\mathcal{N}_1$} We prove first the following lemma:

\begin{lemma}\label{nar10}
With $\mathcal{N}_1$ defined as above, for any $t\in[1,T]$ we have
\begin{equation}\label{nar11}
\Big|2\Re\int_1^t\sum_{k\in \mathbb{Z}}\int_{\R}A_k^2(s,\xi)\widetilde{\mathcal{N}_1}(s,k,\xi)\overline{\widetilde{f}(s,k,\xi)}\,d\xi ds\Big|\lesssim_\delta \eps_1^3.
\end{equation}
\end{lemma}

The rest of this subsection is concerned with the proof of this lemma. Let
\begin{equation}\label{nar12}
H_1:=\partial_vP_{\neq 0}(\Psi\phi),\qquad H_2:=V'\partial_vP_{\neq 0}(\Psi\phi).
\end{equation}

\begin{lemma}\label{nar13}
For any $t\in[1,T]$ and $a\in\{1,2\}$ we have
\begin{equation}\label{nar14}
\begin{split}
&\sum_{k\in \mathbb{Z}\setminus\{0\}}\int_{\R}A_k^2(t,\xi)\frac{\langle t\rangle^2}{|\xi/k|^2+\langle t\rangle^2}\frac{\langle t-\xi/k\rangle^4}{(\langle\xi\rangle/k^2)^2}\big|\widetilde{H_a}(t,k,\xi)\big|^2\,d\xi\lesssim_\delta\eps_1^2,\\
&\int_1^t\sum_{k\in \mathbb{Z}\setminus\{0\}}\int_{\R}|\dot{A}_k(s,\xi)|A_k(s,\xi)\frac{\langle s\rangle^2}{|\xi/k|^2+\langle s\rangle^2}\frac{\langle s-\xi/k\rangle^4}{(\langle\xi\rangle/k^2)^2}\big|\widetilde{H_a}(s,k,\xi)\big|^2\,d\xi ds\lesssim_\delta\eps_1^2.
\end{split}
\end{equation}
\end{lemma}

\begin{proof} The bounds on $H_1$ follow directly from the bootstrap assumptions on $\mathcal{E}_\Theta$ and $\mathcal{B}_\Theta$, and the definitions \eqref{defgellip}. 

Notice that $H_2=H_1+(V'-1)H_1$.  We use Lemma \ref{Multi0} (ii) to prove the bounds \eqref{nar14} for $a=2$. In view of \eqref{nar4} and \eqref{nar14} (with $a=1$), it suffices to prove the multiplier bounds
\begin{equation}\label{nar15}
\begin{split}
A_k(t,\xi)&\frac{\langle t\rangle}{|\xi/k|+\langle t\rangle}\frac{\langle t-\xi/k\rangle^2}{\langle\xi\rangle/k^2}\\
&\lesssim_\delta A_R(t,\xi-\eta)\cdot A_k(t,\eta)\frac{\langle t\rangle}{|\eta/k|+\langle t\rangle}\frac{\langle t-\eta/k\rangle^2}{\langle\eta\rangle/k^2}\cdot \{\langle\xi-\eta\rangle^{-2}+\langle k,\eta\rangle^{-2}\}
\end{split}
\end{equation}
and
\begin{equation}\label{nar16}
\begin{split}
\big|(\dot{A}_k A_k)&(t,\xi)\big|^{1/2}\frac{\langle t\rangle}{|\xi/k|+\langle t\rangle}\frac{\langle t-\xi/k\rangle^2}{\langle\xi\rangle/k^2}\lesssim_\delta \left[\big|(\dot{A}_R/A_R)(t,\xi-\eta)\big|^{1/2}+\big|(\dot{A}_k/A_k)(t,\eta)\big|^{1/2}\right]\\
&\times A_R(t,\xi-\eta)\cdot A_k(t,\eta)\frac{\langle t\rangle}{|\eta/k|+\langle t\rangle}\frac{\langle t-\eta/k\rangle^2}{\langle\eta\rangle/k^2}\cdot \{\langle\xi-\eta\rangle^{-2}+\langle k,\eta\rangle^{-2}\},
\end{split}
\end{equation}
for any $t\in[1,T]$, $\xi,\eta\in\R$, and $k\in\mathbb{Z}\setminus\{0\}$.

To prove \eqref{nar15}--\eqref{nar16} we use Lemma \ref{TLX40}. In addition, by considering the cases $|\xi-\eta|\leq 10|k,\eta|$ and $|\xi-\eta|\geq 10|k,\eta|$, it is easy to see that
\begin{equation}\label{nar17}
\frac{\langle t\rangle}{|\xi/k|+\langle t\rangle}\frac{\langle t-\xi/k\rangle^2}{\langle\xi\rangle/k^2}\lesssim_\delta \frac{\langle t\rangle}{|\eta/k|+\langle t\rangle}\frac{\langle t-\eta/k\rangle^2}{\langle\eta\rangle/k^2}\cdot e^{\delta\min(\langle\xi-\eta\rangle,\langle k,\eta\rangle)^{1/2}}
\end{equation}
for any $t\in[1,T]$, $\xi,\eta\in\R$, and $k\in\mathbb{Z}\setminus\{0\}$. The bounds \eqref{nar15} follow from \eqref{TLX7} and \eqref{nar17}, while the bounds \eqref{nar16} follow from \eqref{TLX7}--\eqref{DtVMulti} and \eqref{nar17}.
\end{proof}

We turn now to the proof of \eqref{nar11}. We write
\begin{equation}\label{exN1}
\begin{split}
&\Big|2\Re\int_1^t\sum_{k\in \mathbb{Z}}\int_{\R}A_k^2(s,\xi)\widetilde{\mathcal{N}_1}(s,k,\xi)\overline{\widetilde{f}(s,k,\xi)}\,d\xi ds\Big|\\
&=C\Big|2\Re\Big\{\sum_{k,\ell\in \mathbb{Z}}\int_1^t\int_{\R^2}A_k^2(s,\xi)\widetilde{H_2}(s,k-\ell,\xi-\eta)i\ell\widetilde{f}(s,\ell,\eta)\overline{\widetilde{f}(s,k,\xi)}\,d\xi d\eta ds\Big\}\Big|\\
&=C\Big|\int_1^t\sum_{k,\ell\in \mathbb{Z}}\int_{\R^2}\big[\ell A_k^2(s,\xi)-k A_\ell^2(s,\eta)\big]\widetilde{H_2}(s,k-\ell,\xi-\eta)\widetilde{f}(s,\ell,\eta)\overline{\widetilde{f}(s,k,\xi)}\,d\xi d\eta ds\Big|,
\end{split}
\end{equation}
where the second identity is proved by symmetrization (recall that $H_2$ is real-valued). 

We define the sets
\begin{equation}\label{nar18.1}
\begin{split}
R_0:=\Big\{&((k,\xi),(\ell,\eta))\in (\Z\times \R)^2:\\
&\min(\langle k,\xi\rangle,\,\langle\ell,\eta\rangle,\,\langle k-\ell,\xi-\eta\rangle)\geq \frac{\langle k,\xi\rangle+\langle\ell,\eta\rangle+\langle k-\ell,\xi-\eta\rangle}{20}\Big\},\\
\end{split}
\end{equation}
\begin{equation}\label{nar18.2}
R_1:=\Big\{((k,\xi),(\ell,\eta))\in (\Z\times \R)^2:\,\langle k-\ell,\xi-\eta\rangle\leq \frac{\langle k,\xi\rangle+\langle\ell,\eta\rangle+\langle k-\ell,\xi-\eta\rangle}{10}\Big\},
\end{equation}
\begin{equation}\label{nar18.3}
R_2:=\Big\{((k,\xi),(\ell,\eta))\in (\Z\times \R)^2:\,\langle\ell,\eta\rangle\leq \frac{\langle k,\xi\rangle+\langle\ell,\eta\rangle+\langle k-\ell,\xi-\eta\rangle}{10}\Big\},
\end{equation}
\begin{equation}\label{nar18.4}
R_3:=\Big\{((k,\xi),(\ell,\eta))\in (\Z\times \R)^2:\,\langle k,\xi\rangle\leq \frac{\langle k,\xi\rangle+\langle\ell,\eta\rangle+\langle k-\ell,\xi-\eta\rangle}{10}\Big\}.
\end{equation}
Then we define the corresponding integrals
\begin{equation}\label{nar19}
\begin{split}
\mathcal{U}_n:=\int_1^t\sum_{k,\ell\in \mathbb{Z}}\int_{\R^2}&\mathbf{1}_{R_n}((k,\xi),(\ell,\eta))\big|\ell A_k^2(s,\xi)-k A_\ell^2(s,\eta)\big|\,|\widetilde{H_2}(s,k-\ell,\xi-\eta)|\\
&\times|\widetilde{f}(s,\ell,\eta)|\,|\widetilde{f}(s,k,\xi)|\,d\xi d\eta ds.
\end{split}
\end{equation}
For $n=0,1$, we use (i) of Lemma \ref{TLXH1}. We remark that $\widetilde{H_a}(t,0,\cdot)\equiv0$ for $a\in\{1,2\}$. Denote $(\sigma,\rho)=(k-\ell,\xi-\eta)$. Using also Lemma \ref{nar13} and \eqref{boot2} we can bound
\begin{equation*}
\begin{split}
\mathcal{U}_n&\lesssim_{\delta}\int_1^t\sum_{k,\ell\in \mathbb{Z}}\int_{\R^2}\sqrt{|(A_k\dot{A}_k)(s,\xi)|}\,\big|\widetilde{f}(s,k,\xi)\big|\sqrt{|(A_{\ell}\dot{A}_{\ell})(s,\eta)|}\,\big|\widetilde{f}(s,\ell,\eta)\big|\frac{\langle s\rangle}{|\rho/\sigma|+\langle s\rangle}\\
&\qquad\times\, \mathbf{1}_{\sigma\neq0}\frac{\langle s-\rho/\sigma\rangle^2}{\langle \rho\rangle/\sigma^2}A_{\sigma}(s,\rho)\big|\widetilde{H_2}(s,\sigma,\rho)\big|e^{-(\delta_0/200)\langle \sigma,\rho\rangle^{1/2}}\,d\xi d\eta ds\\
&\lesssim_{\delta} \Big\|\sqrt{|(A_k\dot{A}_k)(s,\xi)|}\,\widetilde{f}(s,k,\xi)\Big\|_{L^2_{s}L^2_{k,\xi}}\Big\|\sqrt{|(A_{\ell}\dot{A}_{\ell})(s,\eta)|}\,\widetilde{f}(s,\ell,\eta)\Big\|_{L^2_sL^2_{\ell,\eta}}\\
&\qquad\times \Big\|\mathbf{1}_{\sigma\neq0}A_{\sigma}(s,\rho)\frac{\langle s\rangle}{|\rho/\sigma|+\langle s\rangle}\frac{\langle s-\rho/\sigma\rangle^2}{\langle\rho\rangle/\sigma^2}e^{-(\delta_0/300)\langle \sigma,\rho\rangle^{1/2}}\widetilde{H_2}(s,\sigma,\rho)\Big\|_{L^{\infty}_sL^2_{\sigma,\rho}}\\
&\lesssim_{\delta}\epsilon_1^3.
\end{split}
\end{equation*}
Similarly, for $n=2$ we use (ii) of Lemma \ref{TLXH1}. Using also Lemma \ref{nar13} and \eqref{boot2} we can bound
\begin{equation*}
\begin{split}
\mathcal{U}_2&\lesssim_{\delta}\int_1^t\sum_{k,\ell\in \mathbb{Z}}\int_{\R^2}\mathbf{1}_{\sigma\neq0}\sqrt{|(A_{\sigma}\dot{A}_{\sigma})(s,\rho)|}\,\frac{\langle s\rangle}{|\rho/\sigma|+\langle s\rangle}\frac{\langle s-\rho/\sigma\rangle^2}{\langle \rho\rangle/\sigma^2}\big|\widetilde{H_2}(s,\sigma,\rho)\big|\\
&\qquad\times\sqrt{|(A_k\dot{A}_k)(s,\xi)|}\,\big|\widetilde{f}(s,k,\xi)\big|A_{\ell}(s,\eta) e^{-(\delta_0/200)\langle \ell,\eta\rangle^{1/2}}|\widetilde{f}(s,\ell,\eta)|\,d\xi d\eta ds\\
&\lesssim_{\delta} \Big\|\sqrt{|(A_k\dot{A}_k)(s,\xi)|}\,\widetilde{f}(s,k,\xi)\Big\|_{L^2_{s}L^2_{k,\xi}}\Big\|A_{\ell}(s,\eta)\,e^{-(\delta_0/300)\langle \ell,\eta\rangle^{1/2}}\widetilde{f}(s,\ell,\eta)\Big\|_{L^{\infty}_sL^2_{\ell,\eta}}\\
&\qquad\times \Big\|\mathbf{1}_{\sigma\neq0}\sqrt{|(A_{\sigma}\dot{A}_{\sigma})(s,\rho)|}\,\frac{\langle s\rangle}{|\rho/\sigma|+\langle s\rangle}\frac{\langle s-\rho/\sigma\rangle^2}{\langle\rho\rangle/\sigma^2}\widetilde{H_2}(s,\sigma,\rho)\Big\|_{L^{2}_sL^2_{\sigma,\rho}}\\
&\lesssim_{\delta}\epsilon_1^3.
\end{split}
\end{equation*}
The case $n=3$ is identical to the case $n=2$, by symmetry. Thus $\mathcal{U}_n\lesssim_\delta\epsilon_1^3$ for all $n\in\{0,1,2,3\}$, and the desired bounds \eqref{nar11} follow.

\subsection{The nonlinearity $\mathcal{N}_2$} We prove now the following:

\begin{lemma}\label{nar30}
With $\mathcal{N}_2$ defined as \eqref{nar8.1}, for any $t\in[1,T]$ we have
\begin{equation}\label{nar31}
\Big|2\Re\int_1^t\sum_{k\in \mathbb{Z}}\int_{\R}A_k^2(s,\xi)\widetilde{\mathcal{N}_2}(s,k,\xi)\overline{\widetilde{f}(s,k,\xi)}\,d\xi ds\Big|\lesssim_\delta \eps_1^3.
\end{equation}
\end{lemma}

The rest of this subsection is concerned with the proof of this lemma. Let
\begin{equation}\label{nar32}
H_3:=\partial_zP_{\neq 0}(\Psi\phi),\qquad H_4:=V'\partial_zP_{\neq 0}(\Psi\phi).
\end{equation}

\begin{lemma}\label{nar33}
For any $t\in[1,T]$ and $a\in\{3,4\}$ we have
\begin{equation}\label{nar34}
\begin{split}
&\sum_{k\in \mathbb{Z}\setminus\{0\}}\int_{\R}A_k^2(t,\xi)\frac{k^2\langle t\rangle^4\langle t-\xi/k\rangle^4}{(|\xi/k|^2+\langle t\rangle^2)^2}\big|\widetilde{H_a}(t,k,\xi)\big|^2\,d\xi\lesssim_\delta\eps_1^2\\
&\int_1^t\sum_{k\in \mathbb{Z}\setminus\{0\}}\int_{\R}|\dot{A}_k(s,\xi)|A_k(s,\xi)\frac{k^2\langle s\rangle^4\langle s-\xi/k\rangle^4}{(|\xi/k|^2+\langle s\rangle^2)^2}\big|\widetilde{H_a}(s,k,\xi)\big|^2\,d\xi ds\lesssim_\delta\eps_1^2.
\end{split}
\end{equation}
\end{lemma}

\begin{proof} The bounds on $H_3$ follow directly from the bootstrap assumptions on $\mathcal{E}_\Theta$ and $\mathcal{B}_\Theta$, and the definitions \eqref{defgellip}. 

Notice that $H_4=H_3+(V'-1)H_3$.  We use Lemma \ref{Multi0} (ii) to prove the bounds \eqref{nar34} for $a=4$. In view of \eqref{nar4} and \eqref{nar34} (with $a=3$), it suffices to prove the multiplier bounds
\begin{equation}\label{nar34.01}
\begin{split}
A_k(t,\xi)&\frac{|k|\langle t\rangle^2\langle t-\xi/k\rangle^2}{|\xi/k|^2+\langle t\rangle^2}\\
&\lesssim_\delta A_R(t,\xi-\eta)\cdot A_k(t,\eta)\frac{|k|\langle t\rangle^2\langle t-\eta/k\rangle^2}{|\eta/k|^2+\langle t\rangle^2}\cdot \{\langle\xi-\eta\rangle^{-2}+\langle k,\eta\rangle^{-2}\}
\end{split}
\end{equation}
and
\begin{equation}\label{nar34.02}
\begin{split}
\big|\dot{A}_k(t,\xi)&A_k(t,\xi)\big|^{1/2}\frac{|k|\langle t\rangle^2\langle t-\xi/k\rangle^2}{|\xi/k|^2+\langle t\rangle^2}\lesssim_\delta \left[\big|(\dot{A}_R/A_R)(t,\xi-\eta)\big|^{1/2}+\big|(\dot{A}_k/A_k)(t,\eta)\big|^{1/2}\right]\\
&\times A_R(t,\xi-\eta)\cdot A_k(t,\eta)\frac{|k|\langle t\rangle^2\langle t-\eta/k\rangle^2}{|\eta/k|^2+\langle t\rangle^2}\cdot \{\langle\xi-\eta\rangle^{-2}+\langle k,\eta\rangle^{-2}\},
\end{split}
\end{equation}
for any $t\in[1,T]$, $\xi,\eta\in\R$, and $k\in\mathbb{Z}\setminus\{0\}$.

To prove \eqref{nar34.01}--\eqref{nar34.02} we use Lemma \ref{TLX40}. In addition, by considering the cases $|\xi-\eta|\leq 10|k,\eta|$ and $|\xi-\eta|\geq 10|k,\eta|$, it is easy to see that
\begin{equation}\label{nar34.03}
\frac{|k|\langle t\rangle^2\langle t-\xi/k\rangle^2}{|\xi/k|^2+\langle t\rangle^2}\lesssim_\delta \frac{|k|\langle t\rangle^2\langle t-\eta/k\rangle^2}{|\eta/k|^2+\langle t\rangle^2}\cdot e^{\delta\min(\langle\xi-\eta\rangle,\langle k,\eta\rangle)^{1/2}}
\end{equation}
for any $t\in[1,T]$, $\xi,\eta\in\R$, and $k\in\mathbb{Z}\setminus\{0\}$. The bounds \eqref{nar34.01} follow from \eqref{TLX7} and \eqref{nar34.03}, while the bounds \eqref{nar34.02} follow from \eqref{TLX7}--\eqref{DtVMulti} and \eqref{nar34.03}.
\end{proof}

We now turn to the proof of (\ref{nar31}). We write
\begin{equation*}
\begin{split}
&\Big|2\Re\int_1^t\sum_{k\in \mathbb{Z}}\int_{\R}A_k^2(s,\xi)\widetilde{\mathcal{N}_2}(s,k,\xi)\overline{\widetilde{f}(s,k,\xi)}\,d\xi ds\Big|\\
&=C\Big|2\Re\Big\{\sum_{k,\ell\in \mathbb{Z}}\int_1^t\int_{\R^2}A_k^2(s,\xi)\widetilde{H_4}(s,k-\ell,\xi-\eta)i\eta\widetilde{f}(s,\ell,\eta)\overline{\widetilde{f}(s,k,\xi)}\,d\xi d\eta ds\Big\}\Big|\\
&=C\Big|\int_1^t\sum_{k,\ell\in \mathbb{Z}}\int_{\R^2}\big[\eta A_k^2(s,\xi)-\xi A_\ell^2(s,\eta)\big]\widetilde{H_4}(s,k-\ell,\xi-\eta)\widetilde{f}(s,\ell,\eta)\overline{\widetilde{f}(s,k,\xi)}\,d\xi d\eta ds\Big|,
\end{split}
\end{equation*}
where the second identity is proved by symmetrization (recall that $H_4$ is real-valued). 

With $R_0,R_1,R_2,R_3$ as in (\ref{nar18.1})-(\ref{nar18.4}), we define the corresponding integrals
\begin{equation}\label{nar34.1}
\begin{split}
\mathcal{V}_n:=\int_1^t\sum_{k,\ell\in \mathbb{Z}}\int_{\R^2}&\mathbf{1}_{R_n}((k,\xi),(\ell,\eta))\big|\eta A_k^2(s,\xi)-\xi A_\ell^2(s,\eta)\big|\,|\widetilde{H_4}(s,k-\ell,\xi-\eta)|\\
&\times|\widetilde{f}(s,\ell,\eta)|\,|\widetilde{f}(s,k,\xi)|\,d\xi d\eta ds.
\end{split}
\end{equation}

For $n=0,1$, we use (i) of Lemma \ref{TLXH3}. We remark that $\widetilde{H_a}(t,0,\cdot)\equiv0$ for $a\in\{3,4\}$. Denote $(\sigma,\rho)=(k-\ell,\xi-\eta)$. We can bound
\begin{equation*}
\begin{split}
\mathcal{V}_n&\lesssim_{\delta}\int_1^t\sum_{k,\ell\in \mathbb{Z}}\int_{\R^2}\mathbf{1}_{\sigma\neq0}\cdot\frac{\sigma \langle s\rangle^2}{|\rho/\sigma|^2+\langle s\rangle^2}\langle s-\rho/\sigma\rangle^2\,A_{\sigma}(s,\rho)e^{-(\delta_0/200)\langle \sigma,\rho\rangle^{1/2}}\big|\widetilde{H_4}(s,\sigma,\rho)\big|\\
&\qquad\times\sqrt{|A_k\dot{A}_k(s,\xi)|}\sqrt{|A_{\ell}\dot{A}_{\ell}(s,\eta)|}\,\big|\widetilde{f}(s,\ell,\eta)\big|\,\big|\widetilde{f}(s,k,\xi)\big|\,d\xi d\eta ds\\
&\lesssim_{\delta}\left\|\mathbf{1}_{\sigma\neq0}\cdot\frac{\sigma \langle s\rangle^2}{|\rho/\sigma|^2+\langle s\rangle^2}\langle s-\rho/\sigma\rangle^2\,A_{\sigma}(s,\rho)e^{-(\delta_0/300)\langle \sigma,\rho\rangle^{1/2}}\widetilde{H_4}(s,\sigma,\rho)\right\|_{L^{\infty}_sL^2_{\sigma,\rho}}\\
&\qquad\times\left\|\sqrt{|A_k\dot{A}_k(s,\xi)|}\,\widetilde{f}(s,k,\xi)\right\|_{L^2_sL^2_{k,\xi}}\cdot\left\|\sqrt{|A_{\ell}\dot{A}_{\ell}(s,\eta)|}\,\widetilde{f}(s,\ell,\eta)\right\|_{L^2_sL^2_{\ell,\eta}}\\&\lesssim_{\delta}\epsilon_1^3,
\end{split}
\end{equation*}
using (\ref{nar34}) and (\ref{boot2}). Moreover, for $n=2$, we use (ii) of Lemma \ref{TLXH3} to estimate
\begin{equation*}
\begin{split}
\mathcal{V}_2&\lesssim_{\delta}\int_1^t\sum_{k,\ell\in \mathbb{Z}}\int_{\R^2}\mathbf{1}_{\sigma\neq0}\cdot\frac{\sigma \langle s\rangle^2}{|\rho/\sigma|^2+\langle s\rangle^2}\langle s-\rho/\sigma\rangle^2\,\sqrt{|A_{\sigma}\dot{A}_{\sigma}(s,\rho)|}\big|\widetilde{H_4}(s,\sigma,\rho)\big|\\
&\qquad\times\sqrt{|A_k\dot{A}_k(s,\xi)|}\,\big|\widetilde{f}(s,k,\xi)\big|\,A_{\ell}(s,\eta)\,e^{-(\delta_0/200)\langle \ell,\eta\rangle^{1/2}}\big|\widetilde{f}(s,\ell,\eta)\big|\,d\xi d\eta ds\\
&\lesssim_{\delta}\left\|\mathbf{1}_{\sigma\neq0}\cdot\frac{\sigma \langle s\rangle^2}{|\rho/\sigma|^2+\langle s\rangle^2}\langle s-\rho/\sigma\rangle^2\,\sqrt{|A_{\sigma}\dot{A}_{\sigma}(s,\rho)|}\,\,\widetilde{H_4}(s,\sigma,\rho)\right\|_{L^{2}_sL^2_{\sigma,\rho}}\\
&\qquad\times\left\|\sqrt{|A_k\dot{A}_k(s,\xi)|}\,\widetilde{f}(s,k,\xi)\right\|_{L^2_sL^2_{k,\xi}}\cdot\left\|A_{\ell}(s,\eta)\,e^{-(\delta_0/300)\langle \ell,\eta\rangle^{1/2}}\,\widetilde{f}(s,\ell,\eta)\right\|_{L^{\infty}_sL^2_{\ell,\eta}}\\
&\lesssim_{\delta}\epsilon_1^3,
\end{split}
\end{equation*}
using (\ref{nar34}) and (\ref{boot2}). The case $n=3$ is identical to the case $n=2$, by symmetry. Thus $\mathcal{V}_n\lesssim_\delta\epsilon_1^3$ for all $n\in\{0,1,2,3\}$, and the desired bounds \eqref{nar31} follow.

\subsection{The nonlinearity $\mathcal{N}_3$} We now prove the following
\begin{lemma}\label{nar35}
With $\mathcal{N}_3$ defined as \eqref{nar8.1}, for any $t\in[1,T]$ we have
\begin{equation}\label{nar36}
\Big|2\Re\int_1^t\sum_{k\in \mathbb{Z}}\int_{\R}A_k^2(s,\xi)\widetilde{\mathcal{N}_3}(s,k,\xi)\overline{\widetilde{f}(s,k,\xi)}\,d\xi ds\Big|\lesssim_\delta \eps_1^3.
\end{equation}
\end{lemma}

The rest of this subsection is concerned with the proof of this lemma. As before, we write
\begin{equation*}
\begin{split}
&\Big|2\Re\int_1^t\sum_{k\in \mathbb{Z}}\int_{\R}A_k^2(s,\xi)\widetilde{\mathcal{N}_3}(s,k,\xi)\overline{\widetilde{f}(s,k,\xi)}\,d\xi ds\Big|\\
&=C\Big|2\Re\Big\{\sum_{k\in \mathbb{Z}}\int_1^t\int_{\R^2}A_k^2(s,\xi)\widetilde{\dot{V}}(s,\xi-\eta)i\eta\widetilde{f}(s,k,\eta)\overline{\widetilde{f}(s,k,\xi)}\,d\xi d\eta ds\Big\}\Big|\\
&=C\Big|\int_1^t\sum_{k\in \mathbb{Z}}\int_{\R^2}\big[\eta A_k^2(s,\xi)-\xi A_k^2(s,\eta)\big]\widetilde{\dot{V}}(s,\xi-\eta)\widetilde{f}(s,k,\eta)\overline{\widetilde{f}(s,k,\xi)}\,d\xi d\eta ds\Big|.
\end{split}
\end{equation*}
For $i\in\{0,1,2,3\}$ we define the sets
\begin{equation}\label{nar19.1}
\Sigma_i:=\big\{((k,\xi),(l,\eta))\in R_i:\,k=\ell\big\},
\end{equation}
where $R_i$ are as in \eqref{nar18.1}--\eqref{nar18.4}, and the corresponding integrals
\begin{equation}\label{nar19.5}
\begin{split}
\mathcal{W}_n:=\int_1^t\sum_{k\in \mathbb{Z}}\int_{\R^2}&\mathbf{1}_{\Sigma_n}((k,\xi),(k,\eta))\big|\eta A_k^2(s,\xi)-\xi A_k^2(s,\eta)\big|\,\big|\widetilde{\dot{V}}(s,\xi-\eta)\big|\\
&\times\big|\widetilde{f}(s,k,\eta)\big|\,\big|\widetilde{f}(s,k,\xi)\big|\,d\xi d\eta ds.
\end{split}
\end{equation}

To estimate $\mathcal{W}_n$, $n\in\{0,1\}$, we use (i) of Lemma \ref{TLXH2}. Let $\rho=\xi-\eta$, and estimate
\begin{equation*}
\begin{split}
\mathcal{W}_n&\lesssim_{\delta}\int_1^t\sum_{k\in \mathbb{Z}}\int_{\R^2}\big[\langle\rho\rangle\langle s\rangle+\langle \rho\rangle^{1/4}\langle s\rangle^{7/4}\big]\,A_{NR}(s,\rho)\,e^{-(\delta_0/200)\langle\rho\rangle^{1/2}}\big|\widetilde{\dot{V}}(s,\rho)\big|\\
&\qquad\times\sqrt{|(A_k\dot{A}_k)(s,\eta)|}\,\,\big|\widetilde{f}(s,k,\eta)\big|\,\sqrt{|(A_k\dot{A}_k)(s,\xi)|}\,\,\big|\widetilde{f}(s,k,\xi)\big|\,d\xi d\eta ds\\
&\lesssim_{\delta}\left\|\big[\langle\rho\rangle\langle s\rangle+\langle \rho\rangle^{1/4}\langle s\rangle^{7/4}\big]\,A_{NR}(s,\rho)\,e^{-(\delta_0/300)\langle\rho\rangle^{1/2}}\,\widetilde{\dot{V}}(s,\rho)\right\|_{L^{\infty}_sL^2_{\rho}}\\
&\qquad\times \left\|\sqrt{|(A_k\dot{A}_k)(s,\eta)|}\,\,\widetilde{f}(s,k,\eta)\right\|_{L^2_sL^2_{k,\eta}}\cdot\left\|\sqrt{|(A_k\dot{A}_k)(s,\xi)|}\,\,\widetilde{f}(s,k,\xi)\right\|_{L^2_sL^2_{k,\xi}}\\
&\lesssim_{\delta}\epsilon_1^3,
\end{split}
\end{equation*}
using (\ref{nar7}) and the bootstrap bounds (\ref{boot2}). Moreover, for $n=2$, we use (ii) of Lemma \ref{TLXH2} and estimate
\begin{equation*}
\begin{split}
\mathcal{W}_2&\lesssim_{\delta}\int_1^t\sum_{k\in \mathbb{Z}}\int_{\R^2}\big[\langle\rho\rangle\langle s\rangle+\langle \rho\rangle^{1/4}\langle s\rangle^{7/4}\big]\,\sqrt{|(A_{NR}\dot{A}_{NR})(s,\rho)|}\,\big|\widetilde{\dot{V}}(s,\rho)\big|\\
&\qquad\times A_k(s,\eta)\,e^{-(\delta_0/200)\langle k,\eta\rangle^{1/2}}\,\,\big|\widetilde{f}(s,k,\eta)\big|\,\sqrt{|(A_k\dot{A}_k)(s,\xi)|}\,\,\big|\widetilde{f}(s,k,\xi)\big|\,d\xi d\eta ds.\\
&\lesssim_{\delta}\left\|\big[\langle\rho\rangle\langle s\rangle+\langle \rho\rangle^{1/4}\langle s\rangle^{7/4}\big]\,\sqrt{|(A_{NR}\dot{A}_{NR})(s,\rho)|}\,\widetilde{\dot{V}}(s,\rho)\right\|_{L^{2}_sL^2_{\rho}}\\
&\qquad\times \left\|{A}_k(s,\eta)\,e^{-(\delta_0/300)\langle k,\eta\rangle^{1/2}}\,\,\widetilde{f}(s,k,\eta)\right\|_{L^{\infty}_sL^2_{k,\eta}}\cdot\left\|\sqrt{|(A_k\dot{A}_k)(s,\xi)|}\,\,\widetilde{f}(s,k,\xi)\right\|_{L^2_sL^2_{k,\xi}}\\
&\lesssim_{\delta}\epsilon_1^3,
\end{split}
\end{equation*}
using (\ref{nar7}) again. The case $n=3$ is identical to the case $n=2$, by symmetry. Thus $\mathcal{W}_n\lesssim_\delta\epsilon_1^3$ for all $n\in\{0,1,2,3\}$, and the desired bounds \eqref{nar36} follow.

\section{Improved control of the normalized stream function $\Theta$}\label{ellip}

We prove now the main bounds \eqref{boot3} for the function $\Theta$. More precisely:

\begin{proposition}\label{BootImp2}
With the definitions and assumptions in Proposition \ref{MainBootstrap}, we have
\begin{equation}\label{har1}
\mathcal{E}_\Theta(t)+\mathcal{B}_\Theta(t)\lesssim_{\delta}\epsilon_1^3\leq\eps_1^2/20\qquad\text{ for any }t\in[1,T].
\end{equation}
\end{proposition}
 
The rest of the section is concerned with the proof of this proposition. Recall the elliptic equation \eqref{rea26}
\begin{equation}\label{eq:Elliptic_t}
\partial_z^2\phi+|V'|^2(\partial_v-t\partial_z)^2\phi+V''(\partial_v-t\partial_z)\phi=f,
\end{equation}
for $(z,v)\in \mathbb{T}\times[c_0,c_0+1]$, with boundary conditions $\phi(z,c_0)=\phi(z,c_0+1)=0$.

It follows from \eqref{eq:Elliptic_t} that
\begin{equation}\label{EllipticR}
\partial_z^2\phi+(\partial_v-t\partial_z)^2\phi=f+(1-|V'|^2)\left(\partial_v-t\partial_z\right)^2\phi-V''(\partial_v-t\partial_z)\phi,
\end{equation}
for $(z,v)\in \mathbb{T}\times [c_0,c_0+1]$.  Using the support property of $1-V'$ and $V''$, we can replace $\phi$ by $\Psi\,\phi$ in the right hand side of (\ref{EllipticR}). Therefore, see definition \eqref{defgellip},
\begin{equation}\label{har4}
\begin{split}
\Theta&=\big[\partial_z^2+(\partial_v-t\partial_z)^2\big]\,\left(\Psi(v)\,\phi\right)\\
  &=\Psi\,\big[\partial_z^2+(\partial_v-t\partial_z)^2\big]\,\phi+2\partial_v\Psi\,(\partial_v-t\partial_z)\phi+\partial_v^2\Psi\,\phi\\
  &=\Psi\,f+(1-|V'|^2)(\partial_v-t\partial_z)^2(\Psi\phi)-V''(\partial_v-t\partial_z)(\Psi\phi)+2\,\partial_v\Psi\,(\partial_v-t\partial_z)\phi+\partial_v^2\Psi\,\phi\\
 &=\Psi\,f+g_{11}+g_{12}+g_2+g_3.
\end{split}
\end{equation}

We consider separately the contributions of the five terms in the right-hand side of \eqref{har4}. We remark that the terms $\Psi f,g_{11},g_{12}$ are easy to bound, using just the bootstrap assumptions and bilinear estimates (see Lemma \ref{har6}). The terms $g_2$ and $g_3$ are harder to bound (see Lemma \ref{dar30}), mainly because our bootstrap assumption gives information on the localized stream function $\Psi\phi$, but not on $\phi$ itself. 

\subsection{Bounds on the terms $\Psi f$, $g_{11}$, and $g_{12}$} In this subsection, we prove the following:

\begin{lemma}\label{har6} For any $t\in[0,T]$ and $G\in\{\Psi f,g_{11},g_{12}\}$ we have
\begin{equation}\label{har8}
\sum_{k\in \mathbb{Z}\setminus\{0\}}\int_{\R}A_k^2(t,\xi)\frac{|k|^2\langle t\rangle^2}{|\xi|^2+|k|^2\langle t\rangle^2}\big|\widetilde{G}(t,k,\xi)\big|^2\,d\xi\lesssim_\delta\eps_1^3
\end{equation}
and
\begin{equation}\label{har9}
\int_1^t\sum_{k\in \mathbb{Z}\setminus\{0\}}\int_{\R}|\dot{A}_k(s,\xi)|A_k(s,\xi)\frac{|k|^2\langle s\rangle^2}{|\xi|^2+|k|^2\langle s\rangle^2}\big|\widetilde{G}(s,k,\xi)\big|^2\,d\xi ds\lesssim_\delta\eps_1^3.
\end{equation}
\end{lemma}

\begin{proof} {\bf{Case 1.}} Assume first that $G=\Psi f$. We use Lemma \ref{Multi0} (ii). In view of \eqref{rec0} and Proposition \ref{BootImp1}, for \eqref{har8}--\eqref{har9} it suffices to show that
\begin{equation}\label{har10}
A_k(t,\xi)\frac{|k|\langle t\rangle}{|\xi|+|k|\langle t\rangle}\lesssim_\delta e^{\langle\xi-\eta\rangle^{3/4}/2}A_k(t,\eta)\{\langle\xi-\eta\rangle^{-2}+\langle k,\eta\rangle^{-2}\}
\end{equation}
and
\begin{equation}\label{har11}
\left|\dot{A}_k(t,\xi)A_k(t,\xi)\right|^{1/2}\frac{|k|\langle t\rangle}{|\xi|+|k|\langle t\rangle}\lesssim_\delta e^{\langle\xi-\eta\rangle^{3/4}/2}\left|\dot{A}_k(t,\eta)A_k(t,\eta)\right|^{1/2}\{\langle\xi-\eta\rangle^{-2}+\langle k,\eta\rangle^{-2}\}
\end{equation}
for any $t\in[1,T]$, $\xi,\eta\in\R$, and $k\in\mathbb{Z}\setminus\{0\}$. These bounds follow from \eqref{eq:CDW} and \eqref{TLX7}.

{\bf{Case 2.}} Assume now that $G=g_{11}=(1-|V'|^2)(\partial_v-t\partial_z)^2(\Psi\phi)$. We use again  Lemma \ref{Multi0} (ii). In view of \eqref{nar4} and the bootstrap assumptions, it suffices to prove that
\begin{equation}\label{har20}
A_k(t,\xi)\frac{|k|\langle t\rangle}{|\xi|+|k|\langle t\rangle}\lesssim_\delta A_R(t,\xi-\eta)\cdot A_k(t,\eta)\frac{|k|\langle t\rangle}{|\eta|+|k|\langle t\rangle}\frac{k^2+(\eta-tk)^2}{(\eta-tk)^2}\cdot \{\langle\xi-\eta\rangle^{-2}+\langle k,\eta\rangle^{-2}\}
\end{equation}
and
\begin{equation}\label{har21}
\begin{split}
\big|\dot{A}_k(t,\xi)&A_k(t,\xi)\big|^{1/2}\frac{|k|\langle t\rangle}{|\xi|+|k|\langle t\rangle}\lesssim_\delta \left[\big|(\dot{A}_R/A_R)(t,\xi-\eta)\big|^{1/2}+\big|(\dot{A}_k/A_k)(t,\eta)\big|^{1/2}\right]\\
&\times A_R(t,\xi-\eta)\cdot A_k(t,\eta)\frac{|k|\langle t\rangle}{|\eta|+|k|\langle t\rangle}\frac{k^2+(\eta-tk)^2}{(\eta-tk)^2}\cdot \{\langle\xi-\eta\rangle^{-2}+\langle k,\eta\rangle^{-2}\},
\end{split}
\end{equation}
for any $t\in[1,T]$, $\xi,\eta\in\R$, and $k\in\mathbb{Z}\setminus\{0\}$.

In view of \eqref{vfc30.7}, both bounds \eqref{har20} and \eqref{har21} follow from the estimates
\begin{equation}\label{har22}
e^{16\sqrt\delta\min(\langle\xi-\eta\rangle^{1/2},\langle\eta,k\rangle^{1/2})}A_k(t,\xi)\lesssim_\delta A_R(t,\xi-\eta)A_k(t,\eta)\cdot \{\langle\xi-\eta\rangle^{-2}+\langle k,\eta\rangle^{-2}\}
\end{equation}
for any $t\in[1,T]$, $\xi,\eta\in\R$, and $k\in\mathbb{Z}\setminus\{0\}$ (we disregard here the favorable factor $\frac{k^2+(\eta-tk)^2}{(\eta-tk)^2}$ in the right-hand side). The bounds \eqref{har22} follow from \eqref{TLX7}.

{\bf{Case 3.}} Finally, assume that $G=g_{12}=-V''(\partial_v-t\partial_z)(\Psi\phi)$. We use again  Lemma \ref{Multi0} (ii). In view of \eqref{nar4} and the bootstrap assumptions, it suffices to prove that
\begin{equation}\label{har30}
A_k(t,\xi)\frac{|k|\langle t\rangle}{|\xi|+|k|\langle t\rangle}\lesssim_\delta \frac{A_R(t,\xi-\eta)}{\langle\xi-\eta\rangle}\cdot A_k(t,\eta)\frac{|k|\langle t\rangle}{|\eta|+|k|\langle t\rangle}\frac{k^2+(\eta-tk)^2}{|\eta-tk|}\cdot \{\langle\xi-\eta\rangle^{-2}+\langle k,\eta\rangle^{-2}\}
\end{equation}
and
\begin{equation}\label{har31}
\begin{split}
\big|\dot{A}_k(t,\xi)&A_k(t,\xi)\big|^{1/2}\frac{|k|\langle t\rangle}{|\xi|+|k|\langle t\rangle}\lesssim_\delta \left[\big|(\dot{A}_R/A_R)(t,\xi-\eta)\big|^{1/2}+\big|(\dot{A}_k/A_k)(t,\eta)\big|^{1/2}\right]\\
&\frac{A_R(t,\xi-\eta)}{\langle\xi-\eta\rangle}\cdot A_k(t,\eta)\frac{|k|\langle t\rangle}{|\eta|+|k|\langle t\rangle}\frac{k^2+(\eta-tk)^2}{|\eta-tk|}\cdot \{\langle\xi-\eta\rangle^{-2}+\langle k,\eta\rangle^{-2}\},
\end{split}
\end{equation}
for any $t\in[1,T]$, $\xi,\eta\in\R$, and $k\in\mathbb{Z}\setminus\{0\}$.

We use the bounds 
\begin{equation}\label{TLX7.01}
A_k(t,\xi)\lesssim_\delta A_R(t,\xi-\eta)A_k(t,\eta)e^{-(\lambda(t)/20)\min(\langle\xi-\eta\rangle,\langle k,\eta\rangle)^{1/2}},
\end{equation}
\begin{equation}\label{vfc30.701}
\big|(\dot{A}_k/A_k)(t,\xi)\big|\lesssim_\delta \left\{\big|(\dot{A}_R/A_R)(t,\xi-\eta)\big|+\big|(\dot{A}_k/A_k)(t,\eta)\big|\right\}e^{12\sqrt\delta\min(\langle\xi-\eta\rangle,\langle k,\eta\rangle)^{1/2}},
\end{equation}
which are proved in Lemma \ref{TLX40} below. In view of \eqref{vfc30.701}, both bounds \eqref{har20} and \eqref{har21} follow from the estimates
\begin{equation}\label{har32}
\begin{split}
&A_k(t,\xi)\frac{|\eta|+|k|\langle t\rangle}{|\xi|+|k|\langle t\rangle}\lesssim_\delta \frac{A_R(t,\xi-\eta)}{\langle\xi-\eta\rangle}A_k(t,\eta)\frac{k^2+(\eta-tk)^2}{|\eta-tk|}\cdot e^{-20\sqrt\delta\min(\langle\xi-\eta\rangle^{1/2},\langle\eta,k\rangle^{1/2})}
\end{split}
\end{equation}
for any $t\in[1,T]$, $\xi,\eta\in\R$, and $k\in\mathbb{Z}\setminus\{0\}$. Using \eqref{TLX7}, for \eqref{har32} it suffices to prove that
\begin{equation*}
e^{-(\lambda(t)/40)\min(\langle\xi-\eta\rangle^{1/2},\langle\eta,k\rangle^{1/2})}\frac{|\eta|+|k|\langle t\rangle}{|\xi|+|k|\langle t\rangle}\lesssim \frac{1}{\langle\xi-\eta\rangle}\frac{k^2+(\eta-tk)^2}{|\eta-tk|}.
\end{equation*}
This bound is easy to see, by analyzing the two cases $|\xi-\eta|\leq|(k,\eta)|$ and $|\xi-\eta|\geq|(k,\eta)|$. This completes the proof of the lemma.
\end{proof}

\subsection{Bounds on the terms $g_2$, $g_3$} In this subsection, we consider the terms $g_2,g_3$ defined in \eqref{har4}. These terms are harder to bound because the bootstrap assumptions cannot be used directly. 

We have to understand explicitly the solution to the Dirichlet problem on $\mathbb{T}\times[c_0,c_0+1]$.  Let
\begin{equation}\label{har2}
\gamma(t,z,v):=\phi(t,z-tv,v).
\end{equation}
Since $\partial_v\gamma(t,z,v)=(\partial_v-t\partial_z)\phi(t,z-tv,v)$, it follows from \eqref{EllipticR} that the function $\gamma$ satisfies the equation
\begin{equation}\label{eq:elliptice}
\partial_z^2\gamma+\partial_v^2\gamma=f(t,z-tv,v)+(1-|V'|^2)\partial_v^2\gamma-V''\partial_v\gamma,
\end{equation}
for $(z,v)\in \mathbb{T}\times[c_0,c_0+1]$ with Dirichlet boundary conditions $\gamma(z,c_0)=\gamma(z,c_0+1)=0$. Let 
\begin{equation}\label{har3}
b^0(t,z)=\partial_v\gamma(t,z,c_0),\qquad b^1(t,z)=\partial_v\gamma(t,z,c_0+1),
\end{equation}
and let $a_k^0$ and $a_k^1$ denote the Fourier coefficients of the functions $b^0$ and $b^1$. Let $\gamma^\ast(t,k,v)$ denote the Fourier coefficients of the function $\gamma$, i.e.
\begin{equation*}
\gamma^\ast(t,k,v):=\frac{1}{2\pi}\int_{\mathbb{T}}\gamma(t,z,v)e^{-ikz}\,dz.
\end{equation*}

In view of the support assumptions on $f,1-V',V''$, we have $\partial_z^2\gamma+\partial_v^2\gamma=0$ in $\mathbb{T}\times([c_0,c_0+\vartheta_0]\cup[c_0+1-\vartheta_0,c_0+1])$.  We take the partial Fourier transform along $\mathbb{T}$, thus
\begin{equation*}
\partial_{v}^2\gamma^\ast(t,k,v)-k^2\gamma^{\ast}(t,k,v)=0\quad \text{ for }v\in [c_0,c_0+\vartheta_0]\cup[c_0+1-\vartheta_0,c_0+1],
\end{equation*}
for any $t\in[1,T],\,k\in\mathbb{Z}$. Therefore
\begin{equation}\label{dar4}
\begin{split}
&\gamma^\ast(t,k,v)=a_k^0(t)\frac{\sinh[k(v-c_0)]}{k}\qquad\quad\,\,\,\,\text{ if }v\in[c_0,c_0+\vartheta_0],\\
&\gamma^\ast(t,k,v)=a_k^1(t)\frac{\sinh[k(v-c_0-1)]}{k}\qquad\text{ if }v\in[c_0+1-\vartheta_0,c_0+1],
\end{split}
\end{equation}
for all $k\in\mathbb{Z}\setminus\{0\}$. Therefore
\begin{equation}\label{dar5}
\begin{split}
&(P_{\neq 0}\phi)(t,z,v)=\sum_{k\in\mathbb{Z}\setminus \{0\}}e^{ik(z+tv)}a_k^0(t)\frac{\sinh[k(v-c_0)]}{k}\qquad\quad\,\,\,\,\text{ if }v\in[c_0,c_0+\vartheta_0],\\
&(P_{\neq 0}\phi)(t,z,v)=\sum_{k\in\mathbb{Z}\setminus \{0\}}e^{ik(z+tv)}a_k^1(t)\frac{\sinh[k(v-c_0-1)]}{k}\qquad\text{ if }v\in[c_0+1-\vartheta_0,c_0+1].
\end{split}
\end{equation}

We observe now that the functions $\partial_v\Psi$ and $\partial^2_v\Psi$ are both supported in $[c_0+\vartheta_0/4,c_0+\vartheta_0/3]\cup[c_0+1-\vartheta_0/3,c_0+1-\vartheta_0/4]$, so the formulas in \eqref{dar5} are suitable to calculate $g_2$ and $g_3$. To estimate $g_2,g_3$ we prove first suitable bounds on the coefficients $a_k^0(t)$ and $a_k^1(t)$.

\subsubsection{The Fourier coefficients $a_k^0$ and $a_k^1$} \label{har10.6} In order to quantify the effect of the boundary, it is essential to obtain estimates on the Fourier coefficients $a^0_k(t)$ and $a^1_k(t)$. Let
\begin{equation}\label{dar5.5}
\Gamma(t,z,v):=f(t,z-tv,v)+(1-V'(v)^2)\partial_v^2\gamma(t,z,v)-V''(v)\partial_v\gamma(t,z,v)
\end{equation}
denote the function in the right-hand side of \eqref{eq:elliptice}. Taking Fourier transform in $z$ in \eqref{eq:elliptice} we have
\begin{equation}\label{equphik}
\partial_v^2\gamma^\ast(t,k,v)-k^2\gamma^\ast(t,k,v)=\Gamma^\ast(t,k,v):=\frac{1}{2\pi}\int_{\mathbb{T}}\Gamma(t,z,v)e^{-ik z}\,dz.
\end{equation}

For integers $k\in\mathbb{Z}\setminus\{0\}$, let $G_k(v,w)$ denote the Green function with frequency $k$, that is, 
\begin{equation}\label{eq:Helmoltz}
-\frac{d^2}{dv^2}G_k(v,w)+k^2G_k(v,w)=\delta_w(v),
\end{equation}
with Dirichlet boundary conditions $G_k(c_0,w)=G_k(c_0+1,w)=0$, $w\in [c_0,c_0+1]$. It is easy to see that we have the explicit formula 
\begin{equation}\label{eq:GreenFunction}
G_k(v,w)=\frac{1}{k\sinh k}
\begin{cases}
\sinh(k(1+c_0-w))\sinh (k(v-c_0))\qquad&\text{ if }v\leq w,\\
\sinh (k(w-c_0))\sinh(k(1+c_0-v))\qquad&\text{ if }v\geq w.
\end{cases}
\end{equation}
Therefore, using \eqref{equphik}, for any $k\in\mathbb{Z}\setminus \{0\}$,
\begin{equation}\label{dar8}
\gamma^\ast(t,k,v)=-\int_{[c_0,c_0+1]}\Gamma^\ast(t,k,w)G_k(v,w)\,dw.
\end{equation}

Notice that, for any $v,w\in[c_0,c_0+1]$,
\begin{equation}\label{eq:boundGreen1}
G_k(v,w)=G_k(w,v),\qquad 0\leq G_k(v,w)\leq\frac{1}{2|k|}e^{-|k||v-w|}.
\end{equation}
More importantly, for any $w\in (c_0,c_0+1)$, 
\begin{equation}\label{dar9}
\begin{split}
(\partial_vG_k)(c_0,w)&=\frac{\sinh(k(1+c_0-w))}{\sinh k}=:G_k^0(w),\\
(\partial_vG_k)(c_0+1,w)&=\frac{-\sinh(k(w-c_0))}{\sinh k}=:G_k^1(w).
\end{split}
\end{equation}
It follows from \eqref{dar8} and the definitions that
\begin{equation}\label{dar10}
a_k^\iota(t)=-\int_{[c_0,c_0+1]}\Gamma^\ast(t,k,w)G_k^\iota(w)\,dw,
\end{equation}
for any $t\in[1,T]$, $k\in\mathbb{Z}\setminus\{0\}$, and $\iota\in\{0,1\}$.

We examine now the function $\Gamma$ defined in \eqref{dar5.5}. Taking partial Fourier transforms along $\mathbb{T}$ and recalling the definition \eqref{har2} we have
\begin{equation*}
\Gamma^\ast(t,k,v)=e^{-itkv}f^\ast(t,k,v)+(1-V'(t,v)^2)\frac{d^2}{dv^2}[e^{-itkv}\phi^\ast(t,k,v)]-V''(t,v)\frac{d}{dv}[e^{-itkv}\phi^\ast(t,k,v)].
\end{equation*}
Due to the support properties of the functions $f$, $V'-1$, and $V''$ we can insert the cutoff functions in the right-hand side. Let $\Psi'$ denote a smooth function supported in $[c_0+0.8\vartheta_0,c_0+1-0.8\vartheta_0]$, equal to $1$ in $[c_0+0.9\vartheta_0,c_0+1-0.9\vartheta_0]$ and satisfying $\|\widetilde{\Psi'}(\xi)e^{\langle\xi\rangle^{3/4}}\|_{L^\infty}\lesssim 1$. Therefore, for $\iota\in\{0,1\}$,
\begin{equation}\label{dar12}
\begin{split}
\Gamma^\ast(t,k,v)G_k^\iota(v)&=\Psi'(v)G_k^\iota(v)\cdot e^{-itkv}f^\ast(t,k,v)\\
&+(1-V'(t,v)^2)\Psi'(v)G_k^\iota(v)\cdot \frac{d^2}{dv^2}[e^{-itkv}\phi^\ast(t,k,v)\Psi(v)]\\
&-V''(t,v)\Psi'(v)G_k^\iota(v)\cdot \frac{d}{dv}[e^{-itkv}\phi^\ast(t,k,v)\Psi(v)].
\end{split}
\end{equation}

We can now estimate the coefficients $a_k^\iota$ using the formula \eqref{dar10} and the general identity
\begin{equation}\label{dar13}
\int_\R a(v)b(v)\,dv=C\int_\R \mathcal{F} a(\xi)\mathcal{F}b(-\xi)\,d\xi,
\end{equation}
where $\mathcal{F}$ denotes the Fourier transform on $\R$. Let
\begin{equation}\label{dar14}
\begin{split}
X_k(\xi)&:=\sum_{\iota\in\{0,1\}}|\mathcal{F}(\Psi'\cdot G_k^\iota)(\xi)|,\\
Y_k(t,\xi)&:=\sum_{\iota\in\{0,1\}}\left\{\langle\xi\rangle^2|\mathcal{F}[(1-V'(t)^2)\cdot \Psi'G_k^\iota](\xi)|+\langle\xi\rangle|\mathcal{F}[V''(t)\cdot \Psi'G_k^\iota](\xi)|\right\}.
\end{split}
\end{equation}
Notice that all the functions are well-defined as functions on $\mathbb{R}$ that vanish outside the interval $[c_0,c_0+1]$, due to the cutoff factors. We prove now our main estimates on the coefficients $a_k^\iota(t)$.

\begin{lemma}\label{dar20}
With $X_k,Y_k$ defined as above, we have
\begin{equation}\label{dar21}
|a_k^\iota(t)|\lesssim \int_{\R}X_k(tk-\xi)|\widetilde{f}(t,k,\xi)|\,d\xi+ \int_{\R}\frac{Y_k(t,tk-\xi)}{k^2+|tk-\xi|^2}|\widetilde{\Theta}(t,k,\xi)|\,d\xi.
\end{equation}
for any $k\in\mathbb{Z}\setminus \{0\}$, $t\in[0,T]$, and $\iota\in\{0,1\}$. Moreover,
\begin{equation}\label{dar22}
|X_k(\xi)|\lesssim e^{-0.7\va_0|k|}e^{-|\xi|^{2/3}}
\end{equation}
and
\begin{equation}\label{dar22.5}
\begin{split}
\int_{\R}\langle\xi\rangle^{-4}A^2_R(t,\xi)|Y_k(t,\xi)|^2\,d\xi&\lesssim_\delta \eps_1^2e^{-1.4\va_0|k|},\\
\int_1^t\int_{\R}\langle\xi\rangle^{-4}\big|\dot{A}_R(s,\xi)\big|A_R(s,\xi)|Y_k(s,\xi)|^2\,d\xi ds&\lesssim_\delta \eps_1^2e^{-1.4\va_0|k|}.
\end{split}
\end{equation}
\end{lemma}

\begin{proof} The bounds \eqref{dar21} follow directly from the identities \eqref{dar10}--\eqref{dar14}. 

To prove \eqref{dar22} we use Lemma \ref{lm:Gevrey} (ii) first, so there is $C_0\geq 1$ such that $|D^m\Psi'(v)|\leq C_0^m(m+1)^{4m/3}$ for any $m\in\mathbb{Z}_+$ and $v\in[c_0,c_0+1]$. Moreover, using just the definitions \eqref{dar9}, $|D^mG_k^\iota(v)|\lesssim |k|^me^{-0.8\va_0|k|}$ in the support of $\Psi'$, for any $m\in\mathbb{Z}_+$ Therefore
\begin{equation}\label{dar23}
|D^m(\Psi' G_k^\iota)(v)|\leq e^{-0.8\va_0|k|}C_1^m(k^m+(m+1)^{4m/3})\cdot\mathbf{1}_{[c_0+0.8\va_0,c_0+1-0.8\va_0]}(v),
\end{equation} 
for any $m\in\mathbb{Z}_+$, for some constant $C_1\geq 1$. Using integration by parts in $v$,  it follows that
\begin{equation*}
X_k(\xi)\leq e^{-0.8\va_0|k|}C_2^m(k^m+(m+1)^{4m/3})\langle\xi\rangle^{-m},
\end{equation*}
for any $m\in\mathbb{Z}_+,\,\xi\in\mathbb{R}$, for some constant $C_1\geq 1$. As in the proof of Lemma \ref{lm:Gevrey} (i), the desired bounds \eqref{dar22} follow by taking $m=0$ if $|\xi|\lesssim |k|$, and $(m+1)^{4/3}$ close to $\langle\xi\rangle/(10C_2)$ if $|\xi|\gg |k|$.

To prove \eqref{dar22.5} we would like to use Lemma \ref{Multi0} (i). The multiplier bounds we need are
\begin{equation}\label{dar25}
\begin{split}
\langle\xi\rangle^{-a}|A_R(t,\xi)|&\lesssim_\delta \langle\eta\rangle^{-a}|A_R(t,\eta)|\frac{e^{|\xi-\eta|^{2/3}}}{\langle\xi-\eta\rangle^6},\\
\langle\xi\rangle^{-a}\big|(A_R\dot{A}_R)(t,\xi)\big|^{1/2}&\lesssim_\delta \langle\eta\rangle^{-a}\big|(A_R\dot{A}_R)(t,\eta)\big|^{1/2}\frac{e^{|\xi-\eta|^{2/3}}}{\langle\xi-\eta\rangle^6},\\
\end{split}
\end{equation}
for any $a\in[0,2]$, $\xi,\eta\in\mathbb{R}$, and $t\in[1,T]$. These bounds follow from \eqref{TLX4} and \eqref{vfc30}. Let 
\begin{equation*}
\begin{split}
&Y_{k,1}(t,\xi):=\sum_{\iota\in\{0,1\}}\langle\xi\rangle^2|\mathcal{F}[(1-V'(t)^2)\cdot \Psi'G_k^\iota](\xi)|,\\
&Y_{k,2}(t,\xi):=\sum_{\iota\in\{0,1\}}\langle\xi\rangle|\mathcal{F}[V''(t)\cdot \Psi'G_k^\iota](\xi)|,
\end{split}
\end{equation*}
compare with the definition \eqref{dar14}. We use Lemma \ref{Multi0} (i), the bounds \eqref{nar4} and \eqref{dar22}, and the multiplier bounds in the first line of \eqref{dar25} (with $a=0$ and with $a=1$). It follows that
\begin{equation*}
\int_\R\langle\xi\rangle^{-4}A^2_R(t,\xi)|Y_{k,\mu}(t,\xi)|^2\,d\xi\lesssim _\delta \eps_1^2e^{-1.4\va_0|k|},
\end{equation*}
for $\mu\in\{1,2\}$. Similarly, using Lemma \ref{Multi0} (i), the bounds \eqref{nar4} and \eqref{dar22}, and the multiplier bounds in the second line of \eqref{dar25} (with $a=0$ and with $a=1$), we have
\begin{equation*}
\int_0^t\int_\R\langle\xi\rangle^{-4}\big|\dot{A}_R(s,\xi)\big|A_R(s,\xi)|Y_{k,\mu}(s,\xi)|^2\,d\xi ds\lesssim _\delta \eps_1^2e^{-1.4\va_0|k|},
\end{equation*}
for $\mu\in\{1,2\}$. The desired bounds \eqref{dar22.5} follow.
\end{proof}

\subsubsection{Estimates on $g_2$ and $g_3$} We are now ready to prove our main estimates on the functions $g_2=2\partial_v\Psi\,(\partial_v-t\partial_z)\phi$ and $g_3=\partial_v^2\Psi\,\phi$. 

\begin{lemma}\label{dar30} For any $t\in[0,T]$ and $a\in\{2,3\}$ we have
\begin{equation}\label{dar31}
\sum_{k\in \mathbb{Z}\setminus\{0\}}\int_{\R}A_k^2(t,\xi)\frac{|k|^2\langle t\rangle^2}{|\xi|^2+|k|^2\langle t\rangle^2}\big|\widetilde{g_a}(t,k,\xi)\big|^2\,d\xi\lesssim_\delta\eps_1^3
\end{equation}
and
\begin{equation}\label{dar32}
\int_1^t\sum_{k\in \mathbb{Z}\setminus\{0\}}\int_{\R}\big|\dot{A}_k(s,\xi)\big|A_k(t,\xi)\frac{|k|^2\langle s\rangle^2}{|\xi|^2+|k|^2\langle s\rangle^2}\big|\widetilde{g_a}(s,k,\xi)\big|^2\,d\xi ds\lesssim_\delta\eps_1^3.
\end{equation}
\end{lemma}

\begin{proof} We start from the formulas in \eqref{dar5}. As in the proof of \eqref{dar22} (see \eqref{dar23}) we estimate
\begin{equation}\label{dar33}
|\mathcal{F}\{\partial^b_v\Psi(v)\cdot\sinh[k(v-c_0-\iota)]\}|(\rho)\lesssim e^{0.4\va_0|k|}e^{-|\rho|^{2/3}}\langle\rho\rangle^{-4},
\end{equation}
for any $\rho\in\mathbb{R}$, $k\in\Z$, $b\in\{1,2,3\}$, and $\iota\in\{0,1\}$. The factor $e^{0.4\va_0|k|}$ is related to the support properties of the functions $\partial^b_v\Psi$, see \eqref{rec0}. Therefore, using the formulas \eqref{dar5},
\begin{equation}\label{dar34}
|\widetilde{g_a}(t,k,\xi)|\lesssim e^{0.4\va_0|k|}\big(|a_k^0(t)|+|a_k^1(t)|\big)e^{-|\xi-kt|^{2/3}}.
\end{equation}

Since $|A_k(t,\xi)|\lesssim_\delta |A_k(t,kt)|e^{4\delta_0|\xi-tk|^{1/2}}$ (see \eqref{TLX7}), we can use \eqref{dar34} to bound the $\xi$ integral in \eqref{dar31}. For \eqref{dar31} it remains to prove that
\begin{equation}\label{dar36}
\sum_{k\in \mathbb{Z}\setminus\{0\}}A_k^2(t,tk)e^{0.8\va_0|k|}|a_k^\iota(t)|^2\lesssim_\delta\eps_1^3,
\end{equation}
for any $t\in[0,T]$ and $\iota\in\{0,1\}$. Similarly, using also \eqref{eq:CDW}, for \eqref{dar32} it suffices to prove that
\begin{equation}\label{dar37}
\int_1^t\sum_{k\in \mathbb{Z}\setminus\{0\}}\big|\dot{A}_k(s,sk)\big|A_k(s,sk)e^{0.8\va_0|k|}|a_k^\iota(s)|^2\,ds\lesssim_\delta\eps_1^3
\end{equation}
for any $t\in[0,T]$ and $\iota\in\{0,1\}$.

We examine \eqref{dar21}, and define
\begin{equation}\label{dar40}
a_{k,1}(t):=\int_{\R}X_k(tk-\xi)|\widetilde{f}(t,k,\xi)|\,d\xi,\qquad a_{k,2}(t):=\int_{\R}\frac{Y_k(t,tk-\xi)}{k^2+|tk-\xi|^2}|\widetilde{\Theta}(t,k,\xi)|\,d\xi.
\end{equation} 
We bound the two contributions separately, in the next two steps. The desired bounds \eqref{dar36}--\eqref{dar37} follow from \eqref{dar41}, \eqref{dar42}, \eqref{dar51}, and \eqref{dar52}.

{\bf{Step 1.}} We bound first the contributions of $a_{k,1}(t)$. In view of \eqref{dar22},
\begin{equation}\label{dar40.5}
|a_{k,1}(t)|^2\lesssim \int_{\R}e^{-1.4\va_0|k|}e^{-|\xi-tk|^{0.6}}|\widetilde{f}(t,k,\xi)|^2\,d\xi.
\end{equation}
Therefore
\begin{equation}\label{dar41}
\begin{split}
\sum_{k\in \mathbb{Z}\setminus\{0\}}A_k^2(t,tk)e^{0.8\va_0|k|}|a_{k,1}(t)|^2&\lesssim_\delta \sum_{k\in \mathbb{Z}\setminus\{0\}}\int_\R A_k^2(t,tk)e^{-0.6\va_0|k|}e^{-|\xi-tk|^{0.6}}|\widetilde{f}(t,k,\xi)|^2\,d\xi\\
&\lesssim_\delta \eps_1^3.
\end{split}
\end{equation}
The last estimate follows from \eqref{nar1} and the bounds $A_k^2(t,tk)\lesssim_\delta A_k^2(t,\xi)e^{4\delta_0\sqrt{|\xi-tk|}}$, see \eqref{TLX7}. Similarly, using again \eqref{dar40.5}, \eqref{nar1}, \eqref{vfc30}, and \eqref{TLX7},
\begin{equation}\label{dar42}
\begin{split}
\int_1^t\sum_{k\in \mathbb{Z}\setminus\{0\}}&\big|\dot{A}_k(s,sk)\big|A_k(s,sk)e^{0.8\va_0|k|}|a_{k,1}(s)|^2\,ds\\
&\lesssim \int_1^t\sum_{k\in \mathbb{Z}\setminus\{0\}}\int_\R\big|\dot{A}_k(s,sk)\big|A_k(s,sk)e^{-0.6\va_0|k|}e^{-|\xi-sk|^{0.6}}|\widetilde{f}(s,k,\xi)|^2\, d\xi ds\\
&\lesssim_\delta \eps_1^3.
\end{split}
\end{equation}

{\bf{Step 2.}} We bound now the contributions of $a_{k,2}(t)$. Notice that the functions $(1-V'(t))^2\Psi'G_k^\iota$ and  $V''(t)\Psi'G_k^\iota$ that appear in the definition \eqref{dar14} of $Y_k(t,\xi)$ are compactly supported in the interval $[c_0,c_0+1]$. Therefore, in view of the uncertainty principle and \eqref{dar22.5}, one also has the pointwise bounds
\begin{equation}\label{dar45}
A_R^2(t,\xi)|Y_k(t,\xi)|^2\lesssim_\delta \eps_1^2e^{-1.4\va_0|k|}\langle\xi\rangle^4
\end{equation}
for any $k\in\Z\setminus\{0\}$, $\xi\in\mathbb{R}$, and $t\in[1,T]$.

Using \eqref{dar40} and \eqref{dar45}, we have
\begin{equation}\label{dar46}
\begin{split}
|a_{k,2}(t)|^2&\lesssim_\delta \int_{\R}\frac{|Y_k(t,tk-\xi)|^2}{\langle tk-\xi\rangle^4}|\widetilde{\Theta}(t,k,\xi)|^2\min(\langle tk-\xi\rangle^2,\langle\xi\rangle^2)\,d\xi\\
&\lesssim_\delta \eps_1^2e^{-1.4\va_0|k|}\int_{\R}\frac{1}{A^2_R(t,tk-\xi)}|\widetilde{\Theta}(t,k,\xi)|^2\min(\langle tk-\xi\rangle^2,\langle\xi\rangle^2)\,d\xi.
\end{split}
\end{equation}
Therefore, using \eqref{TLX7},
\begin{equation}\label{dar51}
\begin{split}
&\sum_{k\in \mathbb{Z}\setminus\{0\}}A_k^2(t,tk)e^{0.8\va_0|k|}|a_{k,2}(t)|^2\\
&\lesssim_\delta \eps_1^2\sum_{k\in \mathbb{Z}\setminus\{0\}}\int_\R A_k^2(t,tk)e^{-0.6\va_0|k|}\frac{1}{A^2_R(t,tk-\xi)}|\widetilde{\Theta}(t,k,\xi)|^2\min(\langle tk-\xi\rangle^2,\langle\xi\rangle^2)\,d\xi\\
&\lesssim_\delta \eps_1^2\sum_{k\in \mathbb{Z}\setminus\{0\}}\int_\R A_k^2(t,\xi)e^{-0.6\va_0|k|}|\widetilde{\Theta}(t,k,\xi)|^2\frac{\min(\langle tk-\xi\rangle^2,\langle\xi\rangle^2)}{e^{(\lambda(t)/20)\min(\langle tk-\xi\rangle,\langle\xi\rangle)^{1/2}}}\,d\xi\\
&\lesssim_\delta\eps_1^4,
\end{split}
\end{equation}
where we also used and the bootstrap assumption $\mathcal{E}_\Theta(t)\lesssim\eps_1^2$ for the last inequality.

Similarly, we estimate first
\begin{equation*}
\begin{split}
\int_1^t\sum_{k\in \mathbb{Z}\setminus\{0\}}&\big|\dot{A}_k(s,sk)\big|A_k(s,sk)e^{0.8\va_0|k|}|a_{k,2}(s)|^2\,ds\lesssim \int_1^t\sum_{k\in \mathbb{Z}\setminus\{0\}}\int_\R\big|\dot{A}_k(s,sk)\big|\\
&\times A_k(s,sk)e^{0.8\va_0|k|}\frac{|Y_k(s,sk-\xi)|^2}{\langle sk-\xi\rangle^4}|\widetilde{\Theta}(s,k,\xi)|^2\min(\langle sk-\xi\rangle^2,\langle\xi\rangle^2)\, d\xi ds.
\end{split}
\end{equation*}
Using \eqref{TLX7} and \eqref{vfc30.7}, we estimate
\begin{equation}\label{dar51.1}
\int_1^t\sum_{k\in \mathbb{Z}\setminus\{0\}}\big|\dot{A}_k(s,sk)\big|A_k(s,sk)e^{0.8\va_0|k|}|a_{k,2}(s)|^2\,ds\lesssim_\delta I+II,
\end{equation}
where
\begin{equation*}
\begin{split}
I:=\int_1^t\sum_{k\in \mathbb{Z}\setminus\{0\}}\int_\R&\big|\dot{A}_k(s,\xi)\big|A_k(s,\xi)A_R^2(s,sk-\xi)\frac{\min(\langle sk-\xi\rangle^2,\langle\xi\rangle^2)}{e^{(\lambda(s)/40)\min(\langle sk-\xi\rangle,\langle\xi\rangle)^{1/2}}}\\
&\times e^{0.8\va_0|k|}\frac{|Y_k(s,sk-\xi)|^2}{\langle sk-\xi\rangle^4}|\widetilde{\Theta}(s,k,\xi)|^2\, d\xi ds
\end{split}
\end{equation*}
and
\begin{equation*}
\begin{split}
II:=\int_1^t\sum_{k\in \mathbb{Z}\setminus\{0\}}\int_\R&\big|\dot{A}_R(s,sk-\xi)\big|A_R(s,sk-\xi)A_k^2(s,\xi)\frac{\min(\langle sk-\xi\rangle^2,\langle\xi\rangle^2)}{e^{(\lambda(s)/40)\min(\langle sk-\xi\rangle,\langle\xi\rangle)^{1/2}}}\\
&\times e^{0.8\va_0|k|}\frac{|Y_k(s,sk-\xi)|^2}{\langle sk-\xi\rangle^4}|\widetilde{\Theta}(s,k,\xi)|^2\, d\xi ds.
\end{split}
\end{equation*}

Using \eqref{dar45} and the bootstrap assumption $\mathcal{B}_{\Theta}(t)\lesssim\eps_1^2$ we can bound
\begin{equation*}
\begin{split}
I\lesssim_\delta\eps_1^2\int_1^t\sum_{k\in \mathbb{Z}\setminus\{0\}}\int_\R&\big|\dot{A}_k(s,\xi)\big|A_k(s,\xi)\frac{\min(\langle sk-\xi\rangle^2,\langle\xi\rangle^2)}{e^{(\lambda(s)/40)\min(\langle sk-\xi\rangle,\langle\xi\rangle)^{1/2}}}|\widetilde{\Theta}(s,k,\xi)|^2\, d\xi ds\lesssim_\delta\eps_1^4.
\end{split}
\end{equation*}
To estimate $II$ we notice that we also have the pointwise bounds 
\begin{equation*}
A_k^2(s,\xi) \frac{|k|^2\langle s\rangle^2}{|\xi|^2+|k|^2\langle s\rangle^2}|\widetilde{\Theta}(s,k,\xi)|^2\lesssim\eps_1^2,
\end{equation*}
as a consequence of the bootstrap assumption $\mathcal{E}_{\Theta}(t)\lesssim\eps_1^2$ and the uncertainty principle. Thus
\begin{equation*}
\begin{split}
II\lesssim_\delta\eps_1^2\int_1^t\sum_{k\in \mathbb{Z}\setminus\{0\}}\int_\R&\big|\dot{A}_R(s,sk-\xi)\big|A_R(s,sk-\xi)\frac{\min(\langle sk-\xi\rangle^2,\langle\xi\rangle^2)}{e^{(\lambda(s)/40)\min(\langle sk-\xi\rangle,\langle\xi\rangle)^{1/2}}}\\
&\times e^{0.8\va_0|k|}\frac{|Y_k(s,sk-\xi)|^2}{\langle sk-\xi\rangle^4}\big(1+|\xi/(|k|\langle s\rangle)|^2\big)\, d\xi ds.
\end{split}
\end{equation*}
After a change of variables, we estimate
\begin{equation*}
\begin{split}
II\lesssim_\delta\eps_1^2\int_1^t\sum_{k\in \mathbb{Z}\setminus\{0\}}\int_\R&\big|\dot{A}_R(s,\eta)\big|A_R(s,\eta)\frac{|Y_k(s,\eta)|^2}{\langle \eta\rangle^4}\frac{\min(\langle sk-\eta\rangle^2,\langle\eta\rangle^2)}{e^{(\lambda(s)/40)\min(\langle sk-\eta\rangle,\langle\eta\rangle)^{1/2}}}\\
&\times e^{0.8\va_0|k|}\big(1+|\eta/(|k|\langle s\rangle)|^2\big)\, d\eta ds\lesssim_\delta\eps_1^4,
\end{split}
\end{equation*}
where we used the bounds \eqref{dar22.5} in the last inequality. Therefore, in view of \eqref{dar51.1},
\begin{equation}\label{dar52}
\int_1^t\sum_{k\in \mathbb{Z}\setminus\{0\}}\big|\dot{A}_k(s,sk)\big|A_k(s,sk)e^{0.8\va_0|k|}|a_{k,2}(s)|^2\,ds\lesssim_\delta\eps_1^4,
\end{equation}
as desired. This completes the proof of the lemma.
\end{proof}

\section{Improved control of the coordinate functions $V'-1$ and $\mathcal{H}$}\label{coimprov2}

In this section we prove the main bounds \eqref{boot3} for the functions $V'-1$ and $\mathcal{H}$. More precisely:

\begin{proposition}\label{BootImp3}
With the definitions and assumptions in Proposition \ref{MainBootstrap}, we have
\begin{equation}\label{yar1}
\mathcal{E}_{V'-1}(t)+\mathcal{E}_{\mathcal{H}}(t)+\mathcal{B}_{V'-1}(t)+\mathcal{B}_{\mathcal{H}}(t)\leq\eps_1^2/20\qquad\text{ for any }t\in[1,T].
\end{equation}
\end{proposition}

The rest of the section is concerned with the proof of this proposition. Using the equations \eqref{rea24}--\eqref{rea25} and the definitions \eqref{rec1}--\eqref{rec6.5} we calculate
\begin{equation*}
\begin{split}
\frac{d}{dt}[\mathcal{E}_{V'-1}+\mathcal{E}_{\mathcal{H}}](t)&=2\int_\R \dot{A}_R(t,\xi)A_R(t,\xi)\big|\widetilde{(V'-1)}(t,\xi)\big|^2\,d\xi\\
&+2K_\delta^{2}\int_{\R}\dot{A}_{NR}(t,\xi)A_{NR}(t,\xi)\big(\langle t\rangle/\langle\xi\rangle\big)^{3/2}\big|\widetilde{\mathcal{H}}(t,\xi)\big|^2\,d\xi\\
&+K_\delta^{2}\int_{\R}A^2_{NR}(t,\xi)\frac{3}{2}\big(t\langle t\rangle^{-1/2}\langle\xi\rangle^{-3/2}\big)\big|\widetilde{\mathcal{H}}(t,\xi)\big|^2\,d\xi\\
&+2\Re\int_\R A^2_R(t,\xi)\partial_t\widetilde{(V'-1)}(t,\xi)\overline{\widetilde{(V'-1)}(t,\xi)}\,d\xi\\
&+K_\delta^{2}2\Re\int_{\R}A^2_{NR}(t,\xi)\big(\langle t\rangle/\langle\xi\rangle\big)^{3/2}\partial_t\widetilde{\mathcal{H}}(t,\xi)\overline{\widetilde{\mathcal{H}}(t,\xi)}\,d\xi.
\end{split}
\end{equation*}
Therefore, since $\partial_tA_R\leq 0$ and $\partial_tA_{NR}\leq 0$, for any $t\in[1,T]$ we have
\begin{equation}\label{yar2}
\begin{split}
\mathcal{E}_{V'-1}(t)&+\mathcal{E}_{\mathcal{H}}(t)+\mathcal{B}_{V'-1}(t)+\mathcal{B}_{\mathcal{H}}(t)=\mathcal{E}_{V'-1}(1)+\mathcal{E}_{\mathcal{H}}(1)-\mathcal{B}_{V'-1}(t)-\mathcal{B}_{\mathcal{H}}(t)+\mathcal{L}_1(t)+\mathcal{L}_2(t),
\end{split}
\end{equation}
where
\begin{equation}\label{yar3}
\mathcal{L}_1(t):=2\Re\int_1^t\int_\R A^2_R(s,\xi)\partial_s\widetilde{(V'-1)}(s,\xi)\overline{\widetilde{(V'-1)}(s,\xi)}\,d\xi ds,
\end{equation}
\begin{equation}\label{yar4}
\begin{split}
\mathcal{L}_2(t):&=K_\delta^{2}2\Re\int_1^t\int_{\R}A^2_{NR}(s,\xi)\big(\langle s\rangle/\langle\xi\rangle\big)^{3/2}\partial_s\widetilde{\mathcal{H}}(s,\xi)\overline{\widetilde{\mathcal{H}}(s,\xi)}\,d\xi ds\\
&+K_\delta^{2}\int_1^t\int_{\R}A^2_{NR}(s,\xi)\frac{3}{2}\big(s\langle s\rangle^{-1/2}\langle\xi\rangle^{-3/2}\big)\big|\widetilde{\mathcal{H}}(s,\xi)\big|^2\,d\xi ds.
\end{split}
\end{equation}
Since $\mathcal{E}_{V'-1}(1)+\mathcal{E}_{\mathcal{H}}(1)\lesssim\eps_1^3$, for \eqref{yar1} it suffices to prove that, for any $t\in[1,T]$,
\begin{equation}\label{yar6}
-\mathcal{B}_{V'-1}(t)-\mathcal{B}_{\mathcal{H}}(t)+\mathcal{L}_1(t)+\mathcal{L}_2(t)\leq \eps_1^2/30.
\end{equation}

To prove \eqref{yar6} we use the equations \eqref{rea24}--\eqref{rea25}. We extract the quadratic components of $\mathcal{L}_1$ and $\mathcal{L}_2$ (corresponding to the linear terms in the right-hand sides of  \eqref{rea24}--\eqref{rea25}), so we define
\begin{equation}\label{yar7}
\mathcal{L}_{1,2}(t):=2\Re\int_1^t\int_\R \frac{A^2_R(s,\xi)}{s}\widetilde{\mathcal{H}}(s,\xi)\overline{\widetilde{(V'-1)}(s,\xi)}\,d\xi ds,
\end{equation}
and
\begin{equation}\label{yar8}
\begin{split}
\mathcal{L}_{2,2}(t)&:=K_\delta^{2}\int_1^t\int_{\R}\Big\{-A^2_{NR}(s,\xi)\frac{2\langle s\rangle^{3/2}}{s\langle\xi\rangle^{3/2}}|\widetilde{\mathcal{H}}(s,\xi)|^2+A^2_{NR}(s,\xi)\frac{3s/2}{\langle s\rangle^{1/2}\langle\xi\rangle^{3/2}}\big|\widetilde{\mathcal{H}}(s,\xi)\big|^2\Big\}\,d\xi ds\\
&=-K_\delta^{2}\int_1^t\int_{\R}A^2_{NR}(s,\xi)\frac{2+s^2/2}{s\langle\xi\rangle^{3/2}\langle s\rangle^{1/2}}|\widetilde{\mathcal{H}}(s,\xi)|^2\,d\xi ds.
\end{split}
\end{equation}
The desired bound \eqref{yar6} follows from Lemmas \ref{yar10} and \ref{yar20} below.

We prove first an estimate on the quadratic components.

\begin{lemma}\label{yar10}
For any $t\in [1,T]$ we have
\begin{equation}\label{yar11}
-\mathcal{B}_{V'-1}(t)-\mathcal{B}_{\mathcal{H}}(t)+\mathcal{L}_{1,2}(t)+\mathcal{L}_{2,2}(t)\leq \eps_1^2/40.
\end{equation}
\end{lemma}

\begin{proof} Since $\mathcal{L}_{2,2}(t)\leq 0$ for any $t\in[1,T]$, it suffices to prove that
\begin{equation*}
\mathcal{L}_{1,2}(t)\leq \mathcal{B}_{V'-1}(t)+\mathcal{B}_{\mathcal{H}}(t)+\eps_1^2/40.
\end{equation*}
Using Cauchy-Schwartz and the definitions, we have
\begin{equation*}
\mathcal{L}_{1,2}(t)\leq \frac{1}{2}\mathcal{B}_{V'-1}(t)+8\int_1^t\int_{\R}\frac{A^3_R(s,\xi)}{s^2|\dot{A}_R(s,\xi)|}|\widetilde{\mathcal{H}}(s,\xi)|^2\,d\xi ds.
\end{equation*}

In view of \eqref{nar6}, it suffices to show that 
\begin{equation*}
\frac{20A^3_R(s,\xi)}{s^2|\dot{A}_R(s,\xi)|}\leq K_\delta^{2}A_{NR}(s,\xi)|\dot{A}_{NR}(s,\xi)|(\langle s\rangle/\langle \xi\rangle)^{3/2}+(100C_\delta)^{-1}A_{NR}(s,\xi)|\dot{A}_{NR}(s,\xi)|,
\end{equation*}
where $C_\delta$ is the implicit constant in \eqref{nar6}. This is equivalent to proving that
\begin{equation*}
\frac{20A^2_R(s,\xi)}{A^2_{NR}(s,\xi)}\leq s^2\frac{|\dot{A}_R(s,\xi)|}{A_R(s,\xi)}\frac{|\dot{A}_{NR}(s,\xi)|}{A_{NR}(s,\xi)}\big[K_\delta^{2}(\langle s\rangle/\langle \xi\rangle)^{3/2}+(100C_\delta)^{-1}\big].
\end{equation*}
Using \eqref{dor3}, \eqref{dor20}, and \eqref{TLX3.5}, and setting $K_\delta$ sufficiently large, it suffices to prove that
\begin{equation}\label{yar18}
\frac{w_{NR}^2(s,\xi)}{w_R^2(s,\xi)}\lesssim_\delta\langle s\rangle^2(\langle s\rangle/\langle \xi\rangle)\left[\frac{\langle\xi\rangle^{1/2}}{\langle s\rangle^{1+\sigma_0}}+\left|\frac{\partial_sw_{NR}(s,\xi)}{w_{NR}(s,\xi)}\right|\right]\left[\frac{\langle\xi\rangle^{1/2}}{\langle s\rangle^{1+\sigma_0}}+\left|\frac{\partial_sw_{R}(s,\xi)}{w_{R}(s,\xi)}\right|\right],
\end{equation}
for any $s\geq 1$ and $\xi\in\R$. 

We examine the definitions in subsection \ref{weightsdefin} and notice that \eqref{yar18} follows easily if $\frac{w_{NR}(s,\xi)}{w_R(s,\xi)}\lesssim 1$, since the right-hand side is bounded from below by  $(\langle s\rangle^3/\langle\xi\rangle)\frac{\langle\xi\rangle^{1/2}}{\langle s\rangle^{1+\sigma_0}}\frac{\langle\xi\rangle^{1/2}}{\langle s\rangle^{1+\sigma_0}}\gtrsim 1$. Therefore, see \eqref{reb5.5}, it only remains to prove \eqref{yar18} when $\xi>|\delta|^{-10}$ and $|s-\xi/k|\leq \xi/(8k^2)$ for some $k\in\Z$ with $|k|\in\{1,\ldots,k_0(\xi)\}$. In this case for \eqref{yar18} it suffices to prove that
\begin{equation*}
\left(\frac{1+\delta^2|\xi|/(8k^2)}{1+\delta^2|s-\xi/k|}\right)^2\lesssim_\delta\langle s\rangle^2\left|\frac{\partial_sw_{NR}(s,\xi)}{w_{NR}(s,\xi)}\right|\cdot \frac{\langle s\rangle}{\langle\xi\rangle}\left|\frac{\partial_sw_{R}(s,\xi)}{w_{R}(s,\xi)}\right|.
\end{equation*}
This follows easily using \eqref{reb8}, which completes the proof of the lemma.
\end{proof}

We prove now estimates on the cubic and higher order terms. We examine the identities \eqref{rea24} and \eqref{rea25} and define
\begin{equation}\label{yar18.5}
\begin{split}
&F_1:=-\dot{V}\partial_v(V'-1),\\
&G_1:=-\dot{V}\partial_v\mathcal{H},\qquad G_2:=V'[-\big<\partial_vP_{\neq0}\phi\,\partial_zf\big>+\big<\partial_z\phi\,\partial_vf\big>].
\end{split}
\end{equation}
Notice that
\begin{equation}\label{yar19}
\begin{split}
&\mathcal{L}_1(t)=\mathcal{L}_{1,2}(t)+2\Re\int_1^t\int_\R A^2_R(s,\xi)\widetilde{F_1}(s,\xi)\overline{\widetilde{(V'-1)}(s,\xi)}\,d\xi ds,\\
&\mathcal{L}_2(t)=\mathcal{L}_{2,2}(t)+\sum_{a\in\{1,2\}}K_\delta^{2}2\Re\int_1^t\int_{\R}A^2_{NR}(s,\xi)\big(\langle s\rangle/\langle\xi\rangle\big)^{3/2}\widetilde{G_a}(s,\xi)\overline{\widetilde{\mathcal{H}}(s,\xi)}\,d\xi ds.
\end{split}
\end{equation}

The following lemma is our main estimate on the cubic and higher order contributions.

\begin{lemma}\label{yar20} 
For any $t\in[1,T]$ and $a\in\{1,2\}$ we have
\begin{equation}\label{yar21}
\Big|2\Re\int_1^t\int_\R A^2_R(s,\xi)\widetilde{F_1}(s,\xi)\overline{\widetilde{(V'-1)}(s,\xi)}\,d\xi ds\Big|\lesssim_\delta\eps_1^3
\end{equation}
and
\begin{equation}\label{yar22}
\Big|2\Re\int_1^t\int_{\R}A^2_{NR}(s,\xi)\big(\langle s\rangle/\langle\xi\rangle\big)^{3/2}\widetilde{G_a}(s,\xi)\overline{\widetilde{\mathcal{H}}(s,\xi)}\,d\xi ds\Big|\lesssim_\delta\eps_1^3.
\end{equation}
\end{lemma}

\subsection{Proof of Lemma \ref{yar20}} In this subsection we prove the bounds \eqref{yar21} and \eqref{yar22}.

\begin{lemma}\label{yar23}
The bounds \eqref{yar22} hold for $a=2$.
\end{lemma}

\begin{proof}
We estimate, using the Cauchy-Schwarz inequality,
\begin{equation*}
\begin{split}
\Big|\int_1^t\int_{\R}A^2_{NR}(s,\xi)&\big(\langle s\rangle/\langle\xi\rangle\big)^{3/2}\widetilde{G_2}(s,\xi)\overline{\widetilde{\mathcal{H}}(s,\xi)}\,d\xi ds\Big|^2\\
&\lesssim \mathcal{B}_{\mathcal{H}}(t)\int_1^t\int_{\R}|\dot{A}_{NR}(s,\xi)|^{-1}A^3_{NR}(s,\xi)\big(\langle s\rangle/\langle\xi\rangle\big)^{3/2}|\widetilde{G_2}(s,\xi)|^2\,d\xi ds.
\end{split}
\end{equation*}
In view of the bootstrap assumption on $\mathcal{B}_{\mathcal{H}}(t)$, it suffices to prove that
\begin{equation}\label{yar24}
\int_1^t\int_{\R}|\dot{A}_{NR}(s,\xi)|^{-1}A^3_{NR}(s,\xi)\big(\langle s\rangle/\langle\xi\rangle\big)^{3/2}|\widetilde{G_2}(s,\xi)|^2\,d\xi ds\lesssim_\delta \eps_1^4.
\end{equation}

Let $G_3:=-\big<\partial_vP_{\neq0}\phi\,\partial_zf\big>+\big<\partial_z\phi\,\partial_vf\big>$. Therefore, using also the support assumption on $f$,
\begin{equation*}
\begin{split}
G_3(t,v)&=\frac{1}{2\pi}\int_{\mathbb{T}}[-\partial_vP_{\neq0}\phi(t,z,v)\partial_zf(t,z,v)+\partial_z\phi(t,z,v)\partial_vf(t,z,v)\,dz\\
&=C\sum_{k\in\mathbb{Z}}\int_{\R^2}e^{iv(\rho+\eta)}[-\widetilde{\partial_v(\Psi\phi)}(t,k,\rho)\widetilde{\partial_zf}(t,-k,\eta)+\widetilde{\partial_z(\Psi\phi)}(t,k,\rho)\widetilde{\partial_vf}(t,-k,\eta)]\,d\rho d\eta.
\end{split}
\end{equation*}
Therefore
\begin{equation}\label{yar25}
\begin{split}
\widetilde{G_3}(t,\xi)&=C\sum_{k\in\mathbb{Z}}\int_{\R}k\xi\cdot\widetilde{\Psi\phi}(t,k,\xi-\eta)\widetilde{f}(t,-k,\eta)\,d\eta\\
&=C\sum_{k\in\mathbb{Z}\setminus\{0\}}\int_{\R}\frac{\xi}{k}\frac{1}{\langle t-\eta/k\rangle^2}\cdot\widetilde{\Theta}(t,k,\eta)\widetilde{f}(t,-k,\xi-\eta)\,d\eta.
\end{split}
\end{equation}

{\bf{Step 1.}} We prove first suitable bounds on $G_3$, more precisely
\begin{equation}\label{yar27}
\int_{\R}|\dot{A}_{NR}(t,\xi)|^{-2}A^4_{NR}(t,\xi)\big(\langle t\rangle/\langle\xi\rangle\big)^{3/2}|\widetilde{G_3}(t,\xi)|^2\,d\xi\lesssim_\delta \eps_1^4
\end{equation}
and
\begin{equation}\label{yar28}
\int_1^t\int_{\R}|\dot{A}_{NR}(s,\xi)|^{-1}A^3_{NR}(s,\xi)\big(\langle s\rangle/\langle\xi\rangle\big)^{3/2}|\widetilde{G_3}(s,\xi)|^2\,d\xi ds\lesssim_\delta \eps_1^4,
\end{equation}
for any $t\in[0,T]$. 

We use the multiplier bounds
\begin{equation}\label{yar29}
\frac{A^2_{NR}(t,\xi)}{|\dot{A}_{NR}(t,\xi)|}\frac{\langle t\rangle^{3/4}}{\langle\xi\rangle^{3/4}}|\xi/k|\lesssim_\delta A_k(t,\eta)\frac{\langle t\rangle\langle t-\eta/k\rangle^2}{\langle t\rangle+|\eta/k|}A_{-k}(t,\xi-\eta)\{\langle\xi-\eta\rangle^{-2}+\langle\eta\rangle^{-2}\}
\end{equation}
and
\begin{equation}\label{yar30}
\begin{split}
\frac{A^{3/2}_{NR}(t,\xi)}{|\dot{A}_{NR}(t,\xi)|^{1/2}}&\frac{\langle t\rangle^{3/4}}{\langle\xi\rangle^{3/4}}|\xi/k|\lesssim_\delta \big[|(\dot{A}_k/A_k)(t,\eta)|^{1/2}+|(\dot{A}_{-k}/A_{-k})(t,\xi-\eta)|^{1/2}\big]\\
&\times A_k(t,\eta)\frac{\langle t\rangle\langle t-\eta/k\rangle^2}{\langle t\rangle+|\eta/k|}A_{-k}(t,\xi-\eta)\{\langle\xi-\eta\rangle^{-2}+\langle\eta\rangle^{-2}\}
\end{split}
\end{equation}
for any $t\in[0,T]$, $k\in\Z\setminus\{0\}$, and $\xi,\eta\in\R$. The estimates \eqref{yar29} follow from \eqref{TLX41}, while the estimates \eqref{yar30} follow by combining \eqref{TLX41} and \eqref{TLX42}.

As in the proof of Lemma \ref{Multi0}, the estimates \eqref{yar27}--\eqref{yar28} follow from the multiplier bounds \eqref{yar29}--\eqref{yar30}. Indeed, to prove the harder bounds \eqref{yar28} we estimate first
\begin{equation*}
\begin{split}
\Big\{\int_1^t\int_{\R}&|\dot{A}_{NR}(s,\xi)|^{-1}A^3_{NR}(s,\xi)\big(\langle s\rangle/\langle\xi\rangle\big)^{3/2}|\widetilde{G_3}(s,\xi)|^2\,d\xi ds\Big\}^{1/2}\\
&\lesssim\sup_{\|P\|_{L^2([1,t]\times\R)}=1}\int_1^t\int_{\R}|P(s,\xi)||\dot{A}_{NR}(s,\xi)|^{-1/2}A^{3/2}_{NR}(s,\xi)\big(\langle s\rangle/\langle\xi\rangle\big)^{3/4}|\widetilde{G_3}(s,\xi)|\,d\xi ds.
\end{split}
\end{equation*}
Using now \eqref{yar25} and \eqref{yar30}, the right-hand side of the expression above is bounded by
\begin{equation}\label{yar31}
\begin{split}
&C_\delta\int_1^t\int_\R\int_\R \sum_{k\in\Z\setminus\{0\}}\big\{|P(s,\eta+\rho)|\big[|(\dot{A}_k/A_k)(s,\eta)|^{1/2}+|(\dot{A}_{-k}/A_{-k})(s,\rho)|^{1/2}\big]\\
&\times A_k(s,\eta)\frac{\langle s\rangle}{\langle s\rangle+|\eta/k|}A_{-k}(s,\rho)\{\langle\rho\rangle^{-2}+\langle\eta\rangle^{-2}\}\cdot|\widetilde{\Theta}(s,k,\eta)||\widetilde{f}(s,-k,\rho)|\big\}\,d\eta d\rho ds.
\end{split}
\end{equation}
We integrate first the variables $\eta$ and $\rho$. For any $k\in\Z$ and $t\in[1,T]$ let
\begin{equation*}
\begin{split}
\widetilde{f}^\ast(t,k)&:=\Big\{\int_{\R}A^2_k(t,\xi)|\widetilde{f}(t,k,\xi)|^2\,d\xi\Big\}^{1/2},\\
\widetilde{f}^{\ast\ast}(t,k)&:=\Big\{\int_{\R}|\dot{A}_k(t,\xi)|A_k(t,\xi)|\widetilde{f}(t,k,\xi)|^2\,d\xi\Big\}^{1/2}.
\end{split}
\end{equation*}
Similarly, for any $k\in\Z\setminus\{0\}$ and $t\in[1,T]$ let
\begin{equation*}
\begin{split}
\widetilde{\Theta}^\ast(t,k)&:=\Big\{\int_{\R}A^2_k(t,\xi)\frac{\langle t\rangle^2}{\langle t\rangle^2+|\xi/k|^2}|\widetilde{\Theta}(t,k,\xi)|^2\,d\xi\Big\}^{1/2},\\
\widetilde{\Theta}^{\ast\ast}(t,k)&:=\Big\{\int_{\R}|\dot{A}_k(t,\xi)|A_k(t,\xi)\frac{\langle t\rangle^2}{\langle t\rangle^2+|\xi/k|^2}|\widetilde{\Theta}(t,k,\xi)|^2\,d\xi\Big\}^{1/2}.
\end{split}
\end{equation*}
Letting also $P^\ast(s):=\|P(s,\xi)\|_{L^2_\xi}$, the expression in \eqref{yar31} is bounded by
\begin{equation*}
\begin{split}
C_\delta\int_1^t \sum_{k\in\Z\setminus\{0\}}&\big\{P^\ast(s)\widetilde{f}^\ast(s,-k)\widetilde{\Theta}^{\ast\ast}(s,k)+P^\ast(s)\widetilde{f}^{\ast\ast}(s,-k)\widetilde{\Theta}^{\ast}(s,k)\big\}ds\\
&\lesssim_\delta \|P^\ast\|_{L^2_s}\|\widetilde{f}^\ast\|_{L^\infty_sL^2_k}\|\widetilde{\Theta}^{\ast\ast}\|_{L^2_sL^2_k}+\|P^\ast\|_{L^2_s}\|\widetilde{f}^{\ast\ast}\|_{L^2_sL^2_k}\|\widetilde{\Theta}^{\ast}\|_{L^\infty_sL^2_k}.
\end{split}
\end{equation*}
The desired bounds \eqref{yar28} follow since $\|\widetilde{f}^\ast\|_{L^\infty_sL^2_k}+\|\widetilde{f}^{\ast\ast}\|_{L^2_sL^2_k}+\|\widetilde{\Theta}^{\ast}\|_{L^\infty_sL^2_k}+\|\widetilde{\Theta}^{\ast\ast}\|_{L^2_sL^2_k}\lesssim\eps_1$, as a consequence of the bootstrap assumptions on $f$ and $\Theta$. The estimates \eqref{yar27} follow in a similar (in fact slightly easier) way from the multiplier bounds \eqref{yar29}.

{\bf{Step 2.}} We prove now similar bounds on the function $G_2$
\begin{equation}\label{yar47}
\int_{\R}|\dot{A}_{NR}(t,\xi)|^{-2}A^4_{NR}(t,\xi)\big(\langle t\rangle/\langle\xi\rangle\big)^{3/2}|\widetilde{G_2}(t,\xi)|^2\,d\xi\lesssim_\delta \eps_1^4
\end{equation}
and
\begin{equation}\label{yar48}
\int_1^t\int_{\R}|\dot{A}_{NR}(s,\xi)|^{-1}A^3_{NR}(s,\xi)\big(\langle s\rangle/\langle\xi\rangle\big)^{3/2}|\widetilde{G_2}(s,\xi)|^2\,d\xi ds\lesssim_\delta \eps_1^4,
\end{equation}
for any $t\in[1,T]$. 

For this we notice that $G_2=V'\cdot G_3=(V'-1)G_3+G_3$. We would like to use the bounds \eqref{yar27}--\eqref{yar28} and the bootstrap assumptions $\mathcal{E}_{V'-1}(t)+\mathcal{B}_{V'-1}(t)\leq\eps_1^2$ in \eqref{boot2}. In view of Lemma \ref{Multi0} (i), it suffices to prove the multiplier estimates
\begin{equation}\label{yar49}
\frac{A^2_{NR}(t,\xi)}{|\dot{A}_{NR}(t,\xi)|}\frac{\langle t\rangle^{3/4}}{\langle\xi\rangle^{3/4}}\lesssim_\delta \frac{A^2_{NR}(t,\eta)}{|\dot{A}_{NR}(t,\eta)|}\frac{\langle t\rangle^{3/4}}{\langle\eta\rangle^{3/4}}A_R(t,\xi-\eta)\{\langle\xi-\eta\rangle^{-2}+\langle\eta\rangle^{-2}\}
\end{equation}
and
\begin{equation}\label{yar50}
\begin{split}
\frac{A^{3/2}_{NR}(t,\xi)}{|\dot{A}_{NR}(t,\xi)|^{1/2}}&\frac{\langle t\rangle^{3/4}}{\langle\xi\rangle^{3/4}}\lesssim_\delta \big[|(\dot{A}_{NR}/A_{NR})(t,\eta)|^{1/2}+|(\dot{A}_R/A_R)(t,\xi-\eta)|^{1/2}\big]\\
&\times \frac{A^2_{NR}(t,\eta)}{|\dot{A}_{NR}(t,\eta)|}\frac{\langle t\rangle^{3/4}}{\langle\eta\rangle^{3/4}}A_R(t,\xi-\eta)\{\langle\xi-\eta\rangle^{-2}+\langle\eta\rangle^{-2}\}.
\end{split}
\end{equation}
These bounds follow from \eqref{TLX71} and \eqref{DtVMulti} (using also \eqref{vfc30.5}). The desired estimates \eqref{yar24} follow from \eqref{yar48}. This completes the proof of the lemma.
\end{proof}

We estimate now the contributions of the transport terms $F_1$ and $G_1$.

\begin{lemma}\label{tol1}
The bounds \eqref{yar21} hold.
\end{lemma}

\begin{proof} Since $F_1=-\dot{V}\partial_v(V'-1)$, we write
\begin{equation*}
\begin{split}
&\Big|2\Re\int_1^t\int_\R A^2_R(s,\xi)\widetilde{F_1}(s,\xi)\overline{\widetilde{(V'-1)}(s,\xi)}\,d\xi ds\Big|\\
&=C\Big|2\Re\int_1^t\int_\R\int_\R A^2_R(s,\xi)\widetilde{\dot{V}}(s,\xi-\eta)(i\eta)\widetilde{(V'-1)}(s,\eta)\overline{\widetilde{(V'-1)}(s,\xi)}\,d\xi d\eta ds\Big|\\
&=C\Big|\int_1^t\int_\R\int_\R [\eta A^2_R(s,\xi)-\xi A_R^2(s,\eta)]\widetilde{\dot{V}}(s,\xi-\eta)\widetilde{(V'-1)}(s,\eta)\overline{\widetilde{(V'-1)}(s,\xi)}\,d\xi d\eta ds\Big|,
\end{split}
\end{equation*}
where the second identity is obtained by symmetrization,  using the fact that $\dot{V}$ is real-valued.

As in section \ref{fimprov} (see \eqref{nar18.1}--\eqref{nar18.4}), we define the sets
\begin{equation}\label{tol4}
\begin{split}
&S_0:=\Big\{(\xi,\eta)\in\R^2:\,\min(\langle\xi\rangle,\,\langle\eta\rangle,\,\langle\xi-\eta\rangle)\geq \frac{\langle\xi\rangle+\langle\eta\rangle+\langle\xi-\eta\rangle}{20}\Big\},\\
&S_1:=\Big\{(\xi,\eta)\in\R^2:\,\langle\xi-\eta\rangle\leq \frac{\langle\xi\rangle+\langle\eta\rangle+\langle\xi-\eta\rangle}{10}\Big\},\\
&S_2:=\Big\{(\xi,\eta)\in\R^2:\,\langle\eta\rangle\leq \frac{\langle\xi\rangle+\langle\eta\rangle+\langle\xi-\eta\rangle}{10}\Big\},\\
&S_3:=\Big\{(\xi,\eta)\in\R^2:\,\langle\xi\rangle\leq \frac{\langle\xi\rangle+\langle\eta\rangle+\langle\xi-\eta\rangle}{10}\Big\}.
\end{split}
\end{equation}
and the corresponding integrals
\begin{equation}\label{tol5}
\begin{split}
\mathcal{I}_n:=\int_1^t\int_\R\int_\R \mathbf{1}_{S_n}(\xi,\eta)&|\eta A^2_R(s,\xi)-\xi A_R^2(s,\eta)|\,|\widetilde{\dot{V}}(s,\xi-\eta)|\\
&\times|\widetilde{(V'-1)}(s,\eta)|\,|\widetilde{(V'-1)}(s,\xi)|\,d\xi d\eta ds.
\end{split}
\end{equation}
For \eqref{yar21} it suffices to prove that
\begin{equation}\label{tol6}
\mathcal{I}_n\lesssim_\delta \eps_1^3\qquad\text{ for }n\in\{0,1,2,3\}.
\end{equation}

We prove first the bounds \eqref{tol6} for $n=0$ and $n=1$. This is the main case, in which symmetrization is important. It follows from \eqref{TLY2.1} that
\begin{equation}\label{tol7}
\begin{split}
|\eta &A^2_R(s,\xi)-\xi A_R^2(s,\eta)|\\
&\lesssim_\delta s^{1.6}\sqrt{|(A_R\dot{A}_R)(s,\xi)|}\sqrt{|(A_R\dot{A}_R)(s,\eta)|}\cdot A_{NR}(s,\xi-\eta)|\xi-\eta|e^{-(\lambda(s)/40)\langle\xi-\eta\rangle^{1/2}},
\end{split}
\end{equation}
for any $(\xi,\eta)\in S_0\cup S_1$. Therefore
\begin{equation*}
\begin{split}
\mathcal{I}_n\lesssim_\delta \Big\|\sqrt{|(A_R\dot{A}_R)(s,\xi)|}&(\widetilde{V'-1})(s,\xi)\Big\|_{L^2_sL^2_\xi}\Big\|\sqrt{|(A_R\dot{A}_R)(s,\eta)|}(\widetilde{V'-1})(s,\eta)\Big\|_{L^2_sL^2_\eta}\\
&\times\Big\|s^{1.6}A_{NR}(s,\rho)|\rho|\langle\rho\rangle e^{-(\lambda(s)/40)\langle\rho\rangle^{1/2}}\cdot\widetilde{\dot{V}}(s,\rho)\Big\|_{L^\infty_sL^2_\rho},
\end{split}
\end{equation*}
and the desired estimates follow from \eqref{nar4} and \eqref{nar7}. 

We prove now the bounds \eqref{tol6} for $n=2$ and $n=3$. The two estimates are similar; for concreteness we will assume that $n=2$. The bounds \eqref{TLY2.2} show that
\begin{equation*}
\begin{split}
|\eta A^2_R(s,&\xi)-\xi A_R^2(s,\eta)|\lesssim_\delta \langle\eta\rangle A^2_R(s,\xi)\\
&\lesssim_\delta s^{1.1}\langle\xi-\eta\rangle^{0.6}\sqrt{|(A_R\dot{A}_R)(s,\xi)|}\sqrt{|(A_{NR}\dot{A}_{NR})(s,\xi-\eta)|}\cdot A_{R}(s,\eta)e^{-(\lambda(s)/40)\langle\eta\rangle^{1/2}}
\end{split}
\end{equation*}
for any $(\xi,\eta)\in S_2$. Therefore
\begin{equation*}
\begin{split}
\mathcal{I}_2\lesssim_\delta \Big\|\sqrt{|(A_R\dot{A}_R)(s,\xi)|}&(\widetilde{V'-1})(s,\xi)\Big\|_{L^2_sL^2_\xi}\Big\|s^{1.1}\langle\rho\rangle^{0.6}\sqrt{|(A_{NR}\dot{A}_{NR})(s,\rho)|} \cdot\widetilde{\dot{V}}(s,\rho)\Big\|_{L^2_sL^2_\rho}\\
&\times\Big\|A_R(s,\eta)\langle\eta\rangle e^{-(\lambda(s)/40)\langle\eta\rangle^{1/2}}(\widetilde{V'-1})(s,\eta)\Big\|_{L^\infty_sL^2_\eta},
\end{split}
\end{equation*}
and the desired bounds follow from \eqref{nar4} and \eqref{nar7}. This completes the proof of the lemma.
\end{proof}

\begin{lemma}\label{tor1}
The bounds \eqref{yar22} hold for $a=1$.
\end{lemma}

\begin{proof} Since $G_1=-\dot{V}\partial_v\mathcal{H}$, as in the proof of Lemma \ref{tol1} we have
\begin{equation*}
\begin{split}
&\Big|2\Re\int_1^t\int_\R A^2_{NR}(s,\xi)\big(\langle s\rangle/\langle\xi\rangle\big)^{3/2}\widetilde{G_1}(s,\xi)\overline{\widetilde{\mathcal{H}}(s,\xi)}\,d\xi ds\Big|=C\Big|\int_1^t\int_\R\int_\R\langle s\rangle^{3/2}\\
&\times\big[\eta A^2_{NR}(s,\xi)\langle\xi\rangle^{-3/2}-\xi A^2_{NR}(s,\eta)\langle\eta\rangle^{-3/2}\big]\widetilde{\dot{V}}(s,\xi-\eta)\widetilde{\mathcal{H}}(s,\eta)\overline{\widetilde{\mathcal{H}}(s,\xi)}\,d\xi d\eta ds\Big|.
\end{split}
\end{equation*}

With $S_n$ defined as in \eqref{tol4}, we define the corresponding integrals
\begin{equation}\label{tor5}
\begin{split}
\mathcal{J}_n:=\int_1^t\int_\R\int_\R\mathbf{1}_{S_n}(\xi,\eta)\langle s\rangle^{3/2}\big|\eta A^2_{NR}(s,\xi)\langle\xi\rangle^{-3/2}-\xi A^2_{NR}(s,\eta)\langle\eta\rangle^{-3/2}\big|\\
\times|\widetilde{\dot{V}}(s,\xi-\eta)|\,|\widetilde{\mathcal{H}}(s,\eta)|\,|\widetilde{\mathcal{H}}(s,\xi)|\,d\xi d\eta ds.
\end{split}
\end{equation}
For \eqref{yar22} it suffices to prove that
\begin{equation}\label{tor6}
\mathcal{J}_n\lesssim_\delta \eps_1^3\qquad\text{ for }n\in\{0,1,2,3\}.
\end{equation}

We prove first the bounds \eqref{tor6} for $n=0$ and $n=1$. It follows from \eqref{TLY2.1} that
\begin{equation*}
\begin{split}
|\eta &A^2_{NR}(s,\xi)\langle\xi\rangle^{-3/2}-\xi A_{NR}^2(s,\eta)\langle\eta\rangle^{-3/2}|\\
&\lesssim_\delta s^{1.6}\frac{\sqrt{|(A_{NR}\dot{A}_{NR})(s,\xi)|}}{\langle\xi\rangle^{3/4}}\frac{\sqrt{|(A_{NR}\dot{A}_{NR})(s,\eta)|}}{\langle\eta\rangle^{3/4}}\cdot A_{NR}(s,\xi-\eta)|\xi-\eta|e^{-(\lambda(s)/40)\langle\xi-\eta\rangle^{1/2}},
\end{split}
\end{equation*}
for any $(\xi,\eta)\in S_0\cup S_1$. Therefore
\begin{equation*}
\begin{split}
\mathcal{J}_n\lesssim_\delta \bigg\| s^{3/4}&\frac{\sqrt{|(A_{NR}\dot{A}_{NR})(s,\xi)|}}{\langle\xi\rangle^{3/4}}\cdot \widetilde{\mathcal{H}}(s,\xi)\bigg\|_{L^2_sL^2_\xi}\bigg\|s^{3/4}\frac{\sqrt{|(A_{NR}\dot{A}_{NR})(s,\eta)|}}{\langle\eta\rangle^{3/4}}\cdot \widetilde{\mathcal{H}}(s,\eta)\bigg\|_{L^2_sL^2_\eta}\\
&\times\Big\|s^{1.6}A_{NR}(s,\rho)|\rho|\langle\rho\rangle e^{-(\lambda(s)/40)\langle\rho\rangle^{1/2}}\cdot\widetilde{\dot{V}}(s,\rho)\Big\|_{L^\infty_sL^2_\rho},
\end{split}
\end{equation*}
and the desired conclusion follows from \eqref{nar6} and \eqref{nar7}. 

We prove now the bounds \eqref{tor6} for $n=2$ (the case $n=3$ is similar). Using \eqref{TLY2.3} we estimate
\begin{equation*}
\begin{split}
|\eta &A^2_{NR}(s,\xi)\langle\xi\rangle^{-3/2}-\xi A_{NR}^2(s,\eta)\langle\eta\rangle^{-3/2}|\lesssim_\delta\langle\eta \rangle A^2_{NR}(s,\xi)\langle\xi\rangle^{-3/2}\\
&\lesssim_\delta s^{1.1}\langle\xi\rangle^{-1.9}\sqrt{|(A_{NR}\dot{A}_{NR})(s,\xi)|}\sqrt{|(A_{NR}\dot{A}_{NR})(s,\xi-\eta)|}\cdot A_{NR}(s,\eta)e^{-(\lambda(s)/40)\langle\eta\rangle^{1/2}},
\end{split}
\end{equation*}
for any $(\xi,\eta)\in S_2$. Therefore
\begin{equation*}
\begin{split}
\mathcal{J}_2\lesssim_\delta \bigg\|s^{3/4}&\frac{\sqrt{|(A_{NR}\dot{A}_{NR})(s,\xi)|}}{\langle\xi\rangle^{3/4}}\cdot \widetilde{\mathcal{H}}(s,\xi)\bigg\|_{L^2_sL^2_\xi}\Big\| s^{1.1}\langle\rho\rangle^{-0.2}\sqrt{|(A_{NR}\dot{A}_{NR})(s,\rho)|}\cdot \widetilde{\dot{V}}(s,\rho)\Big\|_{L^2_sL^2_\eta}\\
&\times\Big\|s^{3/4}A_{NR}(s,\eta)\langle\eta\rangle e^{-(\lambda(s)/40)\langle\eta\rangle^{1/2}}\cdot\widetilde{\mathcal{H}}(s,\eta)\Big\|_{L^\infty_sL^2_\eta},
\end{split}
\end{equation*}
and the desired conclusion follows from \eqref{nar6} and \eqref{nar7}. 
\end{proof}

\section{The main weights: definitions and basic properties}\label{weights}

\subsection{Definitions}\label{weightsdefin}

In this subsection we give the precise definitions of the weights $w_\ast, b_\ast, A_\ast$, $\ast\in\{NR,R,k\}$, $k\in\mathbb{Z}$. 

\subsubsection{The functions $w_{NR}$, $w_R$, and $w_k$} We define first the functions $w_{NR},w_R:[0,\infty)\times\mathbb{R}\to [0,1]$ which model the non-resonant and resonant growth respectively. Take small $\delta>0$ with $\delta\ll \delta_0$, which is still much larger than $ \overline{\epsilon}$. For $|\eta|\leq\delta^{-10}$ we define simply
\begin{equation}\label{reb1}
w_{NR}(t,\eta):=1,\qquad w_R(t,\eta):=1.
\end{equation}
For $\eta>\delta^{-10}$ we define $k_0(\eta):=\lfloor\sqrt{\delta^3\eta}\rfloor$. For $l\in\{1,\ldots,k_0(\eta)\}$ we define
\begin{equation}\label{reb2}
t_{l,\eta}:=\frac{1}{2}\big(\frac{\eta}{l+1}+\frac{\eta}{l}\big),\qquad t_{0,\eta}:=2\eta,\qquad I_{l,\eta}:=[t_{l,\eta},\,t_{l-1,\eta}].
\end{equation}
Notice that $|I_{l,\eta}|\approx \frac{\eta}{l^2}$ and
\begin{equation*}
\delta^{-3/2}\sqrt{\eta}/2\leq t_{k_0(\eta),\eta}\leq\ldots\leq t_{l,\eta}\leq\eta/l\leq t_{l-1,\eta}\leq\ldots\leq t_{0,\eta}=2\eta.
\end{equation*}

We define
\begin{equation}\label{reb3}
w_{NR}(t,\eta):=1,\,w_{R}(t,\eta):=1\qquad\text{ if }\,\,t\geq t_{0,\eta}=2\eta.
\end{equation}
Then we define, for $k\in\{1,\ldots,k_0(\eta)\}$,
\begin{equation}\label{reb5}
\begin{split}
w_{NR}(t,\eta)&:=\Big(\frac{1+\delta^2|t-\eta/k|}{1+\delta^2|t_{k-1,\eta}-\eta/k|}\Big)^{\delta_0}w_{NR}(t_{k-1,\eta},\eta)\qquad\text{ if }t\in[\eta/k,t_{k-1,\eta}],\\
w_{NR}(t,\eta)&:=\Big(\frac{1}{1+\delta^2|t-\eta/k|}\Big)^{1+\delta_0}w_{NR}(\eta/k,\eta)\qquad\text{ if }t\in[t_{k,\eta},\eta/k].
\end{split}
\end{equation}
We define also the weight $w_R$ by the formula
\begin{equation}\label{reb5.5}
w_R(t,\eta):=
\begin{cases}
w_{NR}(t,\eta)\frac{1+\delta^2|t-\eta/k|}{1+\delta^2\eta/(8k^2)}\qquad&\text{ if }|t-\eta/k|\leq\eta/(8k^2)\\
w_{NR}(t,\eta)\qquad&\text{ if }t\in I_{k,\eta},\,|t-\eta/k|\geq\eta/(8k^2),
\end{cases}
\end{equation}
for any $k\in\{1,\ldots,k_0(\eta)\}$. Notice that
\begin{equation}\label{reb4}
\frac{w_{NR}(t_{k,\eta},\eta)}{w_{NR}(t_{k-1,\eta},\eta)}\approx \Big(\frac{k^2}{\delta^2\eta}\Big)^{1+2\delta_0},\qquad w_{R}(t_{k,\eta},\eta)=w_{NR}(t_{k,\eta},\eta).
\end{equation}
Moreover, notice that for $t\in I_{k,\eta}$,
\begin{equation}\label{reb7}
w_{R}(t,\eta)\approx w_{NR}(t,\eta)\left[\frac{k^2}{\delta^2\eta}\left(1+\delta^2|t-\eta/k|\right)\right],
\end{equation}
and
\begin{equation}\label{reb8}
\frac{\partial_tw_{NR}(t,\eta)}{w_{NR}(t,\eta)}\approx\frac{\partial_tw_R(t,\eta)}{w_R(t,\eta)}\approx \frac{\delta^2}{1+\delta^2\left|t-\eta/k\right|}.
\end{equation}

We observe that 
\begin{equation}\label{reb8.5}
e^{(J_2-J_1)\ln(A/J_2^2)}\leq\prod_{j=J_1+1}^{J_2}\frac{A}{j^2}\leq e^{(J_2-J_1)\ln(A/J_2^2)+4(J_2-J_1)}
\end{equation}
provided that $1\leq J_1+1\leq J_2$. In particular, for $\eta>\delta^{-10}$,
\begin{equation}\label{reb8.6}
\begin{split}
&w_{NR}(t_{k_0(\eta),\eta},\eta)=w_R(t_{k_0(\eta),\eta},\eta)\in[X_\delta(\eta)^4,X_\delta(\eta)^{1/4}],\\
&X_\delta(\eta):=e^{-\delta^{3/2}\ln(\delta^{-1})\sqrt\eta}.
\end{split}
\end{equation}

For small values of $t\leq t_{k_0(\eta),\eta}$ we define the weights $w_{NR}$ and $w_R$ by the formulas 
\begin{equation}\label{reb9}
w_{NR}(t,\eta)=w_R(t,\eta):=(e^{-\delta\sqrt\eta})^\beta w_{NR}(t_{k_0(\eta),\eta},\eta)^{1-\beta}
\end{equation}
if $t=(1-\beta)t_{k_0(\eta),\eta}$, $\beta\in[0,1]$. We notice that
\begin{equation}\label{reb9.5}
\frac{w_{NR}(t_1,\eta)}{w_{NR}(t_2,\eta)}\lesssim e^{4\delta^{5/2}|t_1-t_2|}\qquad\text{ for any }t_1\in[0,t_{k_0(\eta),\eta}],\,t_2\in[0,\infty).
\end{equation}

If $\eta<-\delta^{-10}$, then we define $w_R(t,\eta):=w_R(t,|\eta|)$, $w_{NR}(t,\eta):=w_{NR}(t,|\eta|)$ and the resonant intervals $I_{k,\eta}:=I_{-k,-\eta}$. To summarize, the resonant intervals $I_{k,\eta}$ are defined for $(k,\eta)\in\mathbb{Z}\times\mathbb{R}$ satisfying $|\eta|>\delta^{-10}$, $1\leq |k|\leq  \sqrt{\delta^3|\eta|}$, and $\eta/k>0$.

We define now the weights $w_k(t,\eta)$, which crucially distinguish the way resonant and non-resonant modes grow around the critical times $\eta/k$, by the formula
\begin{equation}\label{eq:resonantweight}
w_k(t,\eta):=\left\{\begin{array}{lll}
w_{NR}(t,\eta)&{\rm \,if\,}&t\not\in I_{k,\eta},\\
w_R(t,\eta)&{\rm \,if\,}&t\in I_{k,\eta}.
\end{array}\right.
\end{equation}
If particular $w_k(t,\eta)=w_{NR}(t,\eta)$ unless $|\eta|>\delta^{-10}$, $1\leq |k|\leq  \sqrt{\delta^3|\eta|}$, $\eta/k>0$, and $t\in I_{k,\eta}$.

The functions $w_{NR}$, $w_{R}$ and $w_k$ have the right size but lack optimal smoothness in the frequency parameter $\eta$, mainly due to the jump discontinuities of the function $k_0(\eta)$. The smoothness of the weights is important in the analysis of the transport terms, as it leads to smaller loss of derivatives after symmetrization in the energy functionals.

To correct this problem we mollify the weights $w_\ast$. We fix $\varphi: \R \to [0,1]$ an even smooth function supported in $[-8/5,8/5]$ and equal to $1$ in $[-5/4,5/4]$ and let $d_0:=\int_\mathbb{R}\varphi(x)\,dx$. For $k\in\mathbb{Z}$ and $\ast\in \{NR,R,k\}$ let
\begin{equation}\label{dor1}
\begin{split}
b_\ast(t,\xi)&:=\int_\R w_\ast(t,\rho)\varphi\Big(\frac{\xi-\rho}{L_\kappa(t,\xi)}\Big)\frac{1}{d_0L_\kappa(t,\xi)}\,d\rho,\\
L_\kappa(t,\xi)&:=1+\frac{\kappa\langle\xi\rangle}{\langle\xi\rangle^{1/2}+\kappa t},\qquad\kappa\in[0,1].
\end{split}
\end{equation}
In other words, the functions $b_\ast(t,\xi)$ are obtained by averaging $w_\ast(t,\rho)$ over intervals of length $L_\kappa(t,\xi)$ around the point $\xi$. The length $L_\kappa(t,\xi)$ in \eqref{dor1} is chosen to optimize the smoothness in $\xi$ of the functions $b_\ast(t,.)$, while not changing significantly the size of the weights. The parameter $\kappa$ is to be taken sufficiently small, depending only on $\delta$; it will be fixed in the proof of Lemma \ref{lm:CDW}, in such a way that the bounds \eqref{dor51} hold.

We can now finally define our main weights $A_{NR}$, $A_R$, and $A_k$. We define first the decreasing function $\lambda:[0,\infty)\to[\delta_0,3\delta_0/2]$ by
\begin{equation}\label{dor2}
\lambda(0)=\frac{3}{2}\delta_0,\,\,\,\,\lambda'(t)=-\frac{\delta_0\sigma_0^2}{\langle t\rangle^{1+\sigma_0}},
\end{equation}
for small positive constant $\sigma_0$ (say $\sigma_0=0.01$). Then we define
\begin{equation}\label{dor3}
A_R(t,\xi):=\frac{e^{\lambda(t)\langle\xi\rangle^{1/2}}}{b_R(t,\xi)}e^{\sqrt{\delta}\langle\xi\rangle^{1/2}},\qquad A_{NR}(t,\xi):=\frac{e^{\lambda(t)\langle\xi\rangle^{1/2}}}{b_{NR}(t,\xi)}e^{\sqrt{\delta}\langle\xi\rangle^{1/2}},
\end{equation}
and, for any $k\in\mathbb{Z}$,
\begin{equation}\label{dor4}
A_k(t,\xi):=e^{\lambda(t)\langle k,\xi\rangle^{1/2}}\Big(\frac{e^{\sqrt{\delta}\langle\xi\rangle^{1/2}}}{b_k(t,\xi)}+e^{\sqrt{\delta}|k|^{1/2}}\Big).
\end{equation}

\subsection{Properties of the weights} In this subsection we prove several bounds on the weights $w_\ast$, $b_\ast$, and $A_\ast$. We start with a lemma:

\begin{lemma}\label{comparisonweights}
For all $t\ge0$, $\xi,\,\eta\in \R$, and $k\in\mathbb{Z}$ we have
\begin{equation}\label{eq:comparisonweights1}
\frac{w_{NR}(t,\xi)}{w_{NR}(t,\eta)}+\frac{w_{R}(t,\xi)}{w_{R}(t,\eta)}+\frac{w_{k}(t,\xi)}{w_{k}(t,\eta)}\lesssim_\delta e^{\sqrt{\delta} |\eta-\xi|^{1/2}}.
\end{equation}
Moreover, if $L_1(t,\eta)$ is as in \eqref{dor1} and $|\xi-\eta|\leq 10L_1(t,\eta)$ then we have the stronger bounds
\begin{equation}\label{dor6}
\frac{w_{NR}(t,\xi)}{w_{NR}(t,\eta)}+\frac{w_{R}(t,\xi)}{w_{R}(t,\eta)}+\frac{w_{k}(t,\xi)}{w_{k}(t,\eta)}\lesssim_\delta 1.
\end{equation}
Finally, if $\min(|\xi|,|\eta|)\geq 2\delta^{-10}$, $|\xi-\eta|\leq \min(|\xi|,|\eta|)/3$, and $t\geq \max(t_{k_0(\xi)-4,\xi},t_{k_0(\eta)-4,\eta})$ then we also have the stronger bounds
\begin{equation}\label{ReSm2}
\max\Big\{\frac{w_{NR}(t,\xi)}{w_{NR}(t,\eta)},\,\frac{w_{R}(t,\xi)}{w_{R}(t,\eta)},\,\frac{w_{k}(t,\xi)}{w_{k}(t,\eta)}\Big\}\leq e^{\sqrt{\delta} |\eta-\xi|^{1/2}}.
\end{equation}
\end{lemma}

\begin{proof} The desired bounds follow easily if $|\eta|\leq \delta^{-10}$ since $w_{NR}(t,\eta)=w_k(t,\eta)=w_R(t,\eta)=1$ in these cases. Assume that $|\eta|>\delta^{-10}$. In view of the definitions we have $w_{\ast}(t,\eta)\geq e^{-\delta\sqrt{|\eta|}}$ for $\ast\in\{NR,R,k\}$, so the desired bounds follow if $|\xi-\eta|\geq 2\delta|\eta|$ (\eqref{dor6} is trivial in this case). They also follow if $|\xi-\eta|\leq 2\delta|\eta|$ and either $|\eta|<2\delta^{-10}$ or $t\geq 3|\eta|/2$. After these reductions, it remains to show that 
\begin{equation}\label{comp1}
\frac{w_{\ast}(t,\xi)}{w_{\ast}(t,\eta)}\lesssim_\delta e^{\sqrt{\delta} |\eta-\xi|^{1/2}},
\end{equation}
\begin{equation}\label{comp1ext}
\frac{w_{\ast}(t,\xi)}{w_{\ast}(t,\eta)}\leq e^{\sqrt{\delta} |\eta-\xi|^{1/2}}\qquad\text{ if }\,\,t\geq t_{k_0(\eta)-4,\eta},
\end{equation}
and
\begin{equation}\label{comp1.5}
\frac{w_{\ast}(t,\xi)}{w_{\ast}(t,\eta)}\lesssim_\delta 1\qquad\text{ if }\,\,|\xi-\eta|\leq 10 L_1(t,\eta)
\end{equation}
for $\ast\in\{NR,R,k\}$, provided that
\begin{equation}\label{vfc1}
\eta\geq 2\delta^{-10},\qquad t\leq 3\eta/2,\qquad |\xi-\eta|\leq 2\delta\eta.
\end{equation}

{\bf{Step 1: proof of \eqref{comp1}--\eqref{comp1ext}.}} Assume that $\eta\geq 2\delta^{-10}$ and $|\xi-\eta|\leq\eta/10$. We claim that
\begin{equation}\label{vfc1.1}
\frac{w_{NR}(t_{a',\xi},\xi)}{w_{NR}(t_{a,\eta},\eta)}\leq e^{\delta^{3/4}|\xi-\eta|^{1/2}}
\end{equation}
for all integers $a\in[1,k_0(\eta)]$ and $a'\in [1,k_0(\xi)]$ satisfying
\begin{equation}\label{vfc1.2}
a'\geq a-\frac{C\delta\sqrt{|\xi-\eta|}}{\ln(\delta^2\eta/a^2)}.
\end{equation}
Indeed, using \eqref{reb5}, we have
\begin{equation}\label{vfc1.3}
w_{NR}(t_{a,\rho},\rho)=\prod_{b=1}^a\Big(\frac{1}{1+\delta^2(t_{b-1,\rho}-\rho/b)}\Big)^{\delta_0}\Big(\frac{1}{1+\delta^2(\rho/b-t_{b,\rho})}\Big)^{1+\delta_0}
\end{equation}
for any $\rho>\delta^{-10}$ and $a\in[1,k_0(\rho)]$. We may assume $a'\leq a$ and estimate
\begin{equation*}
\begin{split}
\frac{w_{NR}(t_{a',\xi},\xi)}{w_{NR}(t_{a,\eta},\eta)}&=\prod_{b=1}^{a'}\Big(\frac{1+\delta^2(t_{b-1,\eta}-\eta/b)}{1+\delta^2(t_{b-1,\xi}-\xi/b)}\Big)^{\delta_0}\Big(\frac{1+\delta^2(\eta/b-t_{b,\eta})}{1+\delta^2(\xi/b-t_{b,\xi})}\Big)^{1+\delta_0}\\
&\times\prod_{b=a'+1}^a[1+\delta^2(t_{b-1,\eta}-\eta/b)]^{\delta_0}[1+\delta^2(\eta/b-t_{b,\eta})]^{1+\delta_0}\\
&\leq\prod_{b=1}^{a}(1+8\delta^2|\xi-\eta|/b^2)^2\times\prod_{b=a'+1}^a(1+8\delta^2\eta/b^2)^2,
\end{split}
\end{equation*}
and the desired bounds \eqref{vfc1.1} follow using also \eqref{reb8.5} and the assumption \eqref{vfc1.2}.

We divide the rest of the proof into several cases.

{\bf{Case 1.}} Assume that \eqref{vfc1} holds and, in addition, 
\begin{equation}\label{vfc2}
t\in[t_{a,\eta},t_{a-1,\eta}],\qquad a\in[1,k_0(\eta)-4],\qquad |\xi-\eta|\geq \eta/(100 a).
\end{equation}
Then $w_{\ast}(t,\eta)\geq w_{NR}(t_{a,\eta},\eta)$. In this case we will prove the stronger bounds
\begin{equation}\label{vfc4}
w_{NR}(t,\xi)\leq e^{\sqrt{\delta} |\eta-\xi|^{1/2}}w_{NR}(t_{a,\eta},\eta).
\end{equation}

Assume that $t=\xi/b$ for some $b\in[a/2,2a]$. Notice that 
\begin{equation}\label{vfc4.5}
 \left|\frac{\eta}{a}-\frac{\xi}{b}\right|\leq\frac{\eta}{a^2}\qquad\text{ and }\qquad \left|\frac{\eta}{a}-\frac{\xi}{b}\right|=\left|\frac{\eta-\xi}{a}+\frac{\xi(b-a)}{ab}\right|.
\end{equation}
Since $|\eta-\xi|/a\gtrsim \eta/a^2$ (see \eqref{vfc2}), we have
$$\left|\frac{\xi(b-a)}{ba}\right|\lesssim \frac{|\xi-\eta|}{a}.$$
Hence $|b-a|\lesssim a|\xi-\eta|/\eta$, so there is $b_0\in[a-Ca|\xi-\eta|/\eta,a]$ such that $t\leq t_{b_0,\xi}$ and $b_0\leq k_0(\xi)$. Using \eqref{vfc1.1} we have
\begin{equation*}
w_{NR}(t_{b_0,\xi},\xi)\leq e^{\delta^{3/4} |\eta-\xi|^{1/2}}w_{NR}(t_{a,\eta},\eta).
\end{equation*}
The desired bounds \eqref{vfc4} follow since $w_{NR}(.,\xi)$ is increasing.

{\bf{Case 2.}} Assume that \eqref{vfc1} holds and, in addition, 
\begin{equation}\label{vfc8}
t\in[t_{a,\eta},t_{a-1,\eta}],\qquad a\in[1,k_0(\eta)-4],\qquad |\xi-\eta|\leq \eta/(100 a).
\end{equation}
Let $t'$ be such that $t'-\xi/a=t-\eta/a$, so $t'-t=(\xi-\eta)/a$. 
%Suppose that $t'\in I_{a',\xi}$ for some $a'\in [1,k_0(\xi)-2]$ with $|a'-a|\leq 1$. 
For $\ast\in\{NR,R,k\}$ we write
\begin{equation}\label{vfc8.1}
\frac{w_*(t,\xi)}{w_*(t,\eta)}=\frac{w_*(t,\xi)}{w_*(t',\xi)}\cdot\frac{w_*(t',\xi)}{w_*(t,\eta)}.
\end{equation}
In view of \eqref{reb8} and recalling the assumptions \eqref{vfc8}, we have
\begin{equation}\label{vfc8.2}
\frac{w_*(t,\xi)}{w_*(t',\xi)}\leq e^{C\ln(1+\delta^2|t-t'|)}\leq e^{C\ln(1+\delta^2|\xi-\eta|/a)}.
\end{equation}

In view of \eqref{vfc1.1}, for \eqref{comp1ext} it suffices to prove that
\begin{equation}\label{vfc8.3}
\frac{w_*(t',\xi)}{w_*(t,\eta)}\leq e^{\sqrt{\delta}|\xi-\eta|^{1/2}/2}.
\end{equation}
This follows directly from the definition \eqref{reb5} and the bounds \eqref{vfc1.1} if $\ast=NR$. The bounds follow also if $\ast=R$ or if $\ast =k$, using \eqref{reb5.5}, \eqref{eq:resonantweight}, and considering two cases, $|t-\eta/a|\geq \eta/(6a^2)$ and $|t-\eta/a|\leq \eta/(6a^2)$ (in this last case we necessarily have $|t'-\xi/a|\leq \xi/(5a^2)$).

{\bf{Case 3.}}  Assume that \eqref{vfc1} holds and, in addition, 
\begin{equation}\label{vfc10}
0\leq t< t_{k_0(\eta)-4,\eta},\qquad 0\leq t\leq t_{k_0(\xi),\xi}.
\end{equation}
It suffices to prove the bounds \eqref{comp1}, in the slightly stronger form
\begin{equation}\label{vfc10.1}
\frac{w_{NR}(t,\xi)}{w_{NR}(t,\eta)}\lesssim_\delta e^{\sqrt{\delta} |\eta-\xi|^{1/2}/2}.
\end{equation}

Indeed, it follows from the definition \eqref{reb9} and \eqref{vfc1.1} that
\begin{equation}\label{vfc10.2}
\frac{w_{NR}(0,\xi)}{w_{NR}(0,\eta)}+\frac{w_{NR}(t_{k_0(\xi),\xi},\xi)}{w_{NR}(t_{k_0(\eta),\eta},\eta)}\lesssim e^{\delta^{3/4}|\xi-\eta|^{1/2}}.
\end{equation}
We define $t''\in[0,t_{k_0(\eta),\eta}]$ by the formulas
\begin{equation*}
t'':=(1-\beta)t_{k_0(\eta),\eta}\qquad\text{ if }\,t=(1-\beta)t_{k_0(\xi),\xi},\,\beta\in[0,1].
\end{equation*}
Using the definitions \eqref{reb9} and the bounds \eqref{vfc10.2}, we have
\begin{equation}\label{vfc10.3}
\frac{w_{NR}(t,\xi)}{w_{NR}(t'',\eta)}\lesssim e^{\delta^{3/4}|\xi-\eta|^{1/2}},
\end{equation}
for any $t\in[0,t_{k_0(\xi),\xi}]$. Moreover, $t''\in[0,t_{k_0(\eta),\eta}]$ and $|t''-t|\lesssim \delta^{-3/2}|\xi-\eta|/\sqrt\eta+\delta^{-3}$. The bounds \eqref{vfc10.1} follow using also the bounds \eqref{reb9.5}.

{\bf{Case 4.}}  Finally, assume that \eqref{vfc1} holds and, in addition, 
\begin{equation}\label{vfc14}
0\leq t< t_{k_0(\eta)-4,\eta},\qquad t\geq t_{k_0(\xi),\xi}.
\end{equation}
Notice that
\begin{equation*}
t_{k_0(\xi),\xi}=\delta^{-3/2}\sqrt{\eta+(\xi-\eta)}+O(\delta^{-3})\geq \delta^{-3/2}\sqrt\eta-\delta^{-3/2}|\xi-\eta|/\sqrt{\eta}-C\delta^{-3}.
\end{equation*}
Therefore, using \eqref{reb4}, we have
\begin{equation*}
w_\ast(t,\xi)\lesssim_\delta w_{NR}(t_{k_0(\xi),\xi},\xi)\cdot e^{C\ln\delta^{-1}\cdot\delta^{3/2}|\xi-\eta|/\sqrt\eta}\lesssim_\delta w_{NR}(t_{k_0(\xi),\xi},\xi)\cdot e^{\delta\sqrt{|\xi-\eta|}}.
\end{equation*}
The desired conclusion follows from the bounds \eqref{vfc10.1} when $t=t_{k_0(\xi),\xi}$ proved in {\bf{Case 3}}.

{\bf{Step 2: proof of \eqref{comp1.5}.}} We can proceed along the same line, but the proof is easier. We claim that if $\eta\geq 2\delta^{-10}$ and $|\xi-\eta|\leq 10\sqrt\eta$ then 
\begin{equation}\label{dor8}
\frac{w_{NR}(t_{a',\xi},\xi)}{w_{NR}(t_{a,\eta},\eta)}\lesssim (\delta^2\eta/a^2)^{\max(0,a-a')}
\end{equation}
for all integers $a\in[1,k_0(\eta)]$ and $a'\in [1,k_0(\xi)]$ satisfying $|a-a'|\leq 10$. This is similar to the proof of \eqref{vfc1.1}, using again \eqref{vfc1.3}. As before, we consider several cases. 

{\bf{Case 1.}} Assume that 
\begin{equation}\label{dor9}
t\in[t_{a,\eta},t_{a-1,\eta}],\qquad a\in[1,k_0(\eta)-4],\qquad |\xi-\eta|\leq 10L_1(t,\eta).
\end{equation}
This is similar to {\bf{Case 2}} in the proof of \eqref{comp1}. Let $t'$ be such that $t'-\xi/a=t-\eta/a$, so $t'-t=(\xi-\eta)/a$. 
%Clearly, $t,t'\in I_{a',\xi}$ for some $a'\in [1,k_0(\xi)-2]$ with $|a'-a|\leq 1$. 
For $\ast\in\{NR,R,k\}$ we write
\begin{equation}\label{dor10}
\frac{w_*(t,\xi)}{w_*(t,\eta)}=\frac{w_*(t,\xi)}{w_*(t',\xi)}\cdot\frac{w_*(t',\xi)}{w_*(t,\eta)}.
\end{equation}
In view of \eqref{reb8} and the assumptions \eqref{dor9} (which imply that $|\xi-\eta|\lesssim 1+a$), we have
\begin{equation}\label{dor11}
\frac{w_*(t,\xi)}{w_*(t',\xi)}\lesssim e^{C\ln(1+\delta^2|t-t'|)}\lesssim e^{C\ln(1+\delta^2|\xi-\eta|/a)}\lesssim 1.
\end{equation}
Moreover, as in the proof of \eqref{vfc8.3},
\begin{equation}\label{dor12}
\frac{w_*(t',\xi)}{w_*(t,\eta)}\lesssim \frac{w_{NR}(t_{a,\xi},\xi)}{w_{NR}(t_{a,\eta},\eta)}+\frac{w_{NR}(t_{a-1,\xi},\xi)}{w_{NR}(t_{a-1,\eta},\eta)}\lesssim 1,
\end{equation}
where the last inequality follows from \eqref{dor8}. The desired bounds \eqref{comp1.5} follow in this case.

{\bf{Case 2.}}  Assume now that
\begin{equation}\label{dor14}
0\leq t\leq t_{k_0(\eta)-4,\eta},\qquad |\xi-\eta|\leq 10L_1(t,\eta).
\end{equation}
It follows from the definition \eqref{reb9} and \eqref{dor8} that
\begin{equation}\label{dor15}
\frac{w_{NR}(0,\xi)}{w_{NR}(0,\eta)}+\frac{w_{NR}(t_{k_0(\xi),\xi},\xi)}{w_{NR}(t_{k_0(\eta),\eta},\eta)}\lesssim_\delta 1.
\end{equation}

If $t\leq t_{k_0(\xi),\xi}$ then we define $t''\in[0,t_{k_0(\eta),\eta}]$ by the formulas
\begin{equation*}
t'':=(1-\beta)t_{k_0(\eta),\eta}\qquad\text{ if }\,t=(1-\beta)t_{k_0(\xi),\xi},\,\beta\in[0,1].
\end{equation*}
Using the definitions \eqref{reb9} and the bounds \eqref{dor15}, we have $\frac{w_{NR}(t,\xi)}{w_{NR}(t'',\eta)}\lesssim_\delta 1$ 
for any $t\in[0,t_{k_0(\xi),\xi}]$. Moreover, $t''\in[0,t_{k_0(\eta),\eta}]$ and $|t''-t|\lesssim\delta^{-3}$. Using also \eqref{reb9.5} it follows that
\begin{equation}\label{dor16}
\frac{w_{NR}(t,\xi)}{w_{NR}(t,\eta)}\lesssim_\delta 1,\qquad\text{ if }\,\,t\leq t_{k_0(\xi),\xi}.
\end{equation}

On the other hand, if  $t\geq t_{k_0(\xi),\xi}$ then we use \eqref{reb4} to see that $w_\ast(t,\xi)\lesssim_\delta w_{NR}(t_{k_0(\xi),\xi},\xi)$. The desired bounds \eqref{comp1.5} follow for all $t\leq  t_{k_0(\eta)-4,\eta}$, using also \eqref{dor16} when $t=t_{k_0(\xi),\xi}$. This completes the proof of the lemma.
\end{proof}

We prove now estimates on the functions $b_\ast$ defined in \eqref{dor1}.

\begin{lemma}\label{bweights}
(i) For $t\geq 0$, $\xi\in\R$, $k\in\Z$, and $\ast\in\{NR,R,k\}$ we have
\begin{equation}\label{dor20}
b_\ast(t,\xi)\approx_\delta w_\ast(t,\xi),
\end{equation}
\begin{equation}\label{dor21}
|\partial_\xi b_\ast(t,\xi)|\lesssim_\delta b_\ast(t,\xi)\frac{1}{L_\kappa(t,\xi)},
\end{equation}
\begin{equation}\label{dor22}
\frac{b_\ast(t,\xi)}{b_{\ast}(t,\eta)}\lesssim_\delta e^{\sqrt{\delta}|\eta-\xi|^{1/2}}.
\end{equation}

(ii) For $t\geq 0$ let
\begin{equation}\label{dor23.1}
\begin{split}
I^\ast_t&:=\{(k,\xi)\in\mathbb{Z}\times\mathbb{R}:\,1\leq |k|\leq \delta^2t\,\text{ and }\,|\xi-tk|\leq t/6\},\\
I^{\ast\ast}_t&:=\{(k,\xi)\in\mathbb{Z}\times\mathbb{R}:\,1\leq |k|\leq \delta^4t\,\text{ and }\,|\xi-tk|\leq t/12\}.
\end{split}
\end{equation}
Then
\begin{equation}\label{dor23.2}
w_k(t,\xi)=w_{NR}(t,\xi)\quad\text{ and }\quad b_k(t,\xi)=b_{NR}(t,\xi)\qquad\text{ if }(k,\xi)\notin I^\ast_t,
\end{equation}
\begin{equation}\label{dor23.3}
b_k(t,\xi)\approx_\delta w_k(t,\xi)\approx_\delta w_R(t,\xi)\approx_\delta b_{R}(t,\xi)\qquad\text{ if }(k,\xi)\in I^\ast_t,
\end{equation}
\begin{equation}\label{dor23.4}
b_k(t,\xi)\approx_\delta w_k(t,\xi)\approx_\delta w_{NR}(t,\xi)\approx_\delta b_{NR}(t,\xi)\qquad\text{ if }(k,\xi)\notin I^{\ast\ast}_t.
\end{equation}
\end{lemma}

\begin{proof}  (i) The bounds \eqref{dor20} follow from the definition \eqref{dor1} and the bounds \eqref{dor6}. To prove \eqref{dor21} we start again from the definition \eqref{dor1}, take $\partial_\xi$ derivatives,  notice that $|\partial_\xi L_\kappa(\xi,t)|\lesssim 1$ and use \eqref{dor6} again. The bounds \eqref{dor22} follow using also \eqref{eq:comparisonweights1}.

(ii) The identities \eqref{dor23.2} follow easily from definitions \eqref{reb5.5}, \eqref{eq:resonantweight}, and \eqref{dor1}. In view of \eqref{dor20}, for \eqref{dor23.3} it suffices to prove that $w_k(t,\xi)\approx_\delta w_R(t,\xi)$ if $(k,\xi)\in I^\ast_t$, which follows from definitions again (notice that $w_{NR}(t,\rho)\approx_\delta w_R(t,\rho)$ if $t\lesssim_\delta |\rho|^{1/2}$).  Finally, for \eqref{dor23.4} it suffices to show that $w_k(t,\xi)\approx_\delta w_{NR}(t,\xi)$ if $(k,\xi)\notin I^{\ast\ast}_t$, which follows from definitions again.
\end{proof}

We prove now several bounds on the main weights $A_{NR}, A_R, A_k$ defined in \eqref{reb11}--\eqref{reb12}.

\begin{lemma}\label{A_kA_ell}
(i) Assume $t\in[0,\infty)$, $k\in\mathbb{Z}$, and $\ast\in\{NR,R,k\}$. Then, for any $\xi,\eta\in\mathbb{R}$ satisfying $|\eta|\geq |\xi|/8$ (or $|(k,\eta)|\geq|(k,\xi)|/8$ if $\ast=k$), we have
\begin{equation}\label{vfc25}
\frac{A_\ast(t,\xi)}{A_\ast(t,\eta)}\lesssim_\delta e^{0.9\lambda(t)|\xi-\eta|^{1/2}}.
\end{equation}

(ii) Assume $t\in[0,\infty)$, $k,\ell\in\mathbb{Z}$ and $\xi,\eta\in\mathbb{R}$ satisfy $|(\ell,\eta)|\geq |(k,\xi)|/8$. If $t\not\in I_{k,\xi}$ or if $t\in I_{k,\xi}\cap I_{\ell,\eta}$, then
\begin{equation}\label{vfc26}
\frac{A_k(t,\xi)}{A_\ell(t,\eta)}\lesssim_\delta e^{0.9\lambda(t)|(k-\ell,\xi-\eta)|^{1/2}}.
\end{equation}
If $t\in I_{k,\xi}$ and $t\not\in I_{\ell,\eta}$, then
\begin{equation}\label{vfc27}
\frac{A_k(t,\xi)}{A_\ell(t,\eta)}\lesssim_\delta \frac{|\xi|}{k^2}\frac{1}{1+\big|t-\xi/k\big|} e^{0.9\lambda(t)|(k-\ell,\xi-\eta)|^{1/2}}.
\end{equation}
\end{lemma}

\begin{proof} If $\ast\in\{NR,R\}$ then the bounds \eqref{vfc25} follow directly from the definitions \eqref{dor2}--\eqref{dor3}, the elementary bounds \eqref{b>a} and the bounds \eqref{dor22}. 

To prove the remaining bounds we start from the definition \eqref{dor4} and estimate
\begin{equation}\label{vfc28.1}
\frac{A_k(t,\xi)}{A_\ell(t,\eta)}\leq \frac{e^{\lambda(t)\langle k,\xi\rangle^{1/2}+\sqrt\delta\langle\xi\rangle^{1/2}}}{e^{\lambda(t)\langle \ell,\eta\rangle^{1/2}+\sqrt\delta\langle\eta\rangle^{1/2}}}\frac{b_\ell(t,\eta)}{b_k(t,\xi)}+\frac{e^{\lambda(t)\langle k,\xi\rangle^{1/2}+\sqrt\delta|k|^{1/2}}}{e^{\lambda(t)\langle \ell,\eta\rangle^{1/2}+\sqrt\delta|\ell|^{1/2}}}.
\end{equation}
The bounds \eqref{vfc25} follow when $\ast=k$, using again \eqref{b>a} and \eqref{dor22}. Since $w_R(t,\rho)\leq w_{NR}(t,\rho)$ and using also \eqref{dor20} we have 
\begin{equation}\label{q2w1}
\frac{b_\ell(t,\eta)}{b_k(t,\xi)}\lesssim_\delta\frac{w_\ell(t,\eta)}{w_k(t,\xi)}\lesssim_\delta \Big[\frac{w_{NR}(t,\eta)}{w_{NR}(t,\xi)}+\frac{w_R(t,\eta)}{w_R(t,\xi)}\Big]
\end{equation}
if $t\not\in I_{k,\xi}$ or if $t\in I_{k,\xi}\cap I_{\ell,\eta}$. The desired bounds \eqref{vfc26} follow using \eqref{b>a} and Lemma \ref{comparisonweights}. Moreover, if $t\in I_{k,\xi}$ and $t\not\in I_{\ell,\eta}$ then
\begin{equation}\label{q2w2}
\frac{b_\ell(t,\eta)}{b_k(t,\xi)}\lesssim_\delta\frac{w_\ell(t,\eta)}{w_k(t,\xi)}\lesssim_\delta\frac{w_{NR}(t,\eta)}{w_R(t,\xi)}\lesssim_\delta\frac{w_{NR}(t,\eta)}{w_{NR}(t,\xi)}\cdot\frac{|\xi|}{k^2}\frac{1}{1+\big|t-\xi/k\big|},
\end{equation}
using \eqref{reb7}. The bounds \eqref{vfc27} follow using again \eqref{b>a} and \eqref{eq:comparisonweights1}.
\end{proof}

We also need a lemma estimating time derivatives of the weights $A_\ast$. 

\begin{lemma}\label{lm:CDW}
(i) For all $t\ge 0,$ $\rho\in\mathbb{R}$, and $\ast\in\{NR,R\}$ we have
\begin{equation}\label{TLX3.5}
-\frac{\partial_tA_\ast(t,\rho)}{A_\ast(t,\rho)}\approx_\delta\left[\frac{\langle\rho\rangle^{1/2}}{\langle t\rangle^{1+\sigma_0}}+\frac{\partial_tw_\ast(t,\rho)}{w_\ast(t,\rho)}\right],
\end{equation}
and, for any $k\in\Z$,
\begin{equation}\label{eq:A_kxi}
-\frac{\partial_tA_k(t,\rho)}{A_k(t,\rho)}\approx_\delta\left[\frac{\langle k,\rho\rangle^{1/2}}{\langle t\rangle^{1+\sigma_0}}+\frac{\partial_tw_k(t,\rho)}{w_k(t,\rho)}\frac{1}{1+e^{\sqrt\delta(|k|^{1/2}-\langle\rho\rangle^{1/2})}w_k(t,\rho)}\right].
\end{equation}

(ii) For all $t\ge 0,$ $\xi,\,\eta\in\mathbb{R}$, and $\ast\in\{NR,R\}$ we have
\begin{equation}\label{vfc30}
\big|(\dot{A}_\ast/A_{\ast})(t,\xi)\big|\lesssim_\delta \big|(\dot{A}_\ast/A_{\ast})(t,\eta)\big|e^{4\sqrt{\delta}|\xi-\eta|^{1/2}}.
\end{equation}
Moreover, if $k,\ell\in\mathbb{Z}$ then
\begin{equation}\label{eq:CDW}
\big|(\dot{A}_k/A_k)(t,\xi)\big|\lesssim_\delta \big|(\dot{A}_\ell/A_{\ell})(t,\eta)\big|e^{4\sqrt{\delta}|k-\ell,\xi-\eta|^{1/2}}.
\end{equation}
Finally, if $\rho\in\mathbb{R}$ and $k\in\mathbb{Z}$ satisfy $|k|\leq\langle\rho\rangle+10$ then
\begin{equation}\label{vfc30.5}
\big|(\dot{A}_k/A_k)(t,\rho)\big|\approx_\delta\big|(\dot{A}_{NR}/A_{NR})(t,\rho)\big|\approx_\delta\big|(\dot{A}_R/A_R)(t,\rho)\big|.
\end{equation}
\end{lemma}

\begin{proof} Using the definitions we calculate, for $\rho\in\R$ and $\ast\in\{NR,R\}$,
\begin{equation}\label{dor30}
-\frac{\partial_tA_\ast(t,\rho)}{A_\ast(t,\rho)}=\frac{\delta_0\sigma_0^2\langle\rho\rangle^{1/2}}{\langle t\rangle^{1+\sigma_0}}+\frac{\partial_tb_\ast(t,\rho)}{b_\ast(t,\rho)}
\end{equation}
and, for any $k\in\Z$,
\begin{equation}\label{dor31}
-\frac{\partial_tA_k(t,\rho)}{A_k(t,\rho)}=\frac{\delta_0\sigma_0^2\langle k,\rho\rangle^{1/2}}{\langle t\rangle^{1+\sigma_0}}+\frac{\partial_tb_k(t,\rho)}{b_k(t,\rho)}\frac{e^{\sqrt\delta\langle\rho\rangle^{1/2}}}{e^{\sqrt\delta\langle\rho\rangle^{1/2}}+e^{\sqrt\delta|k|^{1/2}}b_k(t,\rho)}.
\end{equation}

It follows from \eqref{reb8} that, for any $k\in\Z$, $t\geq 0$, and $\rho\in\R$ we have
\begin{equation}\label{TLX2.9}
\left|\frac{\partial_tw_k(t,\rho)}{w_k(t,\rho)}\right|\approx \left|\frac{\partial_tw_{NR}(t,\rho)}{w_{NR}(t,\rho)}\right|\approx \left|\frac{\partial_tw_R(t,\rho)}{w_R(t,\rho)}\right|.
\end{equation}
We divide the proof of the lemma in several steps.

{\bf{Step 1.}} We show first that for any $t\geq 0$ and $\xi,\eta\in\R$ we have
\begin{equation}\label{TLX3}
\frac{\partial_tw_{NR}(t,\xi)}{w_{NR}(t,\xi)}+\frac{\langle\xi\rangle^{1/2}}{\langle t\rangle^{1+\sigma_0}}\lesssim_\delta \left[\frac{\partial_tw_{NR}(t,\eta)}{w_{NR}(t,\eta)}+\frac{\langle\eta\rangle^{1/2}}{\langle t\rangle^{1+\sigma_0}}\right]e^{\sqrt{\delta}|\xi-\eta|^{1/2}}.
\end{equation}

Indeed, the term $\langle\xi\rangle^{1/2}\langle t\rangle^{-1-\sigma_0}$ is clearly controlled as claimed (see \eqref{eq1-1} below). The first term in the left-hand side vanishes if $t\geq 2|\xi|$ or $|\xi|\leq\delta^{-10}$. On the other hand, if $t\leq 2|\xi|$ and $|\xi|>\delta^{-10}$ then this term is $\lesssim 1$ (see \eqref{reb8}), and the inequality is clear if $|\xi-\eta|$ is large. After these reductions, we have to prove that
\begin{equation}\label{dor33.5}
\frac{\partial_tw_{NR}(t,\xi)}{w_{NR}(t,\xi)}\lesssim_\delta \left[\frac{\partial_tw_{NR}(t,\eta)}{w_{NR}(t,\eta)}+\frac{\langle\eta\rangle^{1/2}}{\langle t\rangle^{1+\sigma_0}}\right]e^{\sqrt{\delta}|\xi-\eta|^{1/2}},
\end{equation}
provided that
\begin{equation}\label{dor33}
\xi>\delta^{-10}, \qquad |\xi-\eta|\leq \delta \xi,\qquad t\leq 2\xi.
\end{equation}

In view of \eqref{reb8} the left-hand side of \eqref{dor33.5} is $\lesssim_\delta\langle t\rangle^{-1}$ if $t\geq 3\xi/2$, and the bound follows easily. Also, using again \eqref{reb8},
\begin{equation}\label{dor34}
\frac{\partial_tw_{NR}(t,\rho)}{w_{NR}(t,\rho)}\approx_\delta 1\qquad\text{ if }|\rho|>\delta^{-10}\text{ and }t\leq \delta^{-6}|\rho|^{1/2},
\end{equation}
and \eqref{dor33.5} follows if $t\leq \delta^{-5}|\xi|^{1/2}$. After these further reductions, it remains to prove \eqref{dor33.5} under the stronger assumptions
\begin{equation}\label{dor35}
\xi>\delta^{-10}, \qquad |\xi-\eta|\leq \delta^2 \sqrt\xi,\qquad t\in I_{a,\eta}\cap I_{b,\xi},\,a,b\leq \delta^2\sqrt\xi.
\end{equation}

It follows from \eqref{reb8} that 
\begin{equation}\label{vfc21.1}
\frac{\partial_tw_{NR}(t,\xi)}{w_{NR}(t,\xi)}\left[\frac{\partial_tw_{NR}(t,\eta)}{w_{NR}(t,\eta)}\right]^{-1}\lesssim \frac{1+\delta^2\left|t-\eta/a\right|}{1+\delta^2\left|t-\xi/b\right|}.
\end{equation}
As in the proof of Lemma \ref{comparisonweights}, if $a\neq b$ then simple arguments show that $\left|t-\eta/a\right|\approx\eta/a^2$ and $\left|t-\xi/b\right|\approx \eta/a^2$, and \eqref{dor33.5} follows from \eqref{vfc21.1}. On the other hand, if $a=b$ then
\begin{eqnarray}\label{dor35.1} 
\frac{1+\delta^2\left|t-\eta/a\right|}{1+\delta^2\left|t-\xi/b\right|}\lesssim 1+\delta^2\left|\frac{\eta-\xi}{a}\right|\lesssim e^{\sqrt\delta|\eta-\xi|^{1/2}}.
\end{eqnarray}
This completes the proof of \eqref{TLX3}.

{\bf{Step 2.}} We show now that if $t\geq 0$ and $\xi,\eta\in\R$ satisfy $|\xi-\eta|\leq 10L_1(t,\eta)$ then we have the stronger bounds
\begin{equation}\label{dor40}
\frac{\partial_tw_{NR}(t,\xi)}{w_{NR}(t,\xi)}+\frac{\langle\xi\rangle^{1/2}}{\langle t\rangle^{1+\sigma_0}}\lesssim_\delta \left[\frac{\partial_tw_{NR}(t,\eta)}{w_{NR}(t,\eta)}+\frac{\langle\eta\rangle^{1/2}}{\langle t\rangle^{1+\sigma_0}}\right].
\end{equation}

This is similar to the proof of \eqref{TLX3}. The term $\langle\xi\rangle^{1/2}\langle t\rangle^{-1-\sigma_0}$ is clearly controlled as claimed. The first term in the left-hand side vanishes if $t\geq 2|\xi|$ or $|\xi|\leq\delta^{-10}$. It remains to prove that
\begin{equation}\label{dor41}
\frac{\partial_tw_{NR}(t,\xi)}{w_{NR}(t,\xi)}\lesssim_\delta \left[\frac{\partial_tw_{NR}(t,\eta)}{w_{NR}(t,\eta)}+\frac{\langle\eta\rangle^{1/2}}{\langle t\rangle^{1+\sigma_0}}\right],
\end{equation}
provided that
\begin{equation}\label{dor42}
\xi>\delta^{-10}, \qquad |\xi-\eta|\leq 10 L_1(t,\eta),\qquad t\leq 2\xi.
\end{equation}

Using \eqref{reb8} (see also \eqref{dor34}), the bounds \eqref{dor41} follow if $t\geq 3\xi/2$ or if $t\leq\delta^{-4}\langle\xi\rangle^{1/2}$. In the remaining range $t\in[\delta^{-4}\langle\xi\rangle^{1/2},3\xi/2]$, we may assume that $t\in I_{a,\eta}\cap I_{b,\xi}$ for some $a,b\in[1,\delta^2\sqrt\xi]$. If $a=b$ then the bounds \eqref{dor35.1} still apply, and the desired conclusion follows once we notice that $|\xi-\eta|\lesssim L_1(t,\eta)\lesssim_\delta a$.  On the other hand, if $a\neq b$ then $\left|t-\eta/a\right|\approx\eta/a^2\approx \left|t-\xi/b\right|\approx \eta/a^2$, and the desired conclusion \eqref{dor41} follows as before.

{\bf{Step 3.}} We show now that if $\alpha\in[\delta,1]$, $t\geq 0$, $k\in\Z$, $\ast\in\{NR,R,k\}$, and $\rho\in\R$ then{\footnote{One needs to be slightly careful here, since $\partial_tb_\ast$ is not necessarily positive, and the expression in the left-hand side of \eqref{dor50} is only positive after adding the second term.}}
\begin{equation}\label{dor50}
 \left[\frac{\partial_tb_{\ast}(t,\rho)}{b_{\ast}(t,\rho)}+\frac{\alpha\langle\rho\rangle^{1/2}}{\langle t\rangle^{1+\sigma_0}}\right]\approx_\delta\left[\frac{\partial_tw_{\ast}(t,\rho)}{w_{\ast}(t,\rho)}+\frac{\langle\rho\rangle^{1/2}}{\langle t\rangle^{1+\sigma_0}}\right].
\end{equation}
Indeed, starting from the definitions \eqref{dor1} we write
\begin{equation*}
\frac{\partial_tb_{\ast}(t,\rho)}{b_{\ast}(t,\rho)}+\frac{\alpha\langle\rho\rangle^{1/2}}{\langle t\rangle^{1+\sigma_0}}=I+II,
\end{equation*}
where
\begin{equation*}
I:=\frac{1}{b_{\ast}(t,\rho)}\int_\R (\partial_tw_\ast)(t,\mu)\varphi\Big(\frac{\rho-\mu}{L_\kappa(t,\rho)}\Big)\frac{1}{d_0L_\kappa(t,\rho)}\,d\mu+\frac{\alpha\langle\rho\rangle^{1/2}}{2\langle t\rangle^{1+\sigma_0}}
\end{equation*}
and
\begin{equation*}
II:=\frac{1}{b_{\ast}(t,\rho)}\int_\R w_\ast(t,\mu)\frac{d}{dt}\Big\{\varphi\Big(\frac{\rho-\mu}{L_\kappa(t,\rho)}\Big)\frac{1}{d_0L_\kappa(t,\rho)}\Big\}\,d\mu+\frac{\alpha\langle\rho\rangle^{1/2}}{2\langle t\rangle^{1+\sigma_0}}
\end{equation*}

It follows from \eqref{dor20}, \eqref{dor40}, and \eqref{TLX2.9} that
\begin{equation*}
I\approx_\delta\Big[\frac{\partial_tw_{\ast}(t,\rho)}{w_{\ast}(t,\rho)}+\frac{\langle\rho\rangle^{1/2}}{\langle t\rangle^{1+\sigma_0}}\Big].
\end{equation*}
Therefore, for \eqref{dor50} it suffices to show that
\begin{equation}\label{dor51}
0\leq \frac{1}{b_{\ast}(t,\rho)}\int_\R w_\ast(t,\mu)\frac{d}{dt}\Big\{\varphi\Big(\frac{\rho-\mu}{L_\kappa(t,\rho)}\Big)\frac{1}{d_0L_\kappa(t,\rho)}\Big\}\,d\mu+\frac{\alpha\langle\rho\rangle^{1/2}}{2\langle t\rangle^{1+\sigma_0}}\leq\frac{\alpha\langle\rho\rangle^{1/2}}{\langle t\rangle^{1+\sigma_0}}.
\end{equation}

It follows from \eqref{dor1} that, for $\kappa>0$ sufficiently small,
\begin{equation*}
-\frac{\partial_tL_\kappa(t,\rho)}{L_\kappa(t,\rho)}=\frac{\kappa^2\langle\rho\rangle}{(\langle\rho\rangle^{1/2}+\kappa t)(\langle\rho\rangle^{1/2}+\kappa t+\kappa\langle\rho\rangle)}\leq\frac{\kappa\langle\rho\rangle}{\langle\rho\rangle^{3/2}+\kappa t^2}\leq\frac{\kappa^{1/4}\langle\rho\rangle^{1/2}}{\langle t\rangle^{1+\sigma_0}}.
\end{equation*}
Therefore, using also \eqref{dor20} and \eqref{dor6},
\begin{equation*}
\Big|\frac{1}{b_{\ast}(t,\rho)}\int_\R w_\ast(t,\mu)\frac{d}{dt}\Big\{\varphi\Big(\frac{\rho-\mu}{L_\kappa(t,\rho)}\Big)\frac{1}{d_0L_\kappa(t,\rho)}\Big\}\,d\mu\Big|\lesssim_\delta \Big|\frac{\partial_tL_\kappa(t,\rho)}{L_\kappa(t,\rho)}\Big|\lesssim_\delta\frac{\kappa^{1/4}\langle\rho\rangle^{1/2}}{\langle t\rangle^{1+\sigma_0}}.
\end{equation*}
Since $\alpha\in[\delta,1]$, the desired bounds \eqref{dor51} follow if $\kappa=\kappa(\delta)$ is fixed sufficiently small.

{\bf{Step 4.}} We can now prove the bounds in the lemma. Indeed, the bounds \eqref{TLX3.5} follow directly from the identities \eqref{dor30} and the bounds \eqref{dor50} with $\alpha=\delta_0\sigma_0^2$. The bounds \eqref{vfc30} follow from the bounds \eqref{TLX3.5}, \eqref{TLX2.9}, and \eqref{TLX3}.

We prove now the bounds \eqref{eq:A_kxi}. Let $X_k(t,\rho):=e^{\sqrt\delta(|k|^{1/2}-\langle\rho\rangle^{1/2})}b_k(t,\rho)$ and write
\begin{equation*}
\begin{split}
-\frac{\partial_tA_k(t,\rho)}{A_k(t,\rho)}&=\frac{\delta_0\sigma_0^2\langle k,\rho\rangle^{1/2}}{\langle t\rangle^{1+\sigma_0}}+\frac{\partial_tb_k(t,\rho)}{b_k(t,\rho)}\frac{1}{1+X_k(t,\rho)}\\
&=\frac{1}{1+X_k(t,\rho)}\left[\frac{\partial_tb_k(t,\rho)}{b_k(t,\rho)}+\frac{\delta_0\sigma_0^2\langle \rho\rangle^{1/2}}{2\langle t\rangle^{1+\sigma_0}}\right]+\frac{\delta_0\sigma_0^2}{\langle t\rangle^{1+\sigma_0}}\left[\langle k,\rho\rangle^{1/2}-\frac{\langle \rho\rangle^{1/2}}{2(1+X_k(t,\rho))}\right],
\end{split}
\end{equation*}
using \eqref{dor31}. Both terms in the expression above are positive, and the desired bounds \eqref{eq:A_kxi} follow using \eqref{dor50} and the fact that $X_k(t,\rho)\approx_\delta e^{\sqrt\delta(|k|^{1/2}-\langle\rho\rangle^{1/2})}w_k(t,\rho)$ (see \eqref{dor20}).

The bounds \eqref{vfc30.5} follow from \eqref{TLX3.5}--\eqref{eq:A_kxi} and the bounds \eqref{TLX2.9}. The condition $|k|\leq\langle\rho\rangle$ guarantees that $e^{\sqrt\delta(|k|^{1/2}-\langle\rho\rangle^{1/2})}w_k(t,\rho)\lesssim 1$. 

Finally, we prove \eqref{eq:CDW}. Notice that
\begin{equation}\label{eq1-1}
\frac{\langle k,\xi\rangle^{1/2}}{\langle t\rangle^{1+\sigma_0}}\lesssim_\delta \frac{\langle \ell,\eta\rangle^{1/2}}{\langle t\rangle^{1+\sigma_0}}e^{\sqrt{\delta}|k-\ell,\xi-\eta|^{1/2}}
\end{equation}
for any $\xi,\eta\in\mathbb{R}$, $k,\ell\in\Z$, and $t\in[0,\infty)$. This suffices to control the first term in the left-hand side \eqref{eq:A_kxi}. To control the second term we prove first that
\begin{equation}\label{TLX1}
\frac{e^{\sqrt\delta\langle\xi\rangle^{1/2}}}{e^{\sqrt\delta\langle\xi\rangle^{1/2}}+e^{\sqrt\delta|k|^{1/2}}w_k(t,\xi)}\lesssim_\delta \frac{e^{\sqrt\delta\langle\eta\rangle^{1/2}}}{e^{\sqrt\delta\langle\eta\rangle^{1/2}}+e^{\sqrt\delta|\ell|^{1/2}}w_\ell(t,\eta)}e^{3\sqrt{\delta}|k-\ell,\xi-\eta|^{1/2}}.
\end{equation}
Indeed, in proving \eqref{TLX1} we may assume that $e^{\sqrt\delta\langle\eta\rangle^{1/2}}\leq e^{\sqrt\delta|\ell|^{1/2}}w_\ell(t,\eta)$. After simplifications it suffices to show that
\begin{equation}\label{TLX2}
\frac{e^{\sqrt\delta|\ell|^{1/2}}w_\ell(t,\eta)}{e^{\sqrt\delta\langle\xi\rangle^{1/2}}+e^{\sqrt\delta|k|^{1/2}}w_k(t,\xi)}\lesssim _\delta e^{2\sqrt{\delta}|k-\ell,\xi-\eta|^{1/2}}.
\end{equation}
Notice that, as a consequence of the definitions,
\begin{equation*}
e^{\sqrt\delta\langle\xi\rangle^{1/2}}+e^{\sqrt\delta|k|^{1/2}}w_k(t,\xi)\gtrsim e^{\sqrt\delta\langle\xi\rangle^{1/2}}+e^{\sqrt\delta|k|^{1/2}}w_{NR}(t,\xi).
\end{equation*}
The bounds \eqref{TLX2} follow using also Lemma \ref{comparisonweights}. The estimates \eqref{eq:CDW} now follow by combining \eqref{TLX3}, \eqref{TLX1}, and \eqref{eq:A_kxi}.
\end{proof}

\section{Weighted bilinear estimates}\label{BilinWeights} We often use the following general lemma to estimate products and paraproducts of functions.

\begin{lemma}\label{Multi0}
(i) Assume that $m,m_1,m_2:\R\to\C$ are symbols satisfying
\begin{equation}\label{TLX5}
|m(\xi)|\leq |m_1(\xi-\eta)|\,|m_2(\eta)|\{\langle\xi-\eta\rangle^{-2}+\langle\eta\rangle^{-2}\}
\end{equation}
for any $\xi,\eta\in\R$. If $M, M_1, M_2$ are the operators defined by these symbols then
\begin{equation}\label{TLX6}
\|M(gh)\|_{L^2(\R)}\lesssim \|M_1g\|_{L^2(\R)}\|M_2h\|_{L^2(\R)}.
\end{equation}

(ii) Similarly, if $m,m_2:\Z\times\R\to\C$ and $m_1:\R\to\C$ are symbols satisfying
\begin{equation}\label{TLX5.1}
|m(k,\xi)|\leq |m_1(\xi-\eta)|\,|m_2(k,\eta)|\{\langle\xi-\eta\rangle^{-2}+\langle k,\eta\rangle^{-2}\}
\end{equation}
for any $\xi,\eta\in\R$, $k\in\Z$, and $M, M_1, M_2$ are the operators defined by these symbols, then
\begin{equation}\label{TLX6.1}
\|M(gh)\|_{L^2(\mathbb{T}\times\R)}\lesssim \|M_1g\|_{L^2(\R)}\|M_2h\|_{L^2(\mathbb{T}\times\R)}.
\end{equation}

(iii) Finally, assume that $m,m_1,m_2:\Z\times\R\to\C$ are symbols satisfying
\begin{equation}\label{TLX5.2}
|m(k,\xi)|\leq |m_1(k-\ell,\xi-\eta)|\,|m_2(\ell,\eta)|\{\langle k-\ell,\xi-\eta\rangle^{-2}+\langle \ell,\eta\rangle^{-2}\}
\end{equation}
for any $\xi,\eta\in\R$, $k,\ell\in\Z$. If $M, M_1, M_2$ are the operators defined by these symbols, then
\begin{equation}\label{TLX6.2}
\|M(gh)\|_{L^2(\mathbb{T}\times\R)}\lesssim \|M_1g\|_{L^2(\mathbb{T}\times\R)}\|M_2h\|_{L^2(\mathbb{T}\times\R)}.
\end{equation}
\end{lemma}

\begin{proof} The proofs of the three claims are similar. For example, to prove (ii) we estimate
\begin{equation*}
\begin{split}
|\widehat{M(gh)}(k,\xi)|&\lesssim |m(k,\xi)|\int_{\R}|\widehat{g}(\xi-\eta)||\widehat{h}(k,\eta)|\,d\eta\\
&\lesssim \int_{\R}|m_1(\xi-\eta)|\,|m_2(k,\eta)|\{\langle\xi-\eta\rangle^{-2}+\langle k,\eta\rangle^{-2}\}|\widehat{g}(\xi-\eta)||\widehat{h}(k,\eta)|\,d\eta\\
&\lesssim \int_{\R}\frac{|\widehat{M_1g}(\xi-\eta)|}{\langle\xi-\eta\rangle^{2}}|\widehat{M_2h}(k,\eta)|\,d\eta+\int_\R|\widehat{M_1g}(\xi-\eta)|\frac{|\widehat{M_2h}(k,\eta)|}{\langle k,\eta\rangle^2}\,d\eta.
\end{split}
\end{equation*}
Therefore, for any function $f$ with $\|f\|_{L^2(\Z\times\R)}\lesssim 1$ we estimate
\begin{equation*}
\begin{split}
\Big|\sum_{k\in\Z}\int_\R f(k,\xi)&\widehat{M(gh)}(k,\xi)\,d\xi\Big|\lesssim \sum_{k\in\Z}\int_{\R\times\R} |f(k,\xi)|\frac{|\widehat{M_1g}(\eta)|}{\langle\eta\rangle^{2}}|\widehat{M_2h}(k,\xi-\eta)|\,d\eta d\xi\\
&+ \sum_{k\in\Z}\int_{\R\times\R} |f(k,\xi)||\widehat{M_1g}(\xi-\eta)|\frac{|\widehat{M_2h}(k,\eta)|}{\langle k,\eta\rangle^2}\,d\eta d\xi\\
&\lesssim \|\widehat{M_2h}\|_{L^2(\Z\times\R)} \int_{\R}\frac{|\widehat{M_1g}(\eta)|}{\langle\eta\rangle^{2}}\,d\eta+\|\widehat{M_1g}\|_{L^2(\R)} \sum_{k\in\Z}\int_{\R} \frac{|\widehat{M_2h}(k,\eta)|}{\langle k,\eta\rangle^2}\,d\eta\\
&\lesssim \|\widehat{M_1g}\|_{L^2(\R)} \|\widehat{M_2h}\|_{L^2(\Z\times\R)}.
\end{split}
\end{equation*}
The bounds \eqref{TLX6.1} follow. 
\end{proof}

To apply Lemma \ref{Multi0} we need good bounds on products of weights. In the next lemmas we collect several such bounds which are used to prove many of the bilinear estimates in the paper.

\begin{lemma}\label{lm:Multi}
For any $t\ge 1$, $\alpha\in[0,4]$, $\xi,\eta\in\R$, and $\ast\in\{NR,R\}$ we have
\begin{equation}\label{TLX4}
\langle\xi\rangle^{-\alpha}A_\ast(t,\xi)\lesssim_\delta \langle\xi-\eta\rangle^{-\alpha}A_\ast(t,\xi-\eta)\langle\eta\rangle^{-\alpha}A_\ast(t,\eta)e^{-(\lambda(t)/20)\min(\langle\xi-\eta\rangle,\langle \eta\rangle)^{1/2}}
\end{equation}
and
\begin{equation}\label{DtVMulti}
\big|(\dot{A}_\ast/A_\ast)(t,\xi)\big|\lesssim_\delta \left\{\big|(\dot{A}_\ast/A_\ast)(t,\xi-\eta)\big|+\big|(\dot{A}_\ast/A_\ast)(t,\eta)\big|\right\}e^{4\sqrt\delta\min(\langle\xi-\eta\rangle,\langle \eta\rangle)^{1/2}}.
\end{equation}
\end{lemma}

\begin{proof} Recall that $A_{R}(\rho,t)\geq A_{NR}(\rho,t)\gtrsim e^{\lambda(t)\langle\rho\rangle^{1/2}}$ for any $\rho\in\mathbb{R}$. The bounds \eqref{TLX4}--\eqref{DtVMulti} follow from \eqref{vfc25} and \eqref{vfc30}.
\end{proof}

\begin{lemma}\label{TLX40}
For any $t\in[1,\infty)$, $\xi,\eta\in\R$, and $k\in\mathbb{Z}$ we have
\begin{equation}\label{TLX7}
A_k(t,\xi)\lesssim_\delta A_R(t,\xi-\eta)A_k(t,\eta)e^{-(\lambda(t)/20)\min(\langle\xi-\eta\rangle,\langle k,\eta\rangle)^{1/2}}
\end{equation}
and
\begin{equation}\label{vfc30.7}
\big|(\dot{A}_k/A_k)(t,\xi)\big|\lesssim_\delta \left\{\big|(\dot{A}_R/A_R)(t,\xi-\eta)\big|+\big|(\dot{A}_k/A_k)(t,\eta)\big|\right\}e^{12\sqrt\delta\min(\langle\xi-\eta\rangle,\langle k,\eta\rangle)^{1/2}}.
\end{equation}
\end{lemma}

\begin{proof} To prove \eqref{TLX7} we examine the definition \eqref{dor4} and estimate, using \eqref{b>a} and \eqref{dor20},
\begin{equation*}
e^{\lambda(t)\langle k,\xi\rangle^{1/2}+\sqrt\delta|k|^{1/2}}\lesssim e^{\lambda(t)\langle\xi-\eta\rangle^{1/2}} e^{\lambda(t)\langle k,\eta\rangle^{1/2}+\sqrt\delta|k|^{1/2}}e^{-(\lambda(t)/20)\min(\langle\xi-\eta\rangle,\langle k,\eta\rangle)^{1/2}},
\end{equation*}
and
\begin{equation*}
\frac{e^{\lambda(t)\langle k,\xi\rangle^{1/2}}e^{\sqrt\delta\langle\xi\rangle^{1/2}}}{b_k(t,\xi)}\lesssim_\delta \frac{e^{\lambda(t)\langle\xi-\eta\rangle^{1/2}}e^{\sqrt\delta\langle\xi-\eta\rangle^{1/2}}}{b_R(t,\xi-\eta)} \frac{e^{\lambda(t)\langle k,\eta\rangle^{1/2}}e^{\sqrt\delta\langle\eta\rangle^{1/2}}}{b_k(t,\eta)}e^{-(\lambda(t)/20)\min(\langle\xi-\eta\rangle,\langle k,\eta\rangle)^{1/2}}.
\end{equation*}

The bounds \eqref{vfc30.7} follow from \eqref{eq:CDW} if $|\xi-\eta|\leq 4|(k,\eta)|$. On the other hand, if $|(k,\eta)|\leq|\xi-\eta|/4$ then, as a consequence of \eqref{vfc30.5},
\begin{equation*}
\big|(\dot{A}_k/A_k)(t,\xi)\big|\lesssim_\delta \big|(\dot{A}_R/A_R)(t,\xi)\big|.
\end{equation*}
The desired conclusion follows using also \eqref{vfc30}.
\end{proof}

\subsection{Weighted bilinear estimates for section \ref{fimprov}}\label{bilin6}

In this subsection we prove several estimates on products of weights, which are used only in the analysis of the normalized vorticity function in section \ref{fimprov}.

We begin with estimates on the weights that are used in the analysis of $\mathcal{N}_1$.
\begin{lemma}\label{TLXH1}
Assume that $t\ge1$ and recall the definitions of the sets $R_0,R_1,R_2,R_3$ in (\ref{nar18.1})-(\ref{nar18.4}). Denote $(\sigma,\rho):=(k-\ell,\xi-\eta)$. Suppose that $\sigma\neq0$.

(i) If $((k,\xi),(\ell,\eta))\in R_0\cup R_1$, then
\begin{equation}\label{TLXH1.1}
\begin{split}
\frac{|\rho/\sigma|+\langle t \rangle}{\langle t\rangle}\frac{\langle \rho \rangle /\sigma^2}{\langle t-\rho/\sigma \rangle^2}&\big|\ell A_k^2(t,\xi)-kA_{\ell}^2(t,\eta)\big|\\
&\lesssim_{\delta}\sqrt{|(A_k\dot{A}_k)(t,\xi)|}\,\sqrt{|(A_{\ell}\dot{A}_{\ell})(t,\eta)|}\,A_{\sigma}(t,\rho) \,e^{-(\delta_0/200)\langle \sigma,\rho \rangle^{1/2}}.
\end{split}
\end{equation}

(ii) If $((k,\xi),(\ell,\eta))\in R_2$, then
\begin{equation}\label{TLXH1.2}
\begin{split}
\frac{|\rho/\sigma|+\langle t \rangle}{\langle t\rangle}\frac{\langle \rho \rangle /\sigma^2}{\langle t-\rho/\sigma \rangle^2}&\big|\ell A_k^2(t,\xi)-kA_{\ell}^2(t,\eta)\big|\\
&\lesssim_{\delta}\sqrt{|(A_k\dot{A}_k)(t,\xi)|}\,\sqrt{|(A_{\sigma}\dot{A}_{\sigma})(t,\rho)|}\,A_{\ell}(t,\eta) \,e^{-(\delta_0/200)\langle \ell,\eta \rangle^{1/2}}.
\end{split}
\end{equation}
\end{lemma}

\begin{proof} {\bf Step 1.} Assume first that $((k,\xi),(\ell,\eta))\in R_0$ and we need to prove (\ref{TLXH1.1}). This is an easy case, as all frequencies involved are high and we have a gain in derivatives coming from (\ref{vfc26}), (\ref{vfc27}) and the lower bound
 \begin{equation}\label{lowerboundweight}
A_{\sigma}(t,\rho)\ge e^{\lambda(t)\langle \sigma,\rho\rangle^{1/2}}.
\end{equation}
We do not need to use the symmetrization.
Using the elementary inequalities 
\begin{equation}\label{TLXH1.3}
\frac{|\rho/\sigma|+\langle t \rangle}{\langle t\rangle}\frac{\langle \rho \rangle /\sigma^2}{\langle t-\rho/\sigma \rangle^2}\lesssim \langle t \rangle^{-2}e^{(\delta_0/300)\langle \sigma,\rho \rangle^{1/2}},
\end{equation}
for any $(\sigma,\rho)\in \mathbb{Z}\times\mathbb{R}$ with $\sigma\neq 0$, together with $\langle k,\xi \rangle^9+\langle \ell,\eta \rangle^9+\langle \sigma,\rho \rangle^9\lesssim e^{(\delta_0/300)\langle \sigma,\rho \rangle^{1/2}}$, the bounds (\ref{TLXH1.1}) follow from the combination of (\ref{vfc26}), (\ref{vfc27}), and (\ref{eq:A_kxi}).

{\bf Step 2.} Assume that $((k,\xi),(\ell,\eta))\in R_1$ and we need to prove (\ref{TLXH1.1}). In this case we need to use the symmetrization to reduce the loss of derivatives in $z$. 
We write
\begin{equation}\label{Transportterms1and2}
\ell A_k^2(t,\xi)-kA_{\ell}^2(t,\eta):=\mathcal{T}_1+\mathcal{T}_2+\mathcal{T}_3,
\end{equation}
with
\begin{equation}\label{Transportterms1}
\mathcal{T}_1:=\left(\ell e^{2\lambda(t)\langle k,\xi\rangle^{1/2}}-k e^{2\lambda(t)\langle\ell,\eta\rangle^{1/2}}\right)\bigg[\,\frac{e^{\sqrt{\delta}\langle\xi\rangle^{1/2}}}{b_k(t,\xi)}+e^{\sqrt{\delta}|k|^{1/2}}\,\bigg]^2,
\end{equation}
\begin{equation}\label{Transportterms2}
\mathcal{T}_2:=k e^{2\lambda(t)\langle\ell,\eta\rangle^{1/2}}\bigg[\,e^{\sqrt{\delta}|k|^{\frac{1}{2}}}-e^{\sqrt{\delta}|\ell|^{\frac{1}{2}}}\,\bigg]\bigg[\,\frac{e^{\sqrt{\delta}\langle\xi\rangle^{1/2}}}{b_k(t,\xi)}+e^{\sqrt{\delta}|k|^{1/2}}+\frac{e^{\sqrt{\delta}\langle \eta\rangle^{1/2}}}{b_{\ell}(t,\eta)}+e^{\sqrt{\delta}| \ell|^{1/2}}\,\bigg],
\end{equation}
\begin{equation}\label{Transportterms3}
\mathcal{T}_3:=k e^{2\lambda(t)\langle\ell,\eta\rangle^{1/2}}\bigg[\,\frac{e^{\sqrt{\delta}\langle\xi\rangle^{\frac{1}{2}}}}{b_k(t,\xi)}-\frac{e^{\sqrt{\delta}\langle\eta\rangle^{\frac{1}{2}}}}{b_{\ell}(t,\eta)}\,\bigg]\bigg[\,\frac{e^{\sqrt{\delta}\langle\xi\rangle^{1/2}}}{b_k(t,\xi)}+e^{\sqrt{\delta}|k|^{1/2}}+\frac{e^{\sqrt{\delta}\langle \eta\rangle^{1/2}}}{b_{\ell}(t,\eta)}+e^{\sqrt{\delta}| \ell|^{1/2}}\,\bigg].
\end{equation}
Using (\ref{TLXH1.3}), it suffices to prove for each $i\in\{1,2,3\}$
\begin{equation}\label{Tterms}
\frac{1}{\langle t\rangle^2}\big|\mathcal{T}_i\big|\lesssim_{\delta}\sqrt{|(A_k\dot{A}_k)(t,\xi)|}\,\sqrt{|(A_{\ell}\dot{A}_{\ell})(t,\eta)|}\,A_{\sigma}(t,\rho) \,e^{-(\delta_0/100)\langle \sigma,\rho \rangle^{1/2}}.
\end{equation}

In the proofs in this section we will often use the following bounds, which follow from \eqref{dor23.1}--\eqref{dor23.4} and \eqref{eq:comparisonweights1}: if $t\geq 0$ and $(k,\xi),(\ell,\eta)\in\mathbb{Z}\times\mathbb{R}$ then
\begin{equation}\label{iq2w1}
\frac{b_{\ell}(t,\eta)}{b_k(t,\xi)}\lesssim_\delta e^{\sqrt{\delta}|\xi-\eta|^{1/2}}\qquad\text{ if }(k,\xi)\notin I_t^{\ast\ast}\,\,\text{ or }\,\,(k,\xi),(\ell,\eta)\in I_t^\ast,
\end{equation}
and
\begin{equation}\label{iq2w2}
\frac{b_{\ell}(t,\eta)}{b_k(t,\xi)}\lesssim_\delta \frac{|\xi|/k^2}{\langle t-\xi/k\rangle}e^{\sqrt{\delta}|\xi-\eta|^{1/2}}\qquad\text{ if }(k,\xi)\in I_t^{\ast\ast}\text{ and }(\ell,\eta)\notin I_t^{\ast\ast}.
\end{equation}

{\bf Substep 2.1.} We first prove (\ref{Tterms}) for $i=1$. If $(k,\xi)\notin I_t^{\ast\ast}$ then we estimate, using \eqref{iq2w1},
\begin{equation*}
\begin{split}
\frac{1}{\langle t\rangle^2}\big|\mathcal{T}_1\big| \lesssim_{\delta} \langle\sigma,\rho\rangle&\bigg[1+\frac{|k|}{\langle k,\xi\rangle^{1/2}} \bigg]\cdot \frac{1}{\langle t\rangle^2}\left[ e^{2\lambda(t)\langle k,\xi\rangle^{1/2}} + e^{2\lambda(t)\langle\ell,\eta\rangle^{1/2}} \right]\\
 &\times \bigg[\,\frac{e^{\sqrt{\delta}\langle\xi\rangle^{1/2}}}{b_k(t,\xi)}+e^{\sqrt{\delta}|k|^{1/2}}\,\bigg]\bigg[\,\frac{e^{\sqrt{\delta}\langle\eta\rangle^{1/2}}}{b_{\ell}(t,\eta)}+e^{\sqrt{\delta}|\ell|^{1/2}}\,\bigg]e^{8\sqrt{\delta}\langle\sigma,\rho\rangle^{1/2}}.
\end{split}
\end{equation*}
Using (\ref{b>a}) with $\beta=1/2$, we can estimate the expression above as
\begin{equation}\label{TtermsT2intermediate}
%\begin{split}
\frac{1}{\langle t\rangle^2}\big|\mathcal{T}_1\big| \lesssim_{\delta}\frac{\langle k\rangle^{1/2}}{\langle t\rangle^{3/2}} A_k(t,\xi)A_{\ell}(t,\eta)A_{\sigma}(t,\rho) \,e^{-(\lambda(t)/20)\langle \sigma,\rho \rangle^{1/2}}.
%\end{split}
\end{equation}
The bounds (\ref{Tterms}) then follow from (\ref{TtermsT2intermediate}), (\ref{eq:A_kxi}) and (\ref{eq:CDW}).

Assume now that $(k,\xi)\in I_t^{\ast\ast}$. In this case $b_{\ell}(t,\eta)$ could be much bigger than $b_k(t,\xi)$ when $t$ is not resonant with respect to $(\ell,\eta)$, but this loss can be compensated as $t$ needs to be relatively large and we have a decay factor in $t$. Indeed, using \eqref{iq2w2} we can estimate
\begin{equation*}
\begin{split}
\frac{1}{\langle t\rangle^2}\big|\mathcal{T}_1\big| \lesssim_{\delta} |\sigma,\rho|&\bigg[1+\frac{|k|}{\langle k,\xi\rangle^{1/2}} \bigg]\cdot \frac{1}{\langle t\rangle^2}\left[ e^{2\lambda(t)\langle k,\xi\rangle^{1/2}} + e^{2\lambda(t)\langle\ell,\eta\rangle^{1/2}} \right]\cdot \frac{|\xi|/k^2}{\langle t-\xi/k\rangle}\\
 &\times \bigg[\,\frac{e^{\sqrt{\delta}\langle\xi\rangle^{1/2}}}{b_k(t,\xi)}+e^{\sqrt{\delta}|k|^{1/2}}\,\bigg]\bigg[\,\frac{e^{\sqrt{\delta}\langle\eta\rangle^{1/2}}}{b_{\ell}(t,\eta)}+e^{\sqrt{\delta}|\ell|^{1/2}}\,\bigg]e^{5\sqrt{\delta}\langle\sigma,\rho\rangle^{1/2}}.
\end{split}
\end{equation*}
Since $t\approx |\xi/k|$, and using (\ref{b>a}) with $\beta=1/2$, we can estimate the expression above as
\begin{equation}\label{TtermsT1intermediate}
%\begin{split}
\frac{1}{\langle t\rangle^2}\big|\mathcal{T}_1\big| \lesssim_{\delta}\frac{\langle\xi\rangle^{1/2}}{\langle t\rangle^{3/2}} A_k(t,\xi)A_{\ell}(t,\eta)A_{\sigma}(t,\rho) \,e^{-(\lambda(t)/20)\langle \sigma,\rho \rangle^{1/2}},
%\end{split}
\end{equation}
and (\ref{Tterms}) follows from (\ref{TtermsT1intermediate}), (\ref{eq:A_kxi}) and (\ref{eq:CDW}) as before.

{\bf Substep 2.2.} We now prove (\ref{Tterms}) for $i=2$. Notice that 
\begin{equation*}
|k|\big|e^{\sqrt{\delta}|k|^{1/2}}-e^{\sqrt{\delta}|\ell|^{1/2}}\big|\lesssim_{\delta}\langle k\rangle^{1/2}\big[e^{\sqrt{\delta}|k|^{1/2}}+e^{\sqrt{\delta}|\ell|^{1/2}}\big]e^{8\sqrt{\delta}\langle \sigma,\rho\rangle^{1/2}}.
\end{equation*}
As in {\bf{Substep 2.1}}, it is then easy to see that $\mathcal{T}_2$ satisfies the same bounds \eqref{TtermsT2intermediate} as $\mathcal{T}_1$, in all cases, which gives the desired estimates.

{\bf Substep 2.3.} Finally we prove (\ref{Tterms}) in the case $i=3$. Assume first that $(k,\xi)\in I_t^{\ast}$, thus
\begin{equation}\label{tresonant2}
t\approx |\xi/k|,\qquad 1\leq k^2\leq 4\delta^2|\xi|.
\end{equation}
In this case, we do not need the cancellation, and the loss of derivative is compensated by the fact that $t$ is fairly large. By \eqref{iq2w1}--\eqref{iq2w2} and \eqref{eq:comparisonweights1}, we can estimate
\begin{equation*}
\begin{split}
\frac{1}{\langle t\rangle^2}\big|\mathcal{T}_3\big| \lesssim_{\delta}&\,  \frac{|k|}{\langle t\rangle^2}\cdot\frac{|\xi|}{k^2}\frac{e^{2\lambda(t)\langle\ell,\eta\rangle^{1/2}}}{\langle t-\xi/k\rangle}\bigg[\,\frac{e^{\sqrt{\delta}\langle\xi\rangle^{1/2}}}{b_k(t,\xi)}+e^{\sqrt{\delta}|k|^{1/2}}\,\bigg]\bigg[\,\frac{e^{\sqrt{\delta}\langle\eta\rangle^{1/2}}}{b_{\ell}(t,\eta)}+e^{\sqrt{\delta}|\ell|^{1/2}}\,\bigg]e^{8\sqrt{\delta}\langle\sigma,\rho\rangle^{1/2}}\\
\lesssim_{\delta}&\,\frac{1}{\langle t\rangle\langle t-\xi/k\rangle}e^{2\lambda(t)\langle\ell,\eta\rangle^{1/2}}\bigg[\,\frac{e^{\sqrt{\delta}\langle\xi\rangle^{1/2}}}{b_k(t,\xi)}+e^{\sqrt{\delta}|k|^{1/2}}\,\bigg]\bigg[\,\frac{e^{\sqrt{\delta}\langle\eta\rangle^{1/2}}}{b_{\ell}(t,\eta)}+e^{\sqrt{\delta}|\ell|^{1/2}}\,\bigg]e^{8\sqrt{\delta}\langle\sigma,\rho\rangle^{1/2}}.
\end{split}
\end{equation*}
Using (\ref{b>a}) with $\beta=1/2$, we can estimate the expression above as in \eqref{TtermsT1intermediate}, and (\ref{Tterms}) follows from (\ref{eq:A_kxi}) and (\ref{eq:CDW})

The proof is similar if $(\ell,\eta)\in I_t^\ast$. Finally, assume that $(k,\xi)\notin I_t^\ast$ and $(\ell,\eta)\notin I_t^\ast$. If $|k|>10\langle \xi\rangle$ then we use  (\ref{dor20}), (\ref{eq:comparisonweights1}), and the definitions of $w_{NR}$ to see that
\begin{equation}\label{Transportkdominates}
|k| \bigg|\,\frac{e^{\sqrt{\delta}\langle\xi\rangle^{1/2}}}{b_k(t,\xi)}-\frac{e^{\sqrt{\delta}\langle\eta\rangle^{1/2}}}{b_{\ell}(t,\eta)}\,\bigg|\lesssim_{\delta}e^{\sqrt{\delta}|k|^{1/2}/2}e^{5\sqrt{\delta}|\rho|^{1/2}}.
\end{equation}
By (\ref{Transportkdominates}), using (\ref{eq:comparisonweights1}) and (\ref{b>a}) with $\beta=1/2$, we can then bound
\begin{equation}\label{TtermsT3intermediate'}
%\begin{split}
\frac{1}{\langle t\rangle^2}\big|\mathcal{T}_3\big| \lesssim_{\delta}\,\frac{1}{\langle t\rangle^2 } A_k(t,\xi)A_{\ell}(t,\eta)A_{\sigma}(t,\rho) \,e^{-(\lambda(t)/20)\langle \sigma,\rho \rangle^{1/2}},
%\end{split}
\end{equation}
and (\ref{Tterms}) follows from (\ref{eq:A_kxi}) and (\ref{eq:CDW}).
 
On the other hand, if $(k,\xi)\notin I_t^\ast$, $(\ell,\eta)\notin I_t^\ast$, and $|k|\leq 10\langle\xi\rangle$ then $b_k(t,\xi)=b_{NR}(t,\xi)$, $b_{\ell}(t,\eta)=b_{NR}(t,\eta)$. Therefore
 \begin{equation}\label{Tlb1}
|k| \bigg|\,\frac{e^{\sqrt{\delta}\langle\xi\rangle^{1/2}}}{b_k(t,\xi)}-\frac{e^{\sqrt{\delta}\langle\eta\rangle^{1/2}}}{b_{\ell}(t,\eta)}\,\bigg|\lesssim_{\delta}\frac{|k|}{L_{\kappa}(t,\xi)}\cdot \frac{e^{\sqrt{\delta}\langle\xi\rangle^{1/2}}}{b_k(t,\xi)}e^{5\sqrt{\delta}\langle\sigma,\rho\rangle^{1/2}},
\end{equation}
using (\ref{dor21}), (\ref{dor20}), and (\ref{eq:comparisonweights1}). From (\ref{dor1}), it is easy to verify the bound
\begin{equation}\label{Tlb2}
\frac{1}{\langle t\rangle^{2}}\frac{\langle \xi\rangle}{L_{\kappa}(t,\xi)}\lesssim_\delta \frac{\langle \xi\rangle^{1/2}}{\langle t\rangle^{3/2}}.
\end{equation}
Using (\ref{Tlb1})--(\ref{Tlb2}) and the assumption $|k|\leq 10\langle\xi\rangle$, it follows that $\langle t\rangle^{-2}\mathcal{T}_3$ satisfies similar bounds as in \eqref{TtermsT1intermediate}, and (\ref{Tterms}) follows from (\ref{eq:A_kxi}) and (\ref{eq:CDW}).

{\bf Step 3.} Assume that $((k,\xi),(\ell,\eta))\in R_2$ and we need to prove (\ref{TLXH1.2}). In this case we do not use the symmetrization. However, we need to deal with the loss of derivative when
$$\langle t-\rho/\sigma\rangle\ll\langle\rho\rangle/\sigma^2.$$
This loss of derivative is the main reason for the design of the imbalanced weights $A_k(t,\xi)$.

We first prove the following easy bound for $\sigma\in\mathbb{Z}\backslash\{0\}$ and $((k,\xi),(\ell,\eta))\in R_2$,
\begin{equation*}
\begin{split}
\frac{|\rho/\sigma|+\langle t \rangle}{\langle t\rangle}&\frac{\langle \rho \rangle /\sigma^2}{\langle t-\rho/\sigma \rangle^2}|k|A_{\ell}^2(t,\eta)\lesssim_{\delta}\sqrt{|(A_k\dot{A}_k)(t,\xi)|}\,\sqrt{|(A_{\sigma}\dot{A}_{\sigma})(t,\rho)|}\,A_{\ell}(t,\eta) \,e^{-(\delta_0/200)\langle \ell,\eta \rangle^{1/2}}.
\end{split}
\end{equation*}
This follows follows from similar arguments as in {\bf Step 1}, due to the favorable gain in high derivatives. We omit the repetitive details.
 
Therefore, for (\ref{TLXH1.2}) it remains to prove that
 \begin{equation}\label{TLXH1.2.2}
\begin{split}
\frac{|\rho/\sigma|+\langle t \rangle}{\langle t\rangle}&\frac{\langle \rho \rangle /\sigma^2}{\langle t-\rho/\sigma \rangle^2}|\ell |A_k^2(t,\xi)\\
&\lesssim_{\delta}\sqrt{|(A_k\dot{A}_k)(t,\xi)|}\,\sqrt{|(A_{\sigma}\dot{A}_{\sigma})(t,\rho)|}\,A_{\ell}(t,\eta) \,e^{-(\delta_0/200)\langle \ell,\eta \rangle^{1/2}},
\end{split}
\end{equation}
for $\sigma\in\mathbb{Z}\backslash\{0\}$ and $((k,\xi),(\ell,\eta))\in R_2$.  We divide the proof into several cases.

{\bf Case 1.} We assume that 
\begin{equation}\label{Rt1}
|t-\rho/\sigma|\ge \frac{|\rho|}{10|\sigma|}.
\end{equation}
Then
\begin{equation}\label{Rt2}
\frac{|\rho/\sigma|+\langle t \rangle}{\langle t\rangle}\frac{\langle \rho \rangle /\sigma^2}{\langle t-\rho/\sigma \rangle^2}\lesssim \frac{\langle \rho \rangle /\sigma^2}{\langle t \rangle (|\rho/\sigma|+\langle t\rangle)}\lesssim_{\delta}\frac{\langle \rho\rangle^{1/2}}{\langle t\rangle^{3/2}}.
\end{equation}

If $t\not\in I_{k,\xi}$ or if $t\in I_{k,\xi}\cap I_{\sigma,\rho}$ then (\ref{TLXH1.2.2}) follows from (\ref{vfc26}), (\ref{eq:A_kxi}) and (\ref{eq:CDW}). On the other hand, if $t\in I_{k,\xi}$ and $t\not\in I_{\sigma,\rho}$ then
\begin{equation}\label{tinr}
t\approx |\xi/k|,\qquad |\xi|\ge \delta^{-10},\qquad1\leq k^2\leq \delta^3|\xi|,\qquad |t-\xi/k|\lesssim |\xi|/k^2.
\end{equation}
Using (\ref{vfc27}), (\ref{eq:A_kxi}), \eqref{Rt2}, and the lower bound $A_{\ell}(t,\eta)\ge e^{\lambda(t)\langle \ell,\eta\rangle^{1/2}}$, we estimate
\begin{equation}\label{Rt3}
\begin{split}
\frac{|\rho/\sigma|+\langle t \rangle}{\langle t\rangle}&\frac{\langle \rho \rangle /\sigma^2}{\langle t-\rho/\sigma \rangle^2}|\ell |A_k^2(t,\xi)\\
&\lesssim_{\delta}\, \frac{\langle \rho \rangle /\sigma^2}{\langle t \rangle (|\rho/\sigma|+\langle t\rangle)}\frac{|\xi|/k^2}{\langle t-\xi/k\rangle}A_k(t,\xi)A_{\ell}(t,\eta)A_{\sigma}(t,\rho)e^{-(\lambda(t)/20)\langle\ell,\eta\rangle^{1/2}}\\
&\lesssim_{\delta}\, \frac{1}{\langle t-\xi/k\rangle}A_k(t,\xi)A_{\ell}(t,\eta)A_{\sigma}(t,\rho)e^{-(\lambda(t)/20)\langle\ell,\eta\rangle^{1/2}}.
\end{split}
\end{equation}
The bounds (\ref{TLXH1.2.2}) then follow from (\ref{Rt3}), (\ref{eq:A_kxi}), (\ref{eq:CDW}), and \eqref{reb8}.

{\bf Case 2.} We assume now that 
\begin{equation}\label{Rt2.0}
|t-\rho/\sigma|\leq \frac{|\rho|}{10|\sigma|}\qquad {\rm and}\qquad |t-\rho/\sigma|\geq \frac{|\rho|}{10\sigma^2}\,.
\end{equation}
If $t\not\in I_{k,\xi}$ then we estimate
 \begin{equation*}
\begin{split}
\frac{|\rho/\sigma|+\langle t \rangle}{\langle t\rangle}&\frac{\langle \rho \rangle /\sigma^2}{\langle t-\rho/\sigma \rangle^2}|\ell |A_k^2(t,\xi)\lesssim_{\delta}\frac{1}{\langle t-\rho/\sigma\rangle}\,A_k(t,\xi)A_{\sigma}(t,\rho)A_{\ell}(t,\eta) \,e^{-(\lambda(t)/20)\langle \ell,\eta \rangle^{1/2}},
\end{split}
\end{equation*}
using (\ref{vfc26}). The bounds (\ref{TLXH1.2.2}) then follow from (\ref{eq:CDW}) and the bounds
\begin{equation}\label{RtL2.1}
\left|\frac{\partial_tA_{\sigma}(t,\rho)}{A_{\sigma}(t,\rho)}\right|\gtrsim_\delta\min\{1,\sigma^2/|\rho|\}\gtrsim_\delta \frac{1}{\langle t-\rho/\sigma\rangle}
\end{equation}
for $t\geq 1$ as in \eqref{Rt2.0}, which follow from (\ref{eq:A_kxi}), \eqref{reb8}, and \eqref{reb9}

On the other hand, if $t\in I_{k,\xi}$ then \eqref{tinr} holds and we estimate, using (\ref{vfc27}) and (\ref{Rt2.0}), 
\begin{equation*}
\begin{split}
\frac{|\rho/\sigma|+\langle t \rangle}{\langle t\rangle}\frac{\langle \rho \rangle /\sigma^2}{\langle t-\rho/\sigma \rangle^2}|\ell |A_k^2(t,\xi)&\lesssim_{\delta}\frac{\sigma^2}{\langle\rho\rangle}\cdot\frac{|\xi|}{k^2}\frac{1}{\langle t-\xi/k\rangle}\,A_k(t,\xi)A_{\sigma}(t,\rho)A_{\ell}(t,\eta) \,e^{-(\lambda(t)/20)\langle \ell,\eta \rangle^{1/2}}\\
&\lesssim_{\delta}\frac{1}{\langle t-\xi/k\rangle}\,A_k(t,\xi)A_{\sigma}(t,\rho)A_{\ell}(t,\eta) \,e^{-(\lambda(t)/30)\langle \ell,\eta \rangle^{1/2}}.
\end{split}
\end{equation*}
The bounds (\ref{TLXH1.2.2}) then follow from (\ref{eq:A_kxi}), (\ref{eq:CDW}), and \eqref{reb8}.

{\bf Case 3.} Finally, we assume that
\begin{equation}\label{Rt3.0}
 |t-\rho/\sigma|\leq \frac{|\rho|}{10\sigma^2}\,.
\end{equation}
If, in addition, 
\begin{equation*}
|\rho|\leq\delta^{-10}\qquad{\rm or}\qquad |\sigma|\ge k_0(\rho) \qquad {\rm or}\qquad |\ell,\eta|\ge |\rho|/(100\sigma^2)
\end{equation*}
then the bounds  (\ref{TLXH1.2.2}) still follow easily as in {\bf{Case 2}} above, as there is no real loss of derivatives. On the other hand, assume that
\begin{equation}\label{RC3.1}
|\rho|\geq\delta^{-10}\qquad{\rm and}\qquad 1\leq |\sigma|\leq k_0(\rho) \qquad {\rm and}\qquad |\ell,\eta|\leq |\rho|/(100\sigma^2).
\end{equation}
We can assume that $\ell\neq 0$, as otherwise the left hand side of (\ref{TLXH1.2.2}) vanishes. This is the main case, where the imbalance of the weights plays an essential role. The assumptions (\ref{Rt3.0})--(\ref{RC3.1}) and $k\neq\sigma$ imply that $t$ is not resonant with respect to $(k,\xi)$, and $|k|\leq |\xi|$.
We can estimate, using  (\ref{dor3}), \eqref{eq:comparisonweights1}, and (\ref{reb7}),
\begin{equation}\label{nonRtR}
\begin{split}
&A_{\sigma}(t,\rho)\approx A_{R}(t,\rho)\approx_\delta \frac{|\rho|}{\sigma^2}\frac{1}{1+|t-\rho/\sigma|}A_{NR}(t,\rho)\\
&\gtrsim_{\delta} \frac{|\rho|}{\sigma^2}\frac{1}{\langle t-\rho/\sigma\rangle}\frac{e^{\lambda(t)\langle\sigma,\rho\rangle^{1/2}}}{e^{\lambda(t)\langle k,\xi\rangle^{1/2}}}A_k(t,\xi)e^{-5\sqrt{\delta}\langle \ell,\eta\rangle^{1/2}}\gtrsim_{\delta}\frac{|\rho|}{\sigma^2}\frac{1}{\langle t-\rho/\sigma\rangle}\frac{A_{k}(t,\xi)}{A_\ell(t,\eta)}e^{(\lambda(t)/20)\langle \ell,\eta\rangle^{1/2}}.
\end{split}
\end{equation}
Hence
 \begin{equation*}
\begin{split}
\frac{|\rho/\sigma|+\langle t \rangle}{\langle t\rangle}&\frac{\langle \rho \rangle /\sigma^2}{\langle t-\rho/\sigma \rangle^2}|\ell |A_k^2(t,\xi)\lesssim_{\delta}\frac{1}{\langle t-\rho/\sigma\rangle}\,A_k(t,\xi)A_{\sigma}(t,\rho)A_{\ell}(t,\eta) \,e^{-(\lambda(t)/30)\langle \ell,\eta \rangle^{1/2}}.
\end{split}
\end{equation*}
The bounds (\ref{TLXH1.2.2}) then follow as before. This completes the proof of the lemma.
\end{proof}

We also need the following lemma used in the analysis of $\mathcal{N}_2$.
\begin{lemma}\label{TLXH3}
Assume that $t\ge1$ and recall the definitions of the sets $R_0,R_1,R_2,R_3$ in (\ref{nar18.1})-(\ref{nar18.4}). Denote $(\sigma,\rho):=(k-\ell,\xi-\eta)$. Suppose that $\sigma\neq0$.

(i) If $((k,\xi),(\ell,\eta))\in R_0\cup R_1$, then
\begin{equation}\label{TLXH3.1}
\begin{split}
\frac{|\rho/\sigma|^2+\langle t \rangle^2}{\sigma\langle t\rangle^2}\frac{1}{\langle t-\rho/\sigma \rangle^2}&\big|\eta A_k^2(t,\xi)-\xi A_{\ell}^2(t,\eta)\big|\\
&\lesssim_{\delta}\sqrt{|(A_k\dot{A}_k)(t,\xi)|}\,\sqrt{|(A_{\ell}\dot{A}_{\ell})(t,\eta)|}\,A_{\sigma}(t,\rho) \,e^{-(\delta_0/200)\langle \sigma,\rho \rangle^{1/2}}.
\end{split}
\end{equation}

(ii) If $((k,\xi),(\ell,\eta))\in R_2$, then
\begin{equation}\label{TLXH3.2}
\begin{split}
\frac{|\rho/\sigma|^2+\langle t \rangle^2}{\sigma\langle t\rangle^2}\frac{1}{\langle t-\rho/\sigma \rangle^2}&\big|\eta A_k^2(t,\xi)-\xi A_{\ell}^2(t,\eta)\big|\\
&\lesssim_{\delta}\sqrt{|(A_k\dot{A}_k)(t,\xi)|}\,\sqrt{|(A_{\sigma}\dot{A}_{\sigma})(t,\rho)|}\,A_{\ell}(t,\eta) \,e^{-(\delta_0/200)\langle \ell,\eta \rangle^{1/2}}.
\end{split}
\end{equation}
\end{lemma}

\begin{proof} {\bf Step 1.} Assume first that $((k,\xi),(\ell,\eta))\in R_0$ and we prove (\ref{TLXH3.1}). In this case we do not need to use the symmetrization, due to the favorable gain of derivatives exactly as in {\bf Step 1} of the proof of (\ref{TLXH1.1}). We omit the repetitive details.

{\bf Step 2.}  Assume now that $((k,\xi),(\ell,\eta))\in R_1$ and we prove (\ref{TLXH3.1}). In this case we need to use the symmetrization to reduce the loss of derivatives in $v$. 
We write
\begin{equation}\label{Transportrms1and2}
\eta A_k^2(t,\xi)-\xi A_{\ell}^2(t,\eta):=\mathcal{T}'_1+\mathcal{T}'_2+\mathcal{T}'_3,
\end{equation}
with
\begin{equation}\label{Transportrms1}
\mathcal{T}'_1:=\left(\eta e^{2\lambda(t)\langle k,\xi\rangle^{1/2}}-\xi e^{2\lambda(t)\langle\ell,\eta\rangle^{1/2}}\right)\bigg[\,\frac{e^{\sqrt{\delta}\langle\xi\rangle^{1/2}}}{b_k(t,\xi)}+e^{\sqrt{\delta}|k|^{1/2}}\,\bigg]^2,
\end{equation}
\begin{equation}\label{Transportrms2}
\mathcal{T}'_2:=\xi e^{2\lambda(t)\langle\ell,\eta\rangle^{1/2}}\bigg[\,e^{\sqrt{\delta}|k|^{\frac{1}{2}}}-e^{\sqrt{\delta}|\ell|^{\frac{1}{2}}}\,\bigg]\bigg[\,\frac{e^{\sqrt{\delta}\langle\xi\rangle^{1/2}}}{b_k(t,\xi)}+e^{\sqrt{\delta}|k|^{1/2}}+\frac{e^{\sqrt{\delta}\langle \eta\rangle^{1/2}}}{b_{\ell}(t,\eta)}+e^{\sqrt{\delta}| \ell|^{1/2}}\,\bigg],
\end{equation}
\begin{equation}\label{Transportrms3}
\begin{split}
\mathcal{T}'_3:=\xi e^{2\lambda(t)\langle\ell,\eta\rangle^{1/2}}\bigg[\,\frac{e^{\sqrt{\delta}\langle\xi\rangle^{\frac{1}{2}}}}{b_k(t,\xi)}-\frac{e^{\sqrt{\delta}\langle\eta\rangle^{\frac{1}{2}}}}{b_{\ell}(t,\eta)}\,\bigg]\bigg[\,\frac{e^{\sqrt{\delta}\langle\xi\rangle^{1/2}}}{b_k(t,\xi)}+e^{\sqrt{\delta}|k|^{1/2}}+\frac{e^{\sqrt{\delta}\langle \eta\rangle^{1/2}}}{b_{\ell}(t,\eta)}+e^{\sqrt{\delta}| \ell|^{1/2}}\,\bigg].
\end{split}
\end{equation}
It suffices to prove that for each $i\in\{1,2,3\}$
\begin{equation}\label{T'terms}
\frac{1}{\langle t\rangle^2}\big|\mathcal{T}'_i\big|\lesssim_{\delta}\sqrt{|(A_k\dot{A}_k)(t,\xi)|}\,\sqrt{|(A_{\ell}\dot{A}_{\ell})(t,\eta)|}\,A_{\sigma}(t,\rho) \,e^{-(\delta_0/100)\langle \sigma,\rho \rangle^{1/2}}.
\end{equation}

{\bf Substep 2.1.} We first prove (\ref{T'terms}) for $i=1$. The argument is similar to the argument in {\bf Substep 2.1} in the proof of (\ref{TLXH1.1}). For later use we prove slightly stronger bounds. More precisely, we estimate
\begin{equation}\label{Tros1}
\big|\mathcal{T}'_1\big| \lesssim_{\delta} \langle\sigma,\rho\rangle\langle k,\xi\rangle^{1/2}\left[ e^{2\lambda(t)\langle k,\xi\rangle^{1/2}} + e^{2\lambda(t)\langle\ell,\eta\rangle^{1/2}} \right]\bigg[\,\frac{e^{\sqrt{\delta}\langle\xi\rangle^{1/2}}}{b_k(t,\xi)}+e^{\sqrt{\delta}|k|^{1/2}}\bigg]^2.
\end{equation}
We use \eqref{iq2w1} and \eqref{eq:comparisonweights1} if $(k,\xi)\notin I_t^{\ast\ast}$. Recalling also (\ref{b>a}) we have
\begin{equation}\label{Tros2}
%\begin{split}
\frac{1}{\langle t\rangle^{7/4}}\big|\mathcal{T}'_1\big| \lesssim_{\delta}\frac{\langle k,\xi\rangle^{1/2}}{\langle t\rangle^{7/4}} A_k(t,\xi)A_{\ell}(t,\eta)A_{\sigma}(t,\rho) \,e^{-(\lambda(t)/20)\langle \sigma,\rho \rangle^{1/2}}.
%\end{split}
\end{equation}
The bounds (\ref{T'terms}) then follow from (\ref{eq:A_kxi}) and (\ref{eq:CDW}).

On the other hand, if $(k,\xi)\in I_t^{\ast\ast}$ then we use \eqref{iq2w2} and \eqref{Tros1} to estimate
\begin{equation}\label{Tros3}
%\begin{split}
\frac{1}{\langle t\rangle^{7/4}}\big|\mathcal{T}'_1\big| \lesssim_{\delta}\frac{\langle \xi\rangle^{3/2}}{\langle t\rangle^{7/4}k^2}\frac{1}{\langle t-\xi/k\rangle}A_k(t,\xi)A_{\ell}(t,\eta)A_{\sigma}(t,\rho) \,e^{-(\lambda(t)/20)\langle \sigma,\rho \rangle^{1/2}},
%\end{split}
\end{equation}
and (\ref{T'terms}) follows from (\ref{eq:A_kxi}), (\ref{eq:CDW}), and \eqref{reb8} as before.

{\bf Substep 2.2.} We now prove (\ref{T'terms}) for $i=2$. If $|k|>\langle\xi\rangle/10$ then 
\begin{equation}\label{T'klarge}
|\xi| \big|\,e^{\sqrt{\delta}|k|^{1/2}}-e^{\sqrt{\delta}|\ell|^{1/2}}\,\big|\lesssim_{\delta} \langle\xi\rangle^{1/2}\big[e^{\sqrt{\delta}|k|^{1/2}}+e^{\sqrt{\delta}|\ell|^{1/2}}\big]e^{8\sqrt{\delta}\langle\sigma,\rho\rangle^{1/2}},
\end{equation}
and therefore the stronger bounds \eqref{Tros2} hold for $\langle t\rangle^{-7/4}\big|\mathcal{T}'_2|$ as well. On the other hand, if $|k|\leq \langle\xi\rangle/10$ then
\begin{equation}\label{T'ksmall}
|\xi| \big|\,e^{\sqrt{\delta}|k|^{1/2}}-e^{\sqrt{\delta}|\ell|^{1/2}}\,\big|\lesssim_{\delta} \big[e^{0.5\sqrt{\delta}\langle\xi\rangle^{1/2}}+e^{0.5\sqrt{\delta}\langle\eta\rangle^{1/2}}\big]e^{5\sqrt{\delta}\langle\sigma,\rho\rangle^{1/2}},
\end{equation}
so the stronger bounds \eqref{Tros2} hold for $\langle t\rangle^{-7/4}\big|\mathcal{T}'_2|$ in this case as well.

{\bf Substep 2.3.} We now prove (\ref{T'terms}) for $i=3$. As in the proof of \eqref{Tterms}, assume first  that $(k,\xi)\in I_t^\ast$. Using \eqref{iq2w1}--\eqref{iq2w2}, (\ref{dor20}), and (\ref{eq:comparisonweights1}) we estimate, without using the cancellation,
\begin{equation*}
\begin{split}
\frac{1}{\langle t\rangle^2}\big|\mathcal{T}'_3\big|& \lesssim_{\delta}\,  \frac{|\xi|}{\langle t\rangle^2}\cdot\frac{|\xi|}{k^2}\frac{e^{2\lambda(t)\langle\ell,\eta\rangle^{1/2}}}{\langle t-\xi/k\rangle}\,\bigg[\,\frac{e^{\sqrt{\delta}\langle\xi\rangle^{1/2}}}{b_k(t,\xi)}+e^{\sqrt{\delta}|k|^{1/2}}\,\bigg]\bigg[\,\frac{e^{\sqrt{\delta}\langle\eta\rangle^{1/2}}}{b_{\ell}(t,\eta)}+e^{\sqrt{\delta}|\ell|^{1/2}}\,\bigg]e^{8\sqrt{\delta}\langle\sigma,\rho\rangle^{1/2}}\\
&\lesssim_{\delta}\,\frac{e^{2\lambda(t)\langle\ell,\eta\rangle^{1/2}}}{\langle t-\xi/k\rangle}\bigg[\,\frac{e^{\sqrt{\delta}\langle\xi\rangle^{1/2}}}{b_k(t,\xi)}+e^{\sqrt{\delta}|k|^{1/2}}\,\bigg]\bigg[\,\frac{e^{\sqrt{\delta}\langle\eta\rangle^{1/2}}}{b_{\ell}(t,\eta)}+e^{\sqrt{\delta}|\ell|^{1/2}}\,\bigg]e^{8\sqrt{\delta}\langle\sigma,\rho\rangle^{1/2}}.
\end{split}
\end{equation*}
Using (\ref{b>a}) with $\beta=1/2$, we can estimate the expression above as
\begin{equation}\label{Tros5}
%\begin{split}
\frac{1}{\langle t\rangle^2}\big|\mathcal{T}'_3\big| \lesssim_{\delta}\,\frac{1}{\langle t-\xi/k\rangle}A_k(t,\xi)A_{\ell}(t,\eta)A_{\sigma}(t,\rho) \,e^{-(\lambda(t)/20)\langle \sigma,\rho \rangle^{1/2}},
%\end{split}
\end{equation}
and (\ref{T'terms}) for $i=3$ then follows from (\ref{eq:A_kxi}), (\ref{eq:CDW}), and \eqref{reb8}.

The proof is similar if $(\ell,\eta)\in I_t^\ast$. On the other hand, if $(k,\xi)\notin I_t^\ast$ and $(\ell,\eta)\notin I_t^\ast$ then $b_k(t,\xi)=b_{NR}(t,\xi)$ and $b_{\ell}(t,\eta)=b_{NR}(t,\eta)$ (see \eqref{dor23.2}). By (\ref{dor21}), (\ref{dor20}) and (\ref{eq:comparisonweights1}), we have
 \begin{equation}\label{T'lb1}
|\xi| \bigg|\,\frac{e^{\sqrt{\delta}\langle\xi\rangle^{1/2}}}{b_k(t,\xi)}-\frac{e^{\sqrt{\delta}\langle\eta\rangle^{1/2}}}{b_{\ell}(t,\eta)}\,\bigg|\lesssim_{\delta}\frac{\langle \xi\rangle}{L_{\kappa}(t,\xi)}\cdot \frac{e^{\sqrt{\delta}\langle\xi\rangle^{1/2}}}{b_k(t,\xi)}e^{8\sqrt{\delta}\langle\sigma,\rho\rangle^{1/2}}.
\end{equation}
Using (\ref{Tlb2}) and (\ref{T'lb1}), we then estimate
\begin{equation*}
\frac{1}{\langle t\rangle^{7/4}}\big|\mathcal{T}'_3\big| \lesssim_{\delta}\frac{\langle\xi\rangle^{1/2}}{\langle t\rangle^{5/4}}A_k(t,\xi)A_{\ell}(t,\eta)A_{\sigma}(t,\rho)e^{-(\lambda(t)/20)\langle \sigma,\rho\rangle^{1/2}},
\end{equation*}
and (\ref{T'terms}) follows from (\ref{eq:A_kxi}) and (\ref{eq:CDW}).

{\bf Step 3.} Assume that $((k,\xi),(\ell,\eta))\in R_2$ and we prove (\ref{TLXH3.2}). In this case we do not need to use the symmetrization. Assuming $\sigma\in\mathbb{Z}\backslash\{0\}$ and $((k,\xi),(\ell,\eta))\in R_2$, the bounds
\begin{equation}\label{TLXH3.2.1}
\begin{split}
\frac{|\rho/\sigma|^2+\langle t \rangle^2}{\sigma\langle t\rangle^2}&\frac{1}{\langle t-\rho/\sigma \rangle^2}|\xi |A_{\ell}^2(t,\eta)\\
&\lesssim_{\delta}\sqrt{|(A_k\dot{A}_k)(t,\xi)|}\,\sqrt{|(A_{\sigma}\dot{A}_{\sigma})(t,\rho)|}\,A_{\ell}(t,\eta) \,e^{-(\delta_0/200)\langle \ell,\eta \rangle^{1/2}}.
\end{split}
\end{equation}
follow easily from the derivative gain (\ref{vfc26})--(\ref{vfc27}) and the lower bounds $A_{\sigma}(t,\rho)\ge e^{\lambda(t)\langle \sigma,\rho\rangle^{1/2}}$. 

For (\ref{TLXH3.2}) it remains to prove that, for $\sigma\in\mathbb{Z}\backslash\{0\}$ and $((k,\xi),(\ell,\eta))\in R_2$,
\begin{equation}\label{TLXH3.2.2}
\begin{split}
\frac{|\rho/\sigma|^2+\langle t \rangle^2}{\sigma\langle t\rangle^2}&\frac{1}{\langle t-\rho/\sigma \rangle^2}|\eta |A_{k}^2(t,\xi)\\
&\lesssim_{\delta}\sqrt{|(A_k\dot{A}_k)(t,\xi)|}\,\sqrt{|(A_{\sigma}\dot{A}_{\sigma})(t,\rho)|}\,A_{\ell}(t,\eta) \,e^{-(\delta_0/200)\langle \ell,\eta \rangle^{1/2}}.
\end{split}
\end{equation}
This is similar to the proof of (\ref{TLXH1.2.2}). We consider two cases.

{\bf Case 1.} We first assume that 
\begin{equation}\label{Rt1'}
|t-\rho/\sigma|\ge |\rho|/(10|\sigma|).
\end{equation}
Then
\begin{equation}\label{Rt2'}
\frac{|\rho/\sigma|^2+\langle t \rangle^2}{\sigma\langle t\rangle^2}\frac{1}{\langle t-\rho/\sigma \rangle^2}\lesssim \frac{1}{\sigma \langle t\rangle^2},
\end{equation}
and (\ref{TLXH3.2.2}) follows in this case, using similar argument as in {\bf Case 1} of the proof of (\ref{TLXH1.2.2}).

{\bf Case 2.} Finally, we assume that 
\begin{equation}\label{Rt2.0'}
|t-\rho/\sigma|\leq|\rho|/(10|\sigma|).
\end{equation}
If $t\not\in I_{k,\xi}$ then we estimate, using (\ref{vfc26}),
 \begin{equation*}
\frac{|\rho/\sigma|^2+\langle t \rangle^2}{\sigma\langle t\rangle^2}\frac{1}{\langle t-\rho/\sigma \rangle^2}|\eta |A_k^2(t,\xi)\lesssim_{\delta}\frac{1}{\langle t-\rho/\sigma\rangle}\,A_k(t,\xi)A_{\sigma}(t,\rho)A_{\ell}(t,\eta) \,e^{-(\lambda(t)/20)\langle \ell,\eta \rangle^{1/2}}.
\end{equation*}
The bounds (\ref{TLXH3.2.2}) then follow from (\ref{eq:CDW}), \eqref{eq:A_kxi}, and \eqref{reb8}. On the other hand, if $t\in I_{k,\xi}$ and $|t-\rho/\sigma|\geq\frac{|\rho|}{10\sigma^2}$ then we use \eqref{vfc26}--(\ref{vfc27}) and (\ref{Rt2.0'}) to estimate
\begin{equation*}
\begin{split}
\frac{|\rho/\sigma|^2+\langle t \rangle^2}{\sigma\langle t\rangle^2}\frac{1}{\langle t-\rho/\sigma \rangle^2}|\eta |A_k^2(t,\xi)&\lesssim_{\delta}\frac{\sigma^2}{\langle\rho\rangle}\cdot\frac{|\xi|/k^2}{\langle t-\xi/k\rangle}\,A_k(t,\xi)A_{\sigma}(t,\rho)A_{\ell}(t,\eta) \,e^{-(\lambda(t)/20)\langle \ell,\eta \rangle^{1/2}}\\
&\lesssim_{\delta}\frac{1}{\langle t-\xi/k\rangle}A_k(t,\xi)A_{\sigma}(t,\rho)A_{\ell}(t,\eta) \,e^{-(\lambda(t)/30)\langle \ell,\eta \rangle^{1/2}}.
\end{split}
\end{equation*}
The bounds (\ref{TLXH3.2.2}) follow again from (\ref{eq:CDW}), \eqref{eq:A_kxi}, and \eqref{reb8}.

Finally, assume that $t\in I_{k,\xi}$ and $|t-\rho/\sigma|\leq\frac{|\rho|}{10\sigma^2}$. This is different from {\bf Case 3} in the proof of (\ref{TLXH1.2.2}). The imbalance of the weights is no longer useful, as we do not have the condition $\ell\neq0$. On the other hand, there is no loss of derivative either, and we estimate using \eqref{vfc26},
\begin{equation*}
\begin{split}
\frac{|\rho/\sigma|^2+\langle t \rangle^2}{\sigma\langle t\rangle^2}&\frac{1}{\langle t-\rho/\sigma \rangle^2}|\eta |A_k^2(t,\xi)\lesssim_{\delta}\frac{1}{\langle t-\rho/\sigma \rangle^2}A_k(t,\xi)A_{\sigma}(t,\rho)A_{\ell}(t,\eta) \,e^{-(\lambda(t)/20)\langle \ell,\eta \rangle^{1/2}}.
\end{split}
\end{equation*}
The bounds (\ref{TLXH3.2.2}) then follow as before. This completes the proof of the lemma.
\end{proof}

We also need the following lemma used in the analysis of $\mathcal{N}_3$.
\begin{lemma}\label{TLXH2}
Assume that $t\ge1$ and recall the definitions of the sets $\Sigma_0,\Sigma_1,\Sigma_2,\Sigma_3$ in (\ref{nar19.1}). Denote $\rho:=\xi-\eta$.

(i) If $((k,\xi),(k,\eta))\in \Sigma_0\cup \Sigma_1$, then
\begin{equation}\label{TLXH2.1}
\begin{split}
\frac{1}{\langle\rho\rangle\langle t\rangle+\langle \rho\rangle^{1/4}\langle t\rangle^{7/4}}&\big|\eta A_k^2(t,\xi)-\xi A_{k}^2(s,\eta)\big|\\
&\lesssim_{\delta}\sqrt{|(A_k\dot{A}_k)(t,\xi)|}\,\sqrt{|(A_{k}\dot{A}_{k})(t,\eta)|}\,A_{NR}(t,\rho) \,e^{-(\delta_0/200)\langle \rho \rangle^{1/2}}.
\end{split}
\end{equation}

(ii) If $((k,\xi),(k,\eta))\in \Sigma_2$, then
\begin{equation}\label{TLXH2.2}
\begin{split}
\frac{1}{\langle\rho\rangle\langle t\rangle+\langle \rho\rangle^{1/4}\langle t\rangle^{7/4}}&\big|\eta A_k^2(t,\xi)-\xi A_{k}^2(t,\eta)\big|\\
&\lesssim_{\delta}\sqrt{|(A_k\dot{A}_k)(t,\xi)|}\,\sqrt{|(A_{NR}\dot{A}_{NR})(t,\rho)|}\,A_{k}(t,\eta) \,e^{-(\delta_0/200)\langle k,\eta \rangle^{1/2}}.
\end{split}
\end{equation}
\end{lemma}

\begin{proof} (i) If $((k,\xi),(\ell,\eta))\in \Sigma_0$ then there is no derivative loss and the proof of \eqref{TLXH2.1} is similar to the proof of \eqref{TLXH3.1}. If $((k,\xi),(k,\eta))\in \Sigma_1$ then we write, as in \eqref{Transportrms1and2}--\eqref{Transportrms3},
\begin{equation}\label{Trterm1and2}
\eta A_k^2(t,\xi)-\xi A_{k}^2(t,\eta):=\mathcal{T}''_1+\mathcal{T}''_2,
\end{equation}
with
\begin{equation}\label{Trterm1}
\mathcal{T}''_1:=\left(\eta e^{2\lambda(t)\langle k,\xi\rangle^{1/2}}-\xi e^{2\lambda(t)\langle k,\eta\rangle^{1/2}}\right)\bigg[\,\frac{e^{\sqrt{\delta}\langle\xi\rangle^{1/2}}}{b_k(t,\xi)}+e^{\sqrt{\delta}|k|^{1/2}}\,\bigg]^2,
\end{equation}
\begin{equation}\label{Trterm3}
\begin{split}
\mathcal{T}''_2:=\xi e^{2\lambda(t)\langle k,\eta\rangle^{1/2}}\bigg[\,\frac{e^{\sqrt{\delta}\langle\xi\rangle^{1/2}}}{b_k(t,\xi)}-\frac{e^{\sqrt{\delta}\langle\eta\rangle^{1/2}}}{b_{k}(t,\eta)}\,\bigg]\cdot\bigg[\,\frac{e^{\sqrt{\delta}\langle\xi\rangle^{1/2}}}{b_k(t,\xi)}+\frac{e^{\sqrt{\delta}\langle \eta\rangle^{1/2}}}{b_{k}(t,\eta)}+2e^{\sqrt{\delta}|k|^{1/2}}\,\bigg].
\end{split}
\end{equation}
For \eqref{TLXH2.1} it suffices to prove that, for $i\in\{1,2\}$,
\begin{equation}\label{TLXH2.11}
\begin{split}
\frac{1}{\langle t\rangle^{7/4}}\big|\mathcal{T}''_i\big|\lesssim_{\delta}\sqrt{|(A_k\dot{A}_k)(t,\xi)|}\,\sqrt{|(A_{k}\dot{A}_{k})(t,\eta)|}\,A_{NR}(t,\rho) \,e^{-(\delta_0/100)\langle \rho \rangle^{1/2}}.
\end{split}
\end{equation}

In the case $i=1$, the bounds \eqref{TLXH2.11} follow similarly to \eqref{Tros2}--\eqref{Tros3}, with $k=\ell$. If $i=2$ then we estimate, using \eqref{eq:comparisonweights1}, (\ref{dor21}), and (\ref{dor1}),
\begin{equation*}
\begin{split}
\big|\mathcal{T}''_2\big|&\lesssim_{\delta}\langle\xi\rangle\frac{\langle\xi\rangle^{1/2}+t}{\langle\xi\rangle+t}  e^{2\lambda(t)\langle k,\eta\rangle^{1/2}}\bigg[\,\frac{e^{\sqrt{\delta}\langle\xi\rangle^{1/2}}}{b_k(t,\xi)}+e^{\sqrt{\delta}|k|^{1/2}}\,\bigg]\bigg[\,\frac{e^{\sqrt{\delta}\langle\eta\rangle^{1/2}}}{b_k(t,\eta)}+e^{\sqrt{\delta}|k|^{1/2}}\,\bigg]\,e^{8\sqrt{\delta}\langle\rho\rangle^{1/2}}\\
&\lesssim_{\delta}\langle\xi\rangle^{1/2}\langle t\rangle^{1/2}A_k(t,\xi)A_k(t,\eta)A_{NR}(t,\rho)e^{-(\lambda(t)/20)\langle \rho\rangle^{1/2}}.
\end{split}
\end{equation*}
The bounds (\ref{TLXH2.1}) then follow from (\ref{eq:A_kxi}) and (\ref{eq:CDW}). 

(ii) Assume that $((k,\xi),(k,\eta))\in \Sigma_2$ and we need to prove (\ref{TLXH2.2}). We notice first that 
\begin{equation*}
\begin{split}
\frac{|\xi| A_{k}^2(t,\eta)}{\langle\rho\rangle\langle t\rangle+\langle \rho\rangle^{1/4}\langle t\rangle^{7/4}}\lesssim_{\delta}\sqrt{|(A_k\dot{A}_k)(t,\xi)|}\,\sqrt{|(A_{NR}\dot{A}_{NR})(t,\rho)|}\,A_{k}(t,\eta) \,e^{-(\delta_0/200)\langle k,\eta \rangle^{1/2}},
\end{split}
\end{equation*}
which is similar to the simple bounds \eqref{TLXH3.2.1}. It remains to prove the harder inequality
\begin{equation}\label{TLXH2.22}
\begin{split}
\frac{|\eta| A_{k}^2(t,\xi)}{\langle\rho\rangle\langle t\rangle+\langle \rho\rangle^{1/4}\langle t\rangle^{7/4}}\lesssim_{\delta}\sqrt{|(A_k\dot{A}_k)(t,\xi)|}\,\sqrt{|(A_{NR}\dot{A}_{NR})(t,\rho)|}\,A_{k}(t,\eta) \,e^{-(\delta_0/200)\langle k,\eta \rangle^{1/2}}.
\end{split}
\end{equation}
 
For $((k,\xi),(k,\eta))\in \Sigma_2$ we have
\begin{equation}\label{AkANRn3}
A_k(t,\xi)\approx e^{\lambda(t)\langle k,\xi\rangle^{1/2}}\frac{e^{\sqrt{\delta}\langle\xi\rangle^{1/2}}}{b_k(t,\xi)}.
\end{equation}
If $t\notin I_{k,\xi}$ then we estimate, using (\ref{vfc25})  and (\ref{AkANRn3}),
\begin{equation}\label{TLXH2.221}
\begin{split}
\frac{|\eta| A_{k}^2(t,\xi)}{\langle\rho\rangle\langle t\rangle+\langle \rho\rangle^{1/4}\langle t\rangle^{7/4}}\lesssim_{\delta}\frac{1}{\langle t\rangle^{3/2}}A_k(t,\xi)
A_{NR}(t,\rho) A_k(t,\eta) \,e^{-(\lambda(t)/20)\langle k,\eta \rangle^{1/2}},
\end{split}
\end{equation}
and (\ref{TLXH2.22}) follows from (\ref{TLX3.5}) and (\ref{vfc30}). On the other hand, if $t\in I_{k,\xi}$ then we use \eqref{reb7}, (\ref{eq:comparisonweights1}), and (\ref{AkANRn3}) to estimate
\begin{equation}\label{TLXH2.222}
\begin{split}
\frac{|\eta| A_{k}^2(t,\xi)}{\langle\rho\rangle\langle t\rangle+\langle \rho\rangle^{1/4}\langle t\rangle^{7/4}}\lesssim_{\delta}\,&\frac{1}{\langle\rho\rangle\langle t\rangle}\cdot\frac{|\xi|}{k^2}\frac{1}{\langle t-\xi/k\rangle}A_k(t,\xi)A_{NR}(t,\rho) A_k(t,\eta) \,e^{-(\lambda(t)/20)\langle k,\eta \rangle^{1/2}}\\
\lesssim_{\delta}\,&\frac{1}{\langle t\rangle\langle t-\xi/k\rangle} A_k(t,\xi)A_{NR}(t,\rho) A_k(t,\eta) \,e^{-(\lambda(t)/20)\langle k,\eta \rangle^{1/2}}.
\end{split}
\end{equation}
The bounds (\ref{TLXH2.22}) then follow from (\ref{eq:A_kxi}) and (\ref{vfc30.5}). The proof of Lemma \ref{TLXH2} is complete.
\end{proof}

\subsection{Weighted bilinear estimates for section \ref{coimprov2}}\label{bilin7}

In this subsection we prove several estimates on products of weights, which are used only in the analysis of the coordinate functions in section \ref{coimprov2}.

\begin{lemma}\label{TLX40.1}
For any $t\geq 1$, $k\in\Z\setminus\{0\}$, and $\xi,\eta\in\R$ we have, with $\rho=\xi-\eta$,
\begin{equation}\label{TLX41}
\frac{A^2_{NR}(t,\xi)}{|\dot{A}_{NR}(t,\xi)|}\langle t\rangle^{3/4}\langle\xi\rangle^{1/4}\lesssim _\delta A_k(t,\eta)\frac{\langle t\rangle\langle t-\eta/k\rangle^2}{\langle t\rangle+|\eta/k|}A_{-k}(t,\rho)e^{-(\lambda(t)/20)[\min(\langle\rho\rangle,\langle \eta\rangle)+|k|]^{1/2}}
\end{equation}
and
\begin{equation}\label{TLX42}
\begin{split}
|(\dot{A}_{NR}/A_{NR})(t,\xi)|&\lesssim_\delta \big\{|(\dot{A}_k/A_k)(t,\eta)|+|(\dot{A}_{-k}/A_{-k})(t,\rho)|\big\}e^{12\sqrt\delta[\min(\langle\rho\rangle,\langle \eta\rangle)+|k|]^{1/2}}.
\end{split}
\end{equation}
\end{lemma}

\begin{proof} We start with the easier bounds \eqref{TLX42}.  We may assume that $|\rho|\leq|\eta|$ and it suffices to prove that
\begin{equation*}
|(\dot{A}_{NR}/A_{NR})(t,\xi)|\lesssim_\delta |(\dot{A}_k/A_k)(t,\eta)|e^{12\sqrt\delta[\langle\rho\rangle+|k|]^{1/2}}.
\end{equation*}
We use \eqref{TLX3.5}--\eqref{eq:A_kxi}, so it suffices to prove that
\begin{equation}\label{TLX43}
\begin{split}
&\frac{\langle\xi\rangle^{1/2}}{\langle t\rangle^{1+\sigma_0}}+\left|\frac{\partial_tw_{NR}(t,\xi)}{w_{NR}(t,\xi)}\right|\\
&\lesssim_\delta e^{12\sqrt\delta[\langle\rho\rangle+|k|]^{1/2}}\left\{\frac{\langle k,\eta\rangle^{1/2}}{\langle t\rangle^{1+\sigma_0}}+\left|\frac{\partial_tw_k(t,\eta)}{w_k(t,\eta)}\right|\frac{e^{\sqrt\delta\langle\eta\rangle^{1/2}}}{e^{\sqrt\delta\langle\eta\rangle^{1/2}}+e^{\sqrt\delta|k|^{1/2}}w_k(t,\eta)}\right\}.
\end{split}
\end{equation}
The first term in the left-hand side of \eqref{TLX43} is easily bounded as claimed. Also, the second term is suitably bounded if $|k|\leq|\eta|$, as a consequence of \eqref{TLX2.9} and \eqref{TLX3}, or if $|\xi|\leq\delta^{-10}$, or if $t\geq 2|\xi|$. In the remaining case $(|k|\geq|\eta|,\,|\xi|>\delta^{-10},\,t\leq 2|\xi|)$, the right-hand side of \eqref{TLX43} is bounded from below by $e^{12\sqrt\delta|\eta|^{1/2}}\frac{\langle k,\eta\rangle^{1/2}}{\langle t\rangle^{1+\sigma_0}}$, which easily suffices to prove \eqref{TLX43}. The desired bounds \eqref{TLX42} follow in all cases.

We prove now the harder bounds \eqref{TLX41}. We consider several cases:

{\bf{Case 1.}} We start with the harder case
\begin{equation}\label{TLX44}
|\eta|\geq|\rho|\qquad\text{ and }\qquad |t-\eta/k|\leq |t|/10.
\end{equation}
It suffices to prove that
\begin{equation}\label{TLX46}
\frac{A^2_{NR}(t,\eta+\rho)}{|\dot{A}_{NR}(t,\eta+\rho)|}\langle t\rangle^{3/4}\langle\eta\rangle^{1/4}\lesssim _\delta A_k(t,\eta)\langle t-\eta/k\rangle^2A_{-k}(t,\rho)e^{-(\lambda(t)/20)[\langle \rho\rangle+|k|]^{1/2}}
\end{equation}
for any $t\geq 1$, $k\in\Z\setminus\{0\}$, and $\rho,\eta\in\R$ with $|\eta|\ge|\rho|$. The definitions \eqref{dor3}--\eqref{dor4} show that
\begin{equation}\label{TLX46.5}
\frac{A_{NR}(t,\eta+\rho)}{A_k(t,\eta)A_{-k}(t,\rho)}\lesssim \frac{b_k(t,\eta)b_{-k}(t,\rho)}{b_{NR}(t,\eta+\rho)}e^{-(\lambda(t)/8)\langle \rho,k\rangle^{1/2}}\lesssim_\delta \frac{b_k(t,\eta)}{b_{NR}(t,\eta)}e^{-(\lambda(t)/9)\langle \rho,k\rangle^{1/2}},
\end{equation}
where we used \eqref{b>a} in the first inequality and \eqref{dor22} in the second inequality. In view of \eqref{vfc30} and \eqref{dor20}, for \eqref{TLX46} it suffices to prove that
\begin{equation}\label{TLX47}
\frac{A_{NR}(t,\eta)}{|\dot{A}_{NR}(t,\eta)|}\frac{w_k(t,\eta)}{w_{NR}(t,\eta)}\langle t\rangle^{3/4}\langle\eta\rangle^{1/4}\lesssim_\delta  \langle t-\eta/k\rangle^2e^{(\lambda(t)/20)|k|^{1/2}}.
\end{equation}

Notice that, for $t\geq 1$, $\eta\in\mathbb{R}$, and $k\in\Z\setminus\{0\}$ we have
\begin{equation}\label{TLX48}
\frac{w_k(t,\eta)}{w_{NR}(t,\eta)}\lesssim_\delta \frac{\langle t-\eta/k\rangle}{\langle\eta\rangle}|k|^4.
\end{equation}
Indeed, the left-hand side is bounded by $1$, so \eqref{TLX48} is trivial unless $\{|\eta|>\delta^{-10},\,|k|\leq |\eta|^{1/4},|t-\eta/k|\leq |\eta|/(20k^2)\}$. In this case, however, the desired bounds \eqref{TLX48} follow from the definitions \eqref{eq:resonantweight} and \eqref{reb5.5}.

Moreover, we show that if $t,\eta,k$ are as above and satisfy $|t-\eta/k|\leq |t|/10$ then
\begin{equation}\label{TLX49}
\frac{A_{NR}(t,\eta)}{|\dot{A}_{NR}(t,\eta)|}\lesssim_\delta\langle t-\eta/k\rangle |k|^4.
\end{equation}
Indeed, as a consequence of \eqref{TLX3.5}
\begin{equation}\label{TLX49.5}
\frac{A_{NR}(t,x)}{|\dot{A}_{NR}(t,x)|}\lesssim_\delta \langle t\rangle^{1+\sigma_0}\langle x\rangle^{-1/2},\qquad\text{ for any }t\in[1,\infty),\,x\in\R.
\end{equation}
This suffices unless $\{|\eta|>\delta^{-10},\,|k|\leq |\eta|^{1/4},|t-\eta/k|\leq |\eta|/(20k^2)\}$. In this case, however, the left-hand side of \eqref{TLX49} is bounded by $w_{NR}(t,\eta)/\partial_tw_{NR}(t,\eta)$, which suffices due to \eqref{reb8}.

Notice that \eqref{TLX47} easily follows from \eqref{TLX48} and \eqref{TLX49}, which completes the proof of the bounds \eqref{TLX41} in this case.

{\bf{Case 2.}} We assume now that
\begin{equation}\label{TLX54}
|\eta|\geq|\rho|\qquad\text{ and }\qquad |t-\eta/k|\geq |t|/10.
\end{equation}
It suffices to prove that
\begin{equation}\label{TLX56}
\frac{A^2_{NR}(t,\eta+\rho)}{|\dot{A}_{NR}(t,\eta+\rho)|}\langle t\rangle^{3/4}\langle\eta\rangle^{1/4}\lesssim_\delta  A_k(t,\eta)\langle t\rangle(\langle t\rangle+|\eta/k|)A_{-k}(t,\rho)e^{-(\lambda(t)/20)[\langle \rho\rangle+|k|]^{1/2}}.
\end{equation}
We use again the bounds \eqref{TLX46.5} and further estimate $\frac{b_k(t,\eta)}{b_{NR}(t,\eta)}\lesssim_\delta 1$. In view of \eqref{vfc30}, for \eqref{TLX56} it suffices to prove that
\begin{equation}\label{TLX57}
\frac{A_{NR}(t,\eta)}{|\dot{A}_{NR}(t,\eta)|}\langle t\rangle^{3/4}\langle\eta\rangle^{1/4}\lesssim_\delta  \langle t\rangle(\langle t\rangle+|\eta/k|)e^{(\lambda(t)/20)|k|^{1/2}}.
\end{equation}
This follows easily using \eqref{TLX49.5}, which completes the proof of the bounds \eqref{TLX41} in this case.

{\bf{Case 3.}} Finally, assume that
\begin{equation}\label{TLX64}
|\eta|\leq|\rho|.
\end{equation}
It suffices to prove that
\begin{equation}\label{TLX66}
\frac{A^2_{NR}(t,\eta+\rho)}{|\dot{A}_{NR}(t,\eta+\rho)|}\langle t\rangle^{3/4}\langle\rho\rangle^{1/4}\lesssim_\delta  A_k(t,\eta)\frac{\langle t\rangle\langle t-\eta/k\rangle^2}{\langle t\rangle+|\eta/k|}A_{-k}(t,\rho)e^{-(\lambda(t)/20)(\langle \eta\rangle+|k|)^{1/2}}.
\end{equation}
The definitions \eqref{dor3}--\eqref{dor4} show that
\begin{equation}\label{TLX66.5}
\frac{A_{NR}(t,\eta+\rho)}{A_k(t,\eta)A_{-k}(t,\rho)}\lesssim \frac{b_k(t,\eta)b_{-k}(t,\rho)}{b_{NR}(t,\eta+\rho)}e^{-(\lambda(t)/8)\langle \eta,k\rangle^{1/2}}\lesssim_\delta e^{-(\lambda(t)/9)\langle \eta,k\rangle^{1/2}},
\end{equation}
where we used \eqref{b>a} in the first inequality, and \eqref{dor22} and the bounds $b_\ell(t,.)\leq b_{NR}(t,.)\leq 1$ in the second inequality. For \eqref{TLX66} it suffices to prove that 
\begin{equation*}
\frac{A_{NR}(t,\rho)}{|\dot{A}_{NR}(t,\rho)|}\langle t\rangle^{3/4}\langle\rho\rangle^{1/4}\lesssim_\delta  \frac{\langle t\rangle\langle t-\eta/k\rangle^2}{\langle t\rangle+|\eta/k|}e^{(\lambda(t)/20)\langle \eta,k\rangle^{1/2}}.
\end{equation*}
which follows again from \eqref{TLX49.5}. This completes the proof of the bounds \eqref{TLX41}.
\end{proof}

\begin{lemma}\label{TLX70}
For any $t\geq 1$ and $\xi,\eta\in\R$ we have, with $\rho=\xi-\eta$,
\begin{equation}\label{TLX71}
\frac{A^2_{NR}(t,\xi)}{|\dot{A}_{NR}(t,\xi)|}\frac{\langle t\rangle^{3/4}}{\langle\xi\rangle^{3/4}}\lesssim_\delta \frac{A^2_{NR}(t,\eta)}{|\dot{A}_{NR}(t,\eta)|}\frac{\langle t\rangle^{3/4}}{\langle\eta\rangle^{3/4}}A_{NR}(t,\rho)e^{-(\lambda(t)/40)\min(\langle\xi-\eta\rangle,\langle \eta\rangle)^{1/2}}.
\end{equation}
\end{lemma}

\begin{proof} The bound follows easily if $|\rho|\leq |\eta|$, as a consequence of \eqref{TLX4} and \eqref{vfc30}. On the other hand, if $|\eta|\leq |\rho|$ we can still use \eqref{TLX4} and it suffices to show that
\begin{equation}\label{TLX71.5}
\frac{A_{NR}(t,\xi)}{|\dot{A}_{NR}(t,\xi)|}\frac{1}{\langle\xi\rangle^{3/4}}\lesssim_\delta \frac{A_{NR}(t,\eta)}{|\dot{A}_{NR}(t,\eta)|}\frac{1}{\langle\eta\rangle^{3/4}}e^{(\lambda(t)/40)\langle \eta\rangle^{1/2}}.
\end{equation}
The bound \eqref{TLX71.5} follows from \eqref{vfc30} if $|\xi-\eta|\leq 100|\eta|$. On the other hand, if $|\eta|\leq|\xi|/50$, then \eqref{TLX71.5} is equivalent to 
\begin{equation*}
\frac{|\dot{A}_{NR}(t,\eta)|}{A_{NR}(t,\eta)|}\lesssim_\delta \frac{|\dot{A}_{NR}(t,\xi)|}{A_{NR}(t,\xi)}\frac{\langle\xi\rangle^{3/4}}{\langle\eta\rangle^{3/4}}e^{(\lambda(t)/40)\langle \eta\rangle^{1/2}},
\end{equation*}
which is easy to prove using \eqref{TLX3.5} and \eqref{reb8}.
\end{proof}

\begin{lemma}\label{TLY1}
Assume that $t\geq 1$ and recall the definition of the sets $S_0,S_1,S_2,S_3$ in \eqref{tol4}. 

(i) If $(\xi,\eta)\in S_0\cup S_1$, $\alpha\in[0,4]$, and $\ast\in\{NR,R\}$ then, with $\rho=\xi-\eta$,
\begin{equation}\label{TLY2.1}
\begin{split}
\big|\eta &A^2_{\ast}(t,\xi)\langle\xi\rangle^{-\alpha}-\xi A_{\ast}^2(t,\eta)\langle\eta\rangle^{-\alpha}\big|\\
&\lesssim_\delta t^{1.6}\frac{\sqrt{|(A_\ast\dot{A}_\ast)(t,\xi)|}}{\langle\xi\rangle^{\alpha/2}}\frac{\sqrt{|(A_\ast\dot{A}_\ast)(t,\eta)|}}{\langle\eta\rangle^{\alpha/2}}\cdot A_{NR}(t,\rho)|\rho|e^{-(\lambda(t)/40)\langle\rho\rangle^{1/2}}.
\end{split}
\end{equation}

(ii) If $(\xi,\eta)\in S_2$ then
\begin{equation}\label{TLY2.2}
\langle\eta\rangle A^2_{R}(t,\xi)\lesssim_\delta t^{1.1}\langle\xi\rangle^{0.6}\sqrt{|(A_R\dot{A}_R)(t,\xi)|}\sqrt{|(A_{NR}\dot{A}_{NR})(t,\rho)|}\cdot A_{R}(t,\eta)e^{-(\lambda(t)/40)\langle\eta\rangle^{1/2}}
\end{equation}
and
\begin{equation}\label{TLY2.3}
\begin{split}
\langle\eta\rangle A^2_{NR}(t,\xi)&\lesssim_\delta t^{1.1}\langle\xi\rangle^{-0.4}\sqrt{|(A_{NR}\dot{A}_{NR})(t,\xi)|}\sqrt{|(A_{NR}\dot{A}_{NR})(t,\rho)|}\cdot A_{NR}(t,\eta)e^{-(\lambda(t)/40)\langle\eta\rangle^{1/2}}.
\end{split}
\end{equation}
\end{lemma}

\begin{proof} (i) Notice that
\begin{equation}\label{TLY3}
\frac{|\partial_\xi A_\ast(t,\xi)|}{A_\ast(t,\xi)}\lesssim_\delta \frac{\langle\xi\rangle^{1/2}+t}{\langle\xi\rangle+t}
\end{equation}
for any $t\geq 1$, $\xi\in\R$, and $\ast\in\{R,NR\}$, as a consequence of \eqref{dor21} and the definitions. 

If $(\xi,\eta)\in S_0$ then there is no derivative loss, and the estimates \eqref{TLY2.1} follow directly from \eqref{vfc25} and \eqref{TLX3.5}. On the other hand, if $(\xi,\eta)\in S_1$, then $|\xi-\eta|$ is small, and we use \eqref{TLY3}. Recalling \eqref{vfc25}, we have
\begin{equation*}
\begin{split}
|\eta A^2_{\ast}(t,\xi)\langle\xi\rangle^{-\alpha}&-\xi A_{\ast}^2(t,\eta)\langle\eta\rangle^{-\alpha}|\\
&\lesssim_\delta |\rho|\langle\xi\rangle^{1-\alpha}\frac{\langle\xi\rangle^{1/2}+t}{\langle\xi\rangle+t}A_\ast(t,\xi)A_\ast(t,\eta)\cdot A_{NR}(t,\rho)e^{-(\lambda(t)/20)\langle\rho\rangle^{1/2}}.
\end{split}
\end{equation*}
For \eqref{TLY2.1} it suffices to show that, for any $(\xi,\eta)\in S_1$,
\begin{equation}\label{TLY4}
\langle\xi\rangle\frac{\langle\xi\rangle^{1/2}+t}{\langle\xi\rangle+t}\lesssim_\delta t^{1.6}\sqrt{|(\dot{A}_\ast/A_\ast)(t,\xi)|}\sqrt{|(\dot{A}_\ast/A_\ast)(t,\eta)|}\cdot e^{(\lambda(t)/40)\langle\xi-\eta\rangle^{1/2}}.
\end{equation}

Using now \eqref{TLX3.5}, we have 
\begin{equation*}
\sqrt{|(\dot{A}_\ast/A_\ast)(t,\xi)|}\sqrt{|(\dot{A}_\ast/A_\ast)(t,\eta)|}\gtrsim \frac{\langle\xi\rangle^{1/2}}{t^{1+\sigma_0}}.
\end{equation*}
The bounds \eqref{TLY4} follow by checking the cases $t\geq \langle\xi\rangle$, $\langle\xi\rangle\geq t\geq \langle\xi\rangle^{1/2}$, and $\langle\xi\rangle^{1/2}\geq t\geq 1$.

(ii) To prove \eqref{TLY2.2} we use first \eqref{vfc25} and \eqref{vfc30}, therefore
\begin{equation}\label{TLY4.5}
\sqrt{|(A_{NR}\dot{A}_{NR})(t,\xi-\eta)|}\gtrsim_\delta \sqrt{|(A_{NR}\dot{A}_{NR})(t,\xi)|}e^{-0.91\lambda(t)|\eta|^{1/2}}.
\end{equation}
Using this, for \eqref{TLY2.2} it suffices to prove that
\begin{equation*}
A^2_{R}(t,\xi)\lesssim_\delta t^{1.1}\langle\xi\rangle^{0.6}\sqrt{|(A_R\dot{A}_R)(t,\xi)|}\sqrt{|(A_{NR}\dot{A}_{NR})(t,\xi)|}
\end{equation*}
for any $t\geq 1$ and $\xi\in\R$ satisfying $|\xi|\geq 1$. In view of \eqref{vfc30.5}, this is equivalent to proving that
\begin{equation}\label{TLY5}
\frac{A_{NR}(t,\xi)}{A_R(t,\xi)}\frac{|\dot{A}_{R}(t,\xi)|}{A_{R}(t,\xi)}t^{1.1}\langle\xi\rangle^{0.6}\gtrsim_\delta 1.
\end{equation}
It follows from \eqref{reb5.5}, \eqref{dor20},  and the definitions that
\begin{equation*}
\frac{A_{NR}(t,\xi)}{A_R(t,\xi)}\gtrsim \frac{b_R(t,\xi)}{b_{NR}(t,\xi)}\gtrsim_\delta \frac{w_R(t,\xi)}{w_{NR}(t,\xi)}\gtrsim_\delta \langle\xi\rangle^{-1}.
\end{equation*}
The bounds \eqref{TLY5} follow using also \eqref{TLX3.5}. This completes the proof of \eqref{TLY2.2}.

To prove \eqref{TLY2.3} we use again \eqref{TLY4.5}; it remains to show that
\begin{equation*}
A^2_{NR}(t,\xi)\lesssim_\delta t^{1.1}\langle\xi\rangle^{-0.4}\sqrt{|(A_{NR}\dot{A}_{NR})(t,\xi)|}\sqrt{|(A_{NR}\dot{A}_{NR})(t,\xi)|},
\end{equation*}
for any $t\geq 1$ and $\xi\in\R$ satisfying $|\xi|\geq 1$. This follows easily from \eqref{TLX3.5}.
\end{proof}

\appendix
\section{Gevrey spaces and local wellposedness of Euler equations}\label{appendix}

In this section we review some general properties of the Gevrey spaces of functions and prove the local well-posedness result in Lemma \ref{lm:persistenceofhigherregularity}.

\subsection{The Gevrey spaces}\label{GevSec}  We start with a characterization of the Gevrey spaces on the physical side.

\begin{lemma}\label{lm:Gevrey}
(i) Suppose that $0<s<1$, $K>1$, and $f\in C_0^{\infty}(\R^d)$ satisfies the bounds
\begin{equation}\label{growth}
\big|D^{\alpha}f(x)\big|\leq K^{m}(m+1)^{m/s},
\end{equation}
for all integers $m\ge 0$ and multi-indeces $\alpha$ with $|\alpha|=m$. Assume that $({\rm supp}\,f)\subseteq[-L,\,L]^d$, $L\geq 1$. Then
\begin{equation}\label{gevrey}
\big|\widehat{f}(\xi)\big|\lesssim_{K,s} L^de^{-\mu|\xi|^s},
\end{equation}
for all $\xi\in \R^d$ and some $\mu=\mu(K,s)>0$.

Similarly, if $f\in C^{\infty}(\mathbb{T}\times \mathbb{R})$ with ${\rm supp}\,f\subseteq \mathbb{T}\times[0,1]$ satisfies \eqref{growth}, then
\begin{equation}\label{gevreyP}
\big|\widetilde{f}(k,\xi)\big|\lesssim_{K,s} Le^{-\mu|k,\xi|^s},
\end{equation}
for all $k\in\mathbb{Z}, \xi\in \R$ and some $\mu=\mu(K,s)>0$.

(ii) Conversely, assume that, for some $\mu>0$ and $s\in(0,1)$,
\begin{equation}\label{eq:fou}
\big|\widehat{f}(\xi)\big|\leq e^{-\mu |\xi|^s},
\end{equation}
for all $\xi\in \R^d$. Then there is $K>1$ depending on $s$ and $\mu$ such that
\begin{equation}\label{eq:four}
\left|D^{\alpha}f(x)\right|\lesssim_{\mu,s} K^m(m+1)^{m/s},
\end{equation}
for all multi-indices $\alpha$ with $|\alpha|=m$.

Similarly, if $f\in C^{\infty}(\mathbb{T}\times \mathbb{R})$ satisfies, for some $\mu>0,s\in(0,1)$,
\begin{equation}\label{eq:fouP}
\big|\widetilde{f}(k,\xi)\big|\leq e^{-\mu |k,\xi|^s},
\end{equation}
for all $k\in\mathbb{Z}, \xi\in \R$, then the bounds \eqref{eq:four} hold for some $K=K(\mu,s)>0$ and all $x\in \mathbb{T}\times \mathbb{R}$.
\end{lemma}

\begin{proof} (i) We prove only the harder estimates \eqref{gevrey}. We may assume that $|\xi|$ is large. Using the definition of the Fourier transform, integration by parts, and the bounds \eqref{growth}, we see that
\begin{equation}\label{eq:N}
\big|\widehat{f}(\xi)\big|\leq \frac{C^N}{|\xi|^N}K^N(N+1)^{N/s}L^d.
\end{equation}
This holds for all integers $N\geq 1$. Choose $N$ to be the largest integer so that $CK(N+1)^{1/s}\leq |\xi|/e$, thus
$$N=\frac{|\xi|^s}{(CKe)^s}+O(1).$$
Consequently, using \eqref{eq:N} we get that
$$\big|\widehat{f}(\xi)\big|\lesssim_s L^de^{-N}\lesssim_{K,\mu}L^de^{-\mu|\xi|^s},$$
for suitable $\mu>0$.

(ii) We consider only the case when $f\in C^\infty(\R^d)$. Using \eqref{eq:fou} we have 
\begin{equation}\label{eq:fou2}
\|D^\alpha f\|_{L^\infty}\leq C_0^m\|\langle\xi\rangle^m\widehat{f}(\xi)\|_{L^1}\leq C_1^m\big(1+\sup_{|\xi|\geq 2}(|\xi|^{m+d+1}e^{-\mu|\xi|^s})\big).
\end{equation}
We notice that the function $r\to r^Ne^{-r}$, $r>0$, $N\geq 1$, has a maximum at $r=N$. Thus 
\begin{equation}\label{eq:fou5}
r^Ne^{-r}\leq (N/e)^N \qquad\text{ for any }r>0\text{ and }N\geq 1,
\end{equation}
so the right-hand side of \eqref{eq:fou2} is bounded by
\begin{equation*}
C_1^m\big[1+\sup_{r>0}(r/\mu)^{(1/s)\cdot (m+d+1)}e^{-r}\big]\leq K_1^m(N/e)^N,
\end{equation*}
where $N=(m+d+1)/s$ and $K_1$ is sufficiently large. The desired bounds \eqref{eq:four} follow.
\end{proof}

\subsubsection{Gevrey cutoff functions}\label{GevCut} Using Lemma \ref{lm:Gevrey}, one can construct explicit cutoff functions in Gevrey spaces. For $a>0$ let
\begin{equation}\label{gev1}
\psi_a(x):=\begin{cases}
e^{-[1/x^a+1/(1-x)^a]}&\quad\text{ if }x\in[0,1],\\
0&\quad\text{ if }x\notin[0,1].
\end{cases}
\end{equation}
Clearly $\psi_a$ are smooth functions on $\R$, supported in the interval $[0,1]$. Using \eqref{eq:fou5} it is easy to verify that $\psi_a$ satisfies the bounds \eqref{growth} for $s:=a/(a+1)$. Thus, for some $\mu=\mu(a)>0$,
\begin{equation}\label{gev2}
|\widehat{\psi_a}(\xi)|\lesssim e^{-\mu|\xi|^{a/(a+1)}}.
\end{equation} 

One can also construct compactly supported Gevrey cutoff functions which are equal to $1$ in a given interval. Indeed, for any $\rho\in[9/10,1)$, the function
\begin{equation}\label{gev3}
\psi'_{a,\rho}(x):=\frac{\psi_a(x)}{\psi_a(x)+\psi_a(x-\rho)+\psi_a(x+\rho)}
\end{equation}
is smooth, non-negative, supported in $[0,1]$, and equal to $1$ in $[1-\rho,\rho]$. Moreover, it follows from Lemma \ref{lm:Gevrey} (i) that $|\widehat{\psi'_{a,\rho}}(\xi)|\lesssim e^{-\mu|\xi|^{a/(a+1)}}$ for some $\mu=\mu(a,\rho)>0$. 

\subsubsection{Compositions of Gevrey functions} The physical space characterization of Gevrey functions is useful when studying compositions. In our setting, we have the following lemma:

\begin{lemma}\label{GPF}
(i) Assume $\kappa_1>0$, $s\in(0,1)$, and $f\in \mathcal{G}^{\kappa_1,s}(\mathbb{T}\times \R)$. Suppose $M\in(0,\infty)$ and $g:\mathbb{T}\times \R \to \mathbb{T}\times \R$ satisfies, for any $m\geq 1$,
\begin{equation}\label{gbo1}
|D^\alpha g(x,y)|\leq M^m(m+1)^{m/s}\qquad \text{ for any }(x,t)\in\T\times\R\text{ and }|\alpha|\in [1,m].
\end{equation}
Suppose that $f$ and $f\circ g$ are supported in $\mathbb{T}\times [-2,2]$. Then, for a suitable $\kappa_2>0$ depending on $s,\kappa_1,M$,
we have $f\circ g\in \mathcal{G}^{\kappa_2,s}$ and
\begin{equation}\label{Ffgc}
\left\|f\circ g\right\|_{ \mathcal{G}^{\kappa_2,s}}\lesssim_{s,\kappa_1,M} \left\|f\right\|_{ \mathcal{G}^{\kappa_1,s}}.
\end{equation}

(ii) Assume $s\in(0,1)$, $L\in(0,\infty)$, $I,J\subseteq\mathbb{R}$ are open intervals, and $g:I\to J$ is a smooth bijective map satisfying, for any $m\geq 1$,
\begin{equation}\label{gbo2}
|D^\alpha g(x)|\leq L^m(m+1)^{m/s}\qquad \text{ for any }x\in I\text{ and }|\alpha|\in [1,m].
\end{equation}
If $|g'(x)|\geq 1/10$ for any $x\in I$ then the inverse function $g^{-1}:J\to I$ satisfies the bounds
\begin{equation}\label{gbo2.1}
|D^\alpha (g^{-1})(x)|\leq M^m(m+1)^{m/s}\qquad \text{ for any }x\in J\text{ and }|\alpha|\in [1,m],
\end{equation}
for some constant $M=M(L,s)\geq L$.
\end{lemma}

This can be proved using Lemma \ref{lm:Gevrey}, and we omit the details. See also Theorem 6.1 and Theorem 3.2 of \cite{Yamanaka} for more general estimates on compositions of functions in Gevrey spaces.

\subsection{Proof of Lemma \ref{lm:persistenceofhigherregularity}} As we remarked earlier, the lemma can be obtained as a consequence of the more general theory developed in \cite{Foias,Levermore,Vicol2,Vicol}. For the sake of convenience, we provide a complete proof here in our special case, using the Fourier transform. 

To prove Gevrey bounds,  we have to work with suitable weights. First we define the functions $g:(0,\infty)\to (0,\infty)$, by
\begin{equation}\label{weg1}
\begin{split}
g'(r):=
\begin{cases}
s\, r^{s-1}-s\rho^{s-1}&\quad\text{ if }r\in(0,\rho],\\
0&\quad\text{ if }r\geq \rho,
\end{cases}
\qquad g(r):=\int_0^r g'(x)\,dx.
\end{split}
\end{equation}
Here $\rho\geq \rho_0$ is a large parameter (which is needed only to guarantee convergence and continuity in time of the energy functionals below), and the desired Gevrey bounds follow by proving estimates independent of $\rho$ and letting $\rho\to\infty$. 

Then we define the main weights $B:[0,T]\times\mathbb{R}^2\to (0,\infty)$,
\begin{equation}\label{ini5}
B(t,v):=\langle v\rangle^3\exp[\lambda(t)g(\langle v\rangle)],
\end{equation}
where $\lambda(t):[0,T]\to (0,\infty)$, $\lambda(0)=\lambda_0$, is a positive decreasing function to be chosen below.

With $v=(k,\xi)$, we define the energy functionals
\begin{equation}\label{ini6}
\mathcal{E}(t):=\sum_{k\in\mathbb{Z}}\int_{\mathbb{R}}B^2(t,k,\xi)\big|\widetilde{\omega}(t,k,\xi)\big|^2\,d\xi.
\end{equation}
Since $\omega\in C([0,T]:H^{10})$ and $B(t,v)\lesssim_\rho \langle v\rangle^{3}$, the function $\mathcal{E}$ is well-defined and continuous on $[0,T]$. Moreover, $\mathcal{E}(0)\leq \|\langle \nabla\rangle^3\omega_0\|_{\mathcal{G}^{\lambda_0,s}}$.

{\bf{Step 1.}} We fix a smooth function $\Psi(y)$ with 
\begin{equation}\label{ini9}
{\mathrm {supp}}\,\Psi\subseteq [\vartheta/8,1-\vartheta/8],\qquad\Psi|_{[\vartheta/4,1-\vartheta/4]}\equiv 1,\qquad |\widetilde{\Psi}(\xi)|\lesssim e^{-\langle \xi\rangle^{7/8}}\text{ for all }\xi\in\R,
\end{equation}
where, in the rest of this section, all implicit constants are allowed to depend on $\vartheta$. By the support property of $\omega(t)$ and the Euler equation, we calculate
\begin{equation}\label{ini12}
\begin{split}
\frac{d}{dt}\mathcal{E}(t)&=\sum_{k\in\mathbb{Z}}\int_{\mathbb{R}}2B(t,k,\xi)\dot{B}(t,k,\xi)\big|\widetilde{\omega}(t,k,\xi)\big|^2\,d\xi\\
&+2\Re\,\sum_{k\in\mathbb{Z}}\int_{\mathbb{R}}B^2(t,k,\xi)\,\overline{\widetilde{\omega}(t,k,\xi)}\,\widetilde{\partial_t\omega}(t,k,\xi)\,d\xi\\
&=\sum_{k\in\mathbb{Z}}\int_{\mathbb{R}}2B(t,k,\xi)\dot{B}(t,k,\xi)\big|\widetilde{\omega}(t,k,\xi)\big|^2\,d\xi+\mathcal{P}^1(t)+\mathcal{P}^2(t),
\end{split}
\end{equation}
where, using the equations and inserting a factor of $\Psi(y)$,
\begin{equation*}
\begin{split}
\mathcal{P}^1(t)&:=C\Re\sum_{k\in\Z}\int_{\R^2}B^2(t,k,\xi)\,\overline{\widetilde{\omega}(t,k,\xi)}\cdot {\widetilde{(y\Psi)}(\xi-\eta)\,ik\,\widetilde{\omega}(t,k,\eta)}\,d\xi\,d\eta,\\
&=C\Re\sum_{k\in\Z}\int_{\R^2}ik[B^2(t,k,\xi)-B^2(t,k,\eta)]\,\overline{\widetilde{\omega}(t,k,\xi)}\,\widetilde{\omega}(t,k,\eta)\widetilde{(y\Psi)}(\xi-\eta)\,d\xi\,d\eta,
\end{split}
\end{equation*}
and
\begin{equation*}
\begin{split}
\mathcal{P}^2&(t):=C\Re\sum_{k,\ell\in\Z}\int_{\R^2}B^2(t,k,\xi)\,\overline{\widetilde{\omega}(t,k,\xi)}\,\Big\{-\widetilde{\partial_y(\Psi\psi)}(t,k-\ell,\xi-\eta)\,i\ell\,\widetilde{\omega}(t,\ell,\eta)\\
&\qquad\qquad+\widetilde{\partial_x(\Psi\psi)}(t,k-\ell,\xi-\eta)\,i\eta\,\widetilde{\omega}(t,\ell,\eta)\Big\}\,d\xi d\eta\\
&=C\Re\sum_{k,\ell\in\Z}\int_{\R^2}i[\ell B^2(t,k,\xi)-kB^2(t,\ell,\eta)]\,\overline{\widetilde{\omega}(t,k,\xi)}\widetilde{\omega}(t,\ell,\eta)\widetilde{\partial_y(\Psi\psi)}(t,k-\ell,\xi-\eta)\,d\xi d\eta\\
&+C\Re\sum_{k,\ell\in\Z}\int_{\R^2}i[\eta B^2(t,k,\xi)-\xi B^2(t,\ell,\eta)]\,\overline{\widetilde{\omega}(t,k,\xi)}\widetilde{\omega}(t,\ell,\eta)\widetilde{\partial_x(\Psi\psi)}(t,k-\ell,\xi-\eta)\,d\xi d\eta.
\end{split}
\end{equation*}

{\bf{Step 2.}} We estimate now $|\mathcal{P}^1(t)|$ and $|\mathcal{P}^2(t)|$. Using the bounds in Lemma \ref{ineq0}, we see that
\begin{equation}\label{ini50}
\begin{split}
|\ell B^2(t,k,\xi)-&kB^2(t,\ell,\eta)|+|\eta B^2(t,k,\xi)-\xi B^2(t,\ell,\eta)|\lesssim B(t,k,\xi)B(t,\ell,\eta)\\
&\times B(t,k-\ell,\xi-\eta)(\min)^{-3}[1+\lambda(t)g(\langle k,\xi\rangle)]\langle k-\ell,\xi-\eta\rangle,
\end{split}
\end{equation}
where $\min:=\min\{\langle k,\xi\rangle,\langle\ell,\eta\rangle,\langle k-\ell,\xi-\eta\rangle\}$. Indeed, if $|k-\ell,\xi-\eta|\geq (|k,\xi|+|\ell,\eta|)/20$ then the bounds follow directly from \eqref{ineq3}, by estimating each term in the left-hand side independently. On the other hand, if $|k-\ell,\xi-\eta|\leq (|k,\xi|+|\ell,\eta|)/20$ then we use \eqref{ineq2}, which forces us to include the term $\lambda(t)g(\langle k,\xi\rangle)$ in the right-hand side of \eqref{ini50}.

Let
\begin{equation}\label{ini50.1}
\begin{split}
H(t,k,\xi)&:=\langle k,\xi\rangle^2|\widetilde{\Psi\psi}(t,k,\xi)|+|\widetilde{\omega}(t,k,\xi)|,\\
H'(\xi)&:=\langle \xi\rangle|\widetilde{y\Psi}(t,k,\xi)|.
\end{split}
\end{equation}
It follows from \eqref{ini50}, changes of variables, and the identities above that
\begin{equation}\label{ini50.20}
\begin{split}
|\mathcal{P}^1(t)|&\lesssim \sum_{k\in\Z}\int_{\R^2}|\widetilde{\omega}(t,k,\xi)||\widetilde{\omega}(t,k,\eta)||H'(\xi-\eta)|\cdot B(t,k,\xi)B(t,k,\eta)\\
&\times B(t,0,\xi-\eta)[1+\lambda(t)g(\langle k,\xi\rangle)]\,d\xi d\eta, 
\end{split}
\end{equation}
\begin{equation}\label{ini50.21}
\begin{split}
|\mathcal{P}^2(t)|&\lesssim \sum_{k,\ell\in\Z}\int_{\R^2}|H(t,k,\xi)||H(t,\ell,\eta)||H(t,k-\ell,\xi-\eta)|\\
&\times B(t,k,\xi)B(t,\ell,\eta)B(t,k-\ell,\xi-\eta)\mathbf{1}_{\mathcal{R}}((k,\xi),(\ell,\eta))\\
&\times \langle k-\ell,\xi-\eta\rangle^{-3}[1+\lambda(t)g(\langle k,\xi\rangle)]\,d\xi d\eta,
\end{split}
\end{equation}
where $\mathcal{R}:=\big\{\langle k-\ell,\xi-\eta\rangle\leq\min\{\langle k,\xi\rangle,\langle \ell,\eta\rangle\}\big\}$.

To estimate $|\mathcal{P}^1(t)|$ and $|\mathcal{P}^2(t)|$ we use the following elementary bounds (compare with Lemma \ref{Multi0}): for any functions $F_1,F_2,F_3:\Z\times\R\to[0,\infty)$ we have
\begin{equation}\label{ini50.4}
\sum_{k,\ell\in\Z}\int_{\R^2}F_1(k,\xi)F_2(\ell,\eta)F_3(k-\ell,\xi-\eta)\,d\xi d\eta\lesssim \|F_1\|_{L^2}\|F_2\|_{L^2}\|F_3\|_{L^1}.
\end{equation}
For  $t\in[0,T]$ let 
\begin{equation}\label{ini80}
Y(t):=\big\|B(t,k,\xi)\,\widetilde{\omega}(t,k,\xi)\big\|_{L^2_{k,\xi}},\qquad Y'(t):=\big\|\sqrt{g(\langle k,\xi\rangle)}B(t,k,\xi)\,\widetilde{\omega}(t,k,\xi)\big\|_{L^2_{k,\xi}}.
\end{equation}
Using Lemma \ref{ineq6}, we have
\begin{equation}\label{ini50.3}
\big\|W(t,k,\xi)H(t,k,\xi)\big\|_{L^2_{k,\xi}}\lesssim\big\|W(t,k,\xi)\,\widetilde{\omega}(t,k,\xi)\big\|_{L^2_{k,\xi}},
\end{equation}
provided that $W(t,k,\xi)=B(t,k,\xi)$ or $W(t,k,\xi)=\sqrt{g(\langle k,\xi\rangle)}B(t,k,\xi)$.

Using \eqref{ini50.20}, \eqref{ini9}, and \eqref{ini50.4}, we have, for any $t\in[0,T]$,
\begin{equation}\label{ini50.5}
|\mathcal{P}^1(t)|\lesssim (Y(t))^2+\lambda(t)(Y'(t))^2
\end{equation}

To estimate $\mathcal{P}^2$, we first observe that $B(t,v)\leq \langle v\rangle^3+\lambda(t)g(\langle v\rangle)B(t,v)$ with $v=(k,\xi)$. As a consequence, we have
\begin{equation}\label{BtL4}
B(t,v) [1+\lambda(t)g(\langle v\rangle)]\lesssim \langle v\rangle^3+\lambda(t)g(\langle v\rangle)B(t,v).
\end{equation}

Using \eqref{ini50.21}--\eqref{ini50.3} and (\ref{BtL4}), we estimate
\begin{equation}\label{ini50.6}
|\mathcal{P}^2(t)|\lesssim \lambda(t)Y(t)(Y'(t))^2+\|\omega(t)\|_{H^6}(Y(t))^2.
\end{equation}

{\bf{Step 3.}} We reexamine now the identities \eqref{ini12}. Notice that 
\begin{equation*}
\dot{B}(t,k,\xi)=\lambda'(t)g(\langle k,\xi\rangle)B(t,k,\xi).
\end{equation*}
We use \eqref{ini50.5}--\eqref{ini50.6} to estimate
\begin{equation*}
\begin{split}
\frac{d}{dt}\mathcal{E}(t)&\leq \lambda'(t)(Y'(t))^2+\mathcal{P}^1(t)+\mathcal{P}^2(t)\\
&\leq \lambda'(t)(Y'(t))^2+C_\vartheta\big\{(Y(t))^2+\lambda(t)(Y'(t))^2+\lambda(t)Y(t)(Y'(t))^2+\|\omega(t)\|_{H^6}(Y(t))^2\big\}.
\end{split}
\end{equation*}
Now suppose $\lambda(t)$ satisfies
\begin{equation}\label{LDY1}
\lambda'(t)+C_\vartheta (Y(t)+1)\lambda(t)\leq0,
\end{equation}
then we would obtain
\begin{equation*}
\frac{d}{dt}\mathcal{E}(t)\leq C_{\vartheta}(1+\|\omega(t)\|_{H^6})\mathcal{E}(t),
\end{equation*}
which gives
\begin{equation}\label{BonEA}
\mathcal{E}(t)\leq \mathcal{E}(0)\exp{\left[C_{\vartheta}\int_0^t(1+\|\omega(s)\|_{H^6})ds\right]}.
\end{equation}
From (\ref{BonEA}), we see that we can guarantee (\ref{LDY1}) if $\lambda(t)$ is chosen so that
\begin{equation}\label{ChOfL}
\lambda'(t)+4C_{\vartheta}\lambda(t)+4C_{\vartheta}Y(0)\exp{\left[C_{\vartheta}\int_0^t(1+\|\omega(s)\|_{H^6})ds\right]}\lambda(t)\leq0.
\end{equation}
The desired bounds (\ref{ini2}) and (\ref{fdLam1}) follow easily from (\ref{BonEA}), (\ref{ChOfL}), by sending $\rho\to\infty$.

We prove now suitable bounds on the weights $B$ which are used for \eqref{ini50}.

\begin{lemma}\label{ineq0}
For $v,w\in\Z\times\R$ we have
\begin{equation}\label{ineq3}
B(t,v+w)\lesssim B(t,v)B(t,w)\min(\langle v\rangle,\langle w\rangle)^{-3}.
\end{equation}
Moreover, if $|v|\leq |w|$ and $|w-v|\leq |v|/4$, then
\begin{equation}\label{ineq2}
\begin{split}
\big|B(t,w)-B(t,v)\big|&\lesssim[1+\lambda(t) g(\langle v\rangle)]\langle v\rangle^{-1}B(t,v)\cdot B(t,v-w)\langle v-w\rangle^{-2}.
\end{split}
\end{equation}
\end{lemma}

\begin{proof} For $\lambda\in[0,1]$, $s\in[1/4,3/4]$, $\rho\geq 2^{40}$, and $g$ as in \eqref{weg1}, we define the functions $f_\lambda:(0,\infty)\to (0,\infty)$ as
\begin{equation}\label{weg2}
f_\lambda (r):=\exp[\lambda g(r)].
\end{equation}

{\bf{Step 1.}} We show first that for $x,y\geq 1$ we have
\begin{equation}\label{ini32.5}
\frac{f_\lambda(x+y)}{f_\lambda(x)f_\lambda(y)}\leq 1.
\end{equation}
Indeed, we may assume $x\leq y$ and calculate
\begin{equation*}
\frac{f_\lambda(x+y)}{f_\lambda(x)f_\lambda(y)}=\exp[\lambda (g(x+y)-g(x)-g(y))]=\exp\Big[-\lambda\int_0^xg'(r)\,dr+\lambda \int_y^{y+x}g'(r)\,dr\Big].
\end{equation*}
The bounds \eqref{ini32.5} follow since $g':(0,\infty)\to (0,\infty)$ is a continuous and decreasing function.

{\bf{Step 2.}} We show now that if $y\geq 1$ and $x\in[0,y]$ then
\begin{equation}\label{ini41}
|f_\lambda(y+x)-f_\lambda(y)|\leq f_\lambda(x)f_\lambda(y)\lambda g(y)(x/y).
\end{equation}
Indeed, using the monotonicity of $g'$ we write
\begin{equation}\label{ini42}
|f_\lambda(y+x)-f_\lambda(y)|=\int_y^{x+y}f_\lambda(r)\lambda g'(r)\,dr\leq \lambda\int_0^{x}f_\lambda(y+r)g'(y)\,dr.
\end{equation}
Using now \eqref{ini32.5} we estimate the right-hand side of \eqref{ini42} by
\begin{equation*}
\begin{split}
\lambda g'(y)\int_0^{x}f_\lambda (y+r)\,dr\leq \lambda g'(y)xf_\lambda (y+x)\leq f_\lambda(x)f_\lambda(y)\lambda g'(y)x.
\end{split}
\end{equation*}
Since $yg'(y)\leq g(y)$, this completes the proof of \eqref{ini41}. 

{\bf{Step 3.}} Recall that $B(t,v)=\langle v\rangle^3f_{\lambda(t)}(\langle v\rangle)$. The bounds \eqref{ineq3} follow from \eqref{ini32.5} and the monotonicity of the functions $f_\lambda$. The bounds \eqref{ineq2} follow from \eqref{ini41} and monotonicity. This completes the proof of the lemma.
\end{proof}
Finally, we prove weighted elliptic estimates for the functions $\Psi\,\psi$.

\begin{lemma} \label{ineq6}
If $W(k,\xi):=e^{\lambda g(\langle k,\xi\rangle)}\langle k,\xi\rangle^dg(\langle k,\xi\rangle)^p$, $\lambda,p\in[0,2]$, $d\in[0,6]$, and $t\in[0,T]$, then
\begin{equation}\label{eq:cutoffPhiMain}
\big\|\langle k,\xi\rangle^2 W(k,\xi)\widetilde{\big(\Psi\,\psi\big)}(t,k,\xi)\big\|_{L^2_{k,\xi}}\lesssim\big\|W(k,\xi)\,\widetilde{\omega}(t,k,\xi)\big\|_{L^2_{k,\xi}},
\end{equation}
uniformly in $\lambda,p,d$, and $\rho$ (the parameter in the definition of the function $g$).
\end{lemma}

\begin{proof}
We only need to use the equation $\Delta\psi=\omega$, $\psi(x,0)=\psi(x,1)=0$, see \eqref{eu2}. For simplicity of notation, we also drop the parameter $t$.

{\bf{Step 1.}} As in subsection \ref{har10.6}, we take the partial Fourier transform along $\T$, thus
\begin{equation}\label{ini16}
\partial_y^2\psi^\ast(k,y)-k^2\psi^\ast(k,y)=\omega^\ast(k,y),\qquad y\in[0,1],\,k\in\Z.
\end{equation}
In particular, since $\omega^\ast(k,y)=0$ if $y\in[0,\vartheta/2]$ or $y\in[1-\vartheta/2,1]$, we have
\begin{equation}\label{ini16.1}
\begin{split}
\psi^\ast(k,y)=
\begin{cases}
b_k^0\sinh(ky)/k&\qquad\text{ if }k\neq 0\text{ and }y\in[0,\vartheta/2],\\
b_k^1\sinh(k(y-1))/k&\qquad\text{ if }k\neq 0\text{ and }y\in[1-\vartheta/2,1],\\
b_k^0y&\qquad\text{ if }k=0\text{ and }y\in[0,\vartheta/2],\\
b_k^1(y-1)&\qquad\text{ if }k=0\text{ and }y\in[1-\vartheta/2,1],
\end{cases}
\end{split}
\end{equation}
where $b_k^0:=(\partial_y\psi^\ast)(k,0)$ and $b_k^1:=(\partial_y\psi^\ast)(k,1)$.

We calculate the coefficients $b_k^0$ and $b_k^1$ using Green functions, as in subsection \ref{har10.6}. Indeed, it follows from \eqref{ini16} that 
\begin{equation}\label{ini15}
\psi^\ast(k,y)=-\int_0^1\omega^\ast(k,z)G_k(y,z)\,dz,
\end{equation}
where
\begin{equation}\label{ini15.1}
G_k(y,z)=\frac{1}{k\sinh k}
\begin{cases}
\sinh(k(1-z))\sinh(ky)&\qquad\text{ if }y\leq z,\\
\sinh(kz)\sinh(k(1-y))&\qquad\text{ if }y\geq z,
\end{cases}
\end{equation}
if $k\neq 0$, and
\begin{equation}\label{ini15.2}
G_0(y,z)=
\begin{cases}
(1-z)y&\qquad\text{ if }y\leq z,\\
z(1-y)&\qquad\text{ if }y\geq z.
\end{cases}
\end{equation}
In particular, for $\iota\in\{0,1\}$,
\begin{equation}\label{ini15.3}
b_k^\iota=-\int_0^1\omega^\ast(k,z)G_k^\iota(z)\,dz,
\end{equation}
where
\begin{equation}\label{ini15.4}
\begin{split}
&G_k^0(z)=\frac{\sinh(k(1-z))}{\sinh k}\,\,\text{ if }\,\,k\neq 0,\qquad G_0^0(t)=1-z,\\
&G_k^1(z)=\frac{-\sinh(k z)}{\sinh k}\,\,\text{ if }\,\,k\neq 0,\quad\qquad\,\,\,\, G_0^1(t)=-z.
\end{split}
\end{equation}

{\bf{Step 2.}} We return now to the proof of \eqref{eq:cutoffPhiMain}. We start from the equation
\begin{equation}\label{ini19}
\Delta(\Psi\psi)(x,y)=\Psi(y)\omega(x,y)+2\Psi'(y)\partial_y\psi(x,y)+\Psi''(y)\psi(x,y).
\end{equation}
As in Lemma \ref{dar20}, see \eqref{dar22}, using \eqref{ini15.3} and the support restriction on $\omega$, it follows that
\begin{equation}\label{ini18}
|b_k^0|^2+|b_k^1|^2\lesssim \int_{\mathbb{R}}|\widetilde{\omega}(k,\xi)|^2e^{-|k|\vartheta/3}e^{-|\xi|^{7/8}}\,d\xi,
\end{equation}
for any $k\in\mathbb{Z}$. Let $\phi(x,y):=2\Psi'(y)\partial_y\psi(x,y)+\Psi''(y)\psi(x,y)$ denote the last two terms in \eqref{ini19}. It follows from \eqref{ini16.1} and \eqref{ini18} that
\begin{equation*}
|\widetilde{\phi}(k,\xi)|\lesssim \|\widetilde{\omega}\|_{L^2_{k,\xi}}e^{-\vartheta|k|/20}e^{-|\xi|^{5/6}},
\end{equation*}
compare with \eqref{dar34}. Moreover, since the decay of $\widetilde{\Psi}$ is faster than the variation of the weights $W$ (compare with the definitions \eqref{ini5}), we have
\begin{equation*}
\big\|W(k,\xi)\widetilde{\big(\Psi\,\omega\big)}(k,\xi)\big\|_{L^2_{k,\xi}}\lesssim\big\|W(k,\xi)\,\widetilde{\omega}(k,\xi)\big\|_{L^2_{k,\xi}}.
\end{equation*}
The last two inequalities and the identity \eqref{ini19} show that
\begin{equation}\label{ini18.5}
\big\|(k^2+\xi^2) W(k,\xi)\widetilde{\big(\Psi\,\psi\big)}(k,\xi)\big\|_{L^2_{k,\xi}}\lesssim\big\|W(k,\xi)\,\widetilde{\omega}(k,\xi)\big\|_{L^2_{k,\xi}}.
\end{equation}

Finally, in the case $k=0$ we can use the formulas for $\psi^\ast(0,y)$ in \eqref{ini16.1} and the bounds \eqref{ini18}. It follows that $|\widetilde{\Psi\psi}(0,\xi)|\lesssim \|\widetilde{\omega}\|_{L^2_{k,\xi}}e^{-|\xi|^{5/6}}$. The desired conclusion \eqref{eq:cutoffPhiMain} follows.
\end{proof}

\end{document}